\documentclass{amsart}
\usepackage[top=25mm, bottom=25mm, left=25mm, right=25mm]{geometry}
\usepackage{graphicx}
\usepackage{amsmath}
\allowdisplaybreaks[4]
\usepackage{amssymb}
\usepackage{amsthm}
\usepackage{amsfonts}
\usepackage{bm}
\usepackage{mathtools}
\usepackage{empheq}
\usepackage{tikz}

\usepackage{enumitem}
\usepackage{booktabs}
\usepackage{dsfont}

\usepackage{longtable}
\usepackage{float}

\usepackage{mathrsfs}
\usepackage{stmaryrd}

\usepackage{url}

\theoremstyle{definition}

\newtheorem{definition}{Definition}[section]
\newtheorem{proposition}[definition]{Proposition}
\newtheorem{lemma}[definition]{Lemma}
\newtheorem{theorem}[definition]{Theorem}
\newtheorem{corollary}[definition]{Corollary}
\newtheorem{remark}[definition]{Remark}

\newtheorem{assumption}[definition]{Assumption}

\numberwithin{equation}{section}

\makeatletter
\def\subsection{\@startsection{subsection}{2}%
  \z@{-.8\linespacing\@plus.2\linespacing}%
  {.5\linespacing}%
  {\normalfont\normalsize\bfseries}}
\makeatother

\title{A Unified Framework for $p$-adic Congruences and $p$-adic interpolations of Sequences}
\author{Yuta Nishibuchi}
\date{}

\keywords{Congruence, Fibonacci numbers, Bernoulli numbers, $p$-adic analysis, $p$-adic integral}
\subjclass[2020]{Primary 11S80; Secondary 11B39, 11B50, 11B68, 13F25}
\address{Yuta Nishibuchi\\Mathematical Institute, Graduate School of Science, Tohoku University\\6-3, Aoba, Aramaki, Aoba-ku, Sendai 980-8578, Japan}
\email{nishibuchi.yuta.s2@dc.tohoku.ac.jp}
\thanks{ORCID: 0009-0007-1058-3819}

\begin{document}

\begin{abstract}
    In this paper, we treat the problem of determining, for a fixed prime $p$, a positive integer $\ell$ that is coprime to $p$, and a real number $r$, when a sequence $\{c_n\}_{n=0}^{\infty}$ with the exponential generating function $F(t)$ satisfies the congruence $c_{n+\ell p^k}\equiv c_n \pmod{p^{k+r}}$ for all $k\ge0$ and large $n$ depending on $k$. Though this type of congruence is often observed for several number-theoretic or combinatorial sequences, such as Bernoulli numbers, Euler numbers, and Fibonacci numbers, it appears that their congruences have been shown individually. 
    
    Traditionally, it is known that sequences derived as moments of $p$-adic measures satisfy stronger congruences. In this paper, instead of using $p$-adic measures, we investigate an algebraic structure of the set of generating functions and establish some conditions on $F(t)$ that are equivalent to or sufficient for the congruence and smoothness of the interpolating functions. This framework is an extension of the special case of $p$-adic moment sequences and provides a uniform manner to check the congruences $c_{n+\ell p^k}\equiv c_n \pmod{p^{k+r}}$ of concrete sequences. As an application, we also provide some equivalent conditions of the definition of Wall-Sun-Sun primes, also called Fibonacci-Wieferich primes.
\end{abstract}

\maketitle

\tableofcontents

\section{Introduction}
\noindent
The aim of this paper is to determine, for a given prime $p$, a positive integer $\ell$ such that $p\nmid\ell$, and a real number $r$, when a sequence $\{c_n\}_{n=0}^{\infty}\subset\mathbb{C}_p$ with the exponential generating function $F(t)$ satisfies the congruence 
    \begin{equation}
    \label{Basic_equ}
        c_{n+\ell p^k}\equiv c_n\mod{p^{k+r}}\quad\text{for all }k\ge0\text{ and large }n\ge n_k.
    \end{equation}
More precisely, when $\{c_n\}$ satisfies the following:
\[\text{ for any }k\ge0,\text{ there is a }n_k\ge0\text{ such that }v_p(c_{n+\ell p^k}-c_n)\ge{k+r}\text{ for all }n\ge n_k\]
where $v_p$ is the additive valuation on $\mathbb{C}_p$ normalized by $v_p(p)=1$.

Congruence \eqref{Basic_equ} often occurs in several number-theoretic or combinatorial sequences. One of the most famous ones is Kummer's congruence for Bernoulli numbers: for any $c\in\mathbb{Z}\setminus(\{1\}\cup p\mathbb{Z})$, 
\[(1-c^{-m})(1-p^{m-1})\frac{B_m}{m}\equiv(1-c^{-n})(1-p^{n-1})\frac{B_n}{n}\mod{p^{k+1}}\quad\text{if }m\equiv n\mod{(p-1)p^k}.\]
Fibonacci numbers satisfy
\[F_{n+\ell p^k}\equiv F_n\mod{p^{k+1}}\quad
\left(\ell=
\begin{cases}
    3 & (\text{if }p=2) \\
    p-1 & (\text{if }p\equiv1,4\mod{5}) \\
    2(p+1) & (\text{if }p\equiv2,3\mod{5},~p\ne2)
\end{cases}
\right)\quad\text{for all }k,n\ge0.\quad(\text{\cite{Jacobson, Lengyel}})
\]
Other examples satisfying \eqref{Basic_equ} for some $(\ell, r)$ include, among many others, Fubini numbers (\cite{Asgari}), Genocchi numbers (\cite{Cenkci}), Tangent numbers (\cite{Cenkci}), Euler numbers (\cite{Metsankyla}), and derangement numbers (\cite{Wang_and_Miska_and_Mezo}). Although this congruence appears in many famous or important sequences, it seems that they have been proven individually. This paper provides a uniform framework to prove congruence \eqref{Basic_equ}, and as byproducts, we also obtain
\begin{enumerate}
    \item a canonical construction of a sequence $\{c'_n\}$ satisfying
    \begin{equation}
    \label{Basic_equ_2}
        c'_{n+\ell p^k}\equiv c'_n\mod{p^{k+r}}\quad\text{for all }k,n\ge0.
    \end{equation}
    from a sequence $\{c_n\}$ satisfying \eqref{Basic_equ}. (This is an extension of ``multiplying Euler factor'' in the Kummer congruence of Bernoulli numbers.)
    \item conditions for the smoothness of $p$-adic interpolating functions
    \begin{equation}
    \label{Basic_equ_3}
        f_i:\mathbb{Z}_p\rightarrow\mathbb{C}_p;n\rightarrow c'_{i+\ell n}\quad(i=0,1,\cdots,\ell-1)
    \end{equation}
    for $\{c'_n\}$ satisfying \eqref{Basic_equ_2}. (Note that if \eqref{Basic_equ_2} holds, $n\mapsto c_{i+\ell n}$ extends continuously to $\mathbb{Z}_p$.)
\end{enumerate}

Traditionally, congruence \eqref{Basic_equ} of the case $(\ell,r)=(p-1,1)$ has been proven using $p$-adic integration. For example, in \cite{Young, Young2}, P. T. Young showed general results of congruences obtained by $p$-adic integration and its application to concrete sequences. In many literatures---including in Young's work---, to prove congruences of sequences $\{c_n\}_{n=0}^{\infty}$ by using $p$-adic integration, one typically uses a $\mathcal{O}_{\mathbb{C}_p}$-valued $p$-adic measure $\mu$ on $\mathbb{Z}_p$ such that
\begin{equation}
\label{Basic_measure_equ}
    \int_{\mathbb{Z}_p}x^nd\mu(x)=c_n.
\end{equation}
If such $\mu$ exists, $\{c_n\}$ satisfies \eqref{Basic_equ} with $(\ell,r)=(p-1,1)$. The Kummer congruence of Bernoulli numbers can also be proven by using measure $\mu_{c}$ such that
\[\int_{\mathbb{Z}_p}x^nd\mu_{c}(x)=(1-c^{-(n+1)})(1-p^n)\frac{B_{n+1}}{n+1}\]
with an arbitrary integer $c$ such that $c\ne1$ and $c\not\equiv 0\mod{p}$.
\footnote{For example, see \cite[pp.~36--45]{Koblitz}.}

However, one cannot always choose $\mu$ to satisfy \eqref{Basic_measure_equ}. Indeed, N. M. Katz provided a criterion for the existence of $\mu$, abstract Kummer congruence (\cite[(4.0.6)]{Katz}), which is stronger than \eqref{Basic_equ} with $(\ell, r)=(p-1,1)$.

In this paper, we present another perspective on $p$-adic congruences by investigating the algebraic structures of sets of generating functions instead of using $p$-adic measures. This perspective provides necessary and sufficient conditions for \eqref{Basic_equ}, \eqref{Basic_equ_2}, and the local analyticity of functions \eqref{Basic_equ_3} with general $\ell$ and $r$. In the remainder of this section, we give an overview of the main results of this paper.

\subsection{Notation and main concepts}
In this paper, $p$ is reserved for primes, and we denote $\mathbb{Z}_p$ as the ring of $p$-adic integers, $\mathbb{Q}_p$ as its fractional field. We fix an algebraic closure of $\mathbb{Q}_p$ and denote by $\mathbb{C}_p$ its completion, and $\mathcal{O}_{\mathbb{C}_p}$ is the valuation ring of $\mathbb{C}_p$. Let $v_p$ be the additive valuation on $\mathbb{C}_p$ normalized by $v_p(p)=1$, and $|\cdot|_p=p^{-v_p(\cdot)}$ be the normalized norm on $\mathbb{C}_p$. For any $a,b\in\mathbb{C}_p$ and $R\in\mathbb{R}$, we extend the congruence relation by
\[a\equiv b\mod{p^R}\Longleftrightarrow v_p(a-b)\ge R.\]

The letter $\ell$ is used for a positive integer coprime to $p$, and $h$ denotes the minimum positive integer such that $\ell\mid p^h-1$. Let $K_h$ be the unramified extension of $\mathbb{Q}_p$ of degree $h$ with the ring of integers $\mathcal{O}_{K_h}=\mathcal{O}_h$ and the maximal ideal $\mathfrak{m}_{h}$. Note that $K_h=\mathbb{Q}_p(\zeta_\ell)$ where $\zeta_\ell$ is a primitive $\ell$-th root of unity. We set
\begin{equation}
\label{q_def}
    q=p^h
\end{equation}
for notational convenience.
We write $\mathbb{Z}_{\ge0}$ for the set of nonnegative integers.

For each $s\in\mathbb{Z}_{\ge0}$, we say that \(f:\mathbb{Z}_p\rightarrow\mathbb{C}_p\) is locally analytic of order \(s\) if, for every \(a\in\mathbb Z_p\), \(f\) admits a power-series expansion in \(x-a\) that converges on \(a+p^s\mathbb Z_p\).

For each $r\in\mathbb{R}$ and a complete subring $R\subset\mathbb{C}_p$, let
\begin{equation}
\label{M_r(R)_def}
    M_{p,\ell,r}(R)=M_r(R)\coloneqq\Bigl\{\{c_n\}_{n=0}^{\infty}\subset R~ \Bigm|~\{c_n\}\text{ satisfies \eqref{Basic_equ}}\Bigr\},
\end{equation}
\begin{equation}
\label{R_r(R)_def}
    \mathcal{M}_{p,\ell,r}(R)=\mathcal{M}_r(R)\coloneqq\left\{\sum_{n=0}^{\infty}c_n\frac{t^n}{n!}~\middle|~\{c_n\}\subset R\text{ and }\{c_n\}\text{ satisfies \eqref{Basic_equ}}\right\},
\end{equation}
and for each $s\in\mathbb{Z}_{\ge0}$, let
\begin{equation}
\label{R_r,s^la(R)_def}
    \mathcal{M}_{p,\ell,r,s}^{la}(R)=\mathcal{M}_{r,s}^{la}(R)\coloneqq\left\{\sum_{n=0}^{\infty}c_n\frac{t^n}{n!}~\middle|~
    \begin{array}{c}
        \{c_n\}\subset R,~\{c_n\}\text{ satisfies \eqref{Basic_equ_2}, and }
        f_0,\cdots,f_{\ell-1}\\\text{ given by \eqref{Basic_equ_3} are locally analytic of order }s
    \end{array}
    \right\}.
\end{equation}

\subsection{An overview}
For clarity, we first present in Figure~\ref{overview_figure} a schematic overview of the relationships among the objects considered in this paper. We explain each part of the figure, including their definitions, in the subsequent subsections.
\begin{figure}[htbp]
\centering
\begin{tikzpicture}[
    >=stealth,
    every node/.style={font=\Large}
]

\node (M) at (0,2.5)
    {$\displaystyle
    M_r(R)
    \qquad = \qquad
    \bigoplus_{\alpha^{\ell}=1} M_r^{(\alpha)}(R)
    \qquad \oplus \qquad
    N(R)$};

\node (R) at (0,-0.5)
    {$\displaystyle
    \mathcal{M}_r(R)
    \qquad = \qquad
    \bigoplus_{\alpha^{\ell}=1} e^{\alpha t}\mathcal{A}_r(R)
    \qquad \oplus \qquad
    \mathcal{B}(R)$};

\node (L) at (-1.55,-2.5)
    {$\displaystyle
    \mathcal{M}_{r,s}^{la}(R)
    \qquad = \qquad
    \bigoplus_{\alpha^{\ell}=1} e^{\alpha t}\mathcal{A}_{r,s}(R)$};

\node (F) at (-1.65,-5.0)
    {$\displaystyle
    \left\{\sum_{n=0}^{\infty}c_n\frac{t^n}{n!}~\middle|~
    \int_{\mathbb{Z}_p}x^nd\mu(x)=c_n~(\exists\mu)\right\}
    ~\subset~
    \bigoplus_{\alpha^{p-1}=1} e^{\alpha t}\mathcal{A}'
    \quad \oplus \quad
    \mathcal{A}'$};

\draw[dotted]
    (-7.8,1.15) -- (4.5,1.15);

\node[font=\large] at (-6.0,1.60)
    {Sequences};

\node[font=\large] at (-6.0,0.70)
    {Generating functions};

\draw[<->]
    (-3.8,1.8) -- (-3.8,0.30);

\draw[<->]
    (-0.1,1.8) -- (-0.1,0.30);

\draw[<->]
    (3.8,1.8) -- (3.8,0.30);

\draw[<->]
    (0.9,1.45) -- (0.9,0.85);

\node at (0.9,1.8)
    {\small $M_r^{(\alpha)}(R)$};

\node at (0.9,0.5)
    {\small $e^{\alpha t}\mathcal{A}_r(R)$};

\draw[-]
    (-0.1,-1.2) -- (-0.1,-1.9);

\node at (1.55,-1.50)
    {\small$\displaystyle \mathcal{A}_{r,s}(R)\subset\mathcal{A}_r(R)$};

\draw[-]
    (-0.1,-3.35) -- (-0.1,-4.35);

\node[align=center] at (1.75,-3.85)
    {\small $\displaystyle \mathcal{A}'\subseteq \mathcal{A}_{r,s}(R)\cap \mathcal{B}(R)$\\ \small if $r\le1$, $p-1\mid\ell$,\\ \small and $p+s\ge3$};

\draw[-]
    (3.8,-1.0) -- (3.8,-4.35);

\end{tikzpicture}
\caption{Schematic overview, $R=\mathbb{C}_p$ or $\mathcal{O}_{\mathbb{C}_p}$.}
\label{overview_figure}
\end{figure}
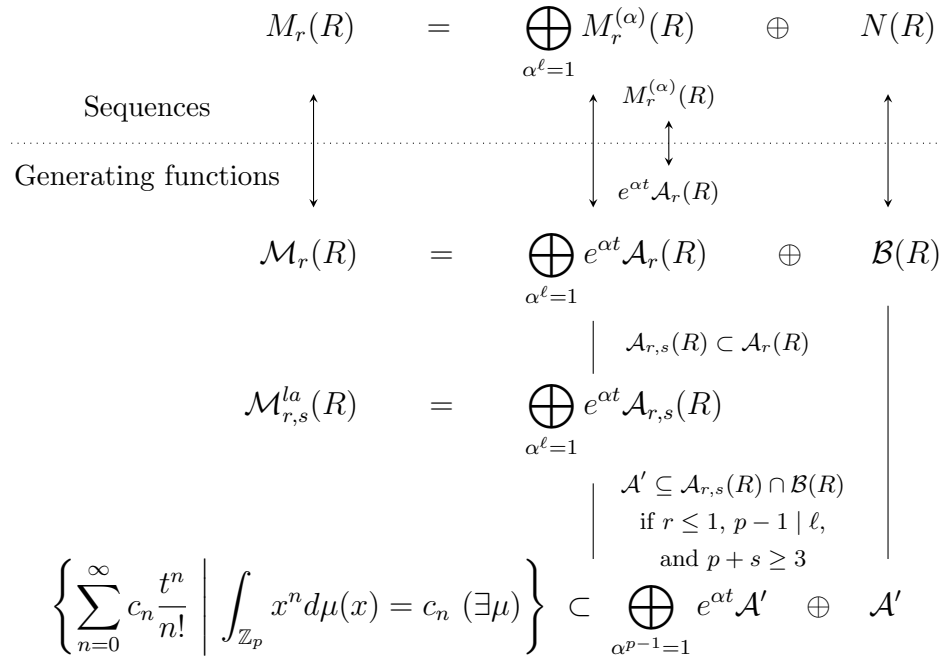

We only treat the case of $R=\mathbb{C}_p$ or $R=\mathcal{O}_{\mathbb{C}_p}$ in the remainder of this section.
This is because, in applications, there is little benefit to be gained from narrowing the set to which sequences belong, but much of the argument may remain valid for a general complete subring $R\subset\mathcal{O}_{\mathbb{C}_p}$, with minor modifications. We consider the case of $R=\mathcal{O}_{\mathbb{C}_p}$ because it yields more convenient results when $v_p(x)\ge0$ for all $x\in R$.

Certain symbols $\mathcal{A}_r(R),\mathcal{A}_{r,s}(R),\mathcal{A}',\text{ and }\mathcal{B}(R)$ in Figure 1 are temporary and will not be used from Section~2 onward.

\subsection{Structures of sets of sequences}
The set $M_r(R)$ naturally has the structure of an $R$-module, with
\begin{align*}
    \{c_n\}+\{d_n\}&=\{c_n+d_n\}\quad(\{c_n\},\{d_n\}\in M_r(R))\quad\text{and}\\
    x\{c_n\}&=\{xc_n\}\quad(x\in R, \{c_n\}\in M_r(R)).
\end{align*}

In section~\ref{Decomposition_of_sequences_subsection}, we show the direct sum decomposition of an $R$-module
\begin{equation}
\label{Decomposition_of_M_r(R)_equ}
    M_r(R)=\left(\bigoplus_{\alpha^{\ell}=1}M_r^{(\alpha)}(R)\right)\oplus N(R),
\end{equation}
where
\begin{equation}
\label{M_r^{(alpha)}(A)_def}
    M_r^{(\alpha)}(R)=\left\{\{c_n\}_{n=0}^{\infty}\subset R~ \middle|~
    c_{n+p^k}\equiv \alpha^{p^k}c_n\mod{p^{k+r}}\text{ for all }k,n\in\mathbb{Z}_{\ge0}
    \right\} 
\end{equation}
for each $\alpha\in R$, and 
\begin{equation}
\label{N(R)_def}
    N(R)=\left\{\{d_n\}_{n=0}^{\infty}\subset R~ \middle|~
    v_p(d_n)\rightarrow\infty~(n\rightarrow\infty)
    \right\}.
\end{equation}
In the proof of that, the following is also demonstrated:
\begin{enumerate}
    \item We have\[M_r'(R)\coloneqq\Bigl\{\{c'_n\}_{n=0}^{\infty}\subset R~\Bigm|~\{c'_n\}\text{ satisfies \eqref{Basic_equ_2}}\Bigr\}=\bigoplus_{\alpha^{\ell}=1}M_r^{(\alpha)}(R).\]
    \item The projection $M_r(R)\rightarrow M_r'(R)$ is given by \[\{c_n\}\mapsto\left\{\lim_{k\rightarrow\infty}c_{n+\ell p^k}\right\}.\]
    \item The minimum possible value for $n_k$ in \eqref{Basic_equ} corresponds to the rate at which $v_p(d_n)\rightarrow\infty$, where $\{d_n\}$ is the $N(R)$-component of $\{c_n\}\in M_r(R)$.
\end{enumerate}
Hence, we can consider that the canonical map $M_r(R)\rightarrow M'_r(R)$ gives a canonical construction of $\{c'_n\}$ satisfying \eqref{Basic_equ_2} from $\{c_n\}$ satisfying \eqref{Basic_equ}. In Proposition~\ref{can_projedtion_Euler_factor_prop}, it is revealed that this projection is an extension of ``multiplying Euler factor'' in a sense. More precisely, the canonical projection maps
\[M_1(\mathbb{Z}_p)\ni\left\{(1-c^{-{n+1}})\frac{B_{n+1}}{n+1}\right\}_{n=0}^{\infty}\mapsto\left\{(1-c^{-{n+1}})(1-p^n)\frac{B_{n+1}}{n+1}\right\}_{n=0}^{\infty}\in M'_1(\mathbb{Z}_p)\]
for Bernoulli numbers $\{B_n\}$ and any integer $c$ such that $c\ne1$, $c\not\equiv0\mod{p}$.

\subsection{Reformulation by generating functions}
Using the Mahler expansion of $p$-adic continuous functions, we reformulate the decomposition \eqref{Decomposition_of_M_r(R)_equ} in the class of generating functions.

In section~\ref{Reformulation_via_power_series_subsection}, we show one-to-one correspondences between
\begin{align*}
    M_r^{(\alpha)}(\mathbb{C}_p)&\longleftrightarrow e^{\alpha t}\mathbb{C}_p\langle t|\ge\lfloor\log_pn\rfloor+r\rangle,\\
    M_r^{(\alpha)}(\mathcal{O}_{\mathbb{C}_p})&\longleftrightarrow e^{\alpha t}\mathbb{C}_p\langle t|\ge\Theta_r\rangle,\\
    \{c_n\}&\longmapsto\sum_{n=0}^{\infty}c_n\frac{t^n}{n!}
\end{align*}
where $\lfloor\cdot\rfloor$ is the floor function, $\lfloor\log_p0\rfloor=-\infty$,
\begin{equation}
\label{C_p<t|ge_theta>_def}
    \mathbb{C}_p\langle t|\ge\Theta\rangle\coloneqq\left\{\sum_{n=0}^{\infty}b_n\frac{t^n}{n!}\in\mathbb{C}_p\llbracket t\rrbracket\text{ }\middle|\text{ }v_p(b_n)\ge\Theta(n)\text{ for all }n\ge0\right\}
\end{equation}
for each function $\Theta:\mathbb{Z}_{\ge0}\rightarrow\mathbb{R}\cup\{\pm\infty\}$, and 
\begin{equation}
\label{Theta_r_def}
    \Theta_r(n)=\begin{cases}
        0 & (\text{if }n=0), \\
        \lfloor\log_pn\rfloor+r & (\text{if }n\ge1)
    \end{cases}
\end{equation}
for each $r\in\mathbb{R}$. That is, $\mathcal{A}_r(R)$ in Figure~\ref{overview_figure} is $\mathbb{C}_p\langle t|\ge\lfloor\log_pn\rfloor+r\rangle$ and $\mathbb{C}_p\langle t|\ge\Theta_r\rangle$, respectively.

A main idea of the proof of this correspondence is the fact that the Mahler coefficients of $p$-adic continuous function $f$ provide information about the smoothness\footnote{Indeed, Mahler coefficients of $f$ provide information about, for example, local analyticity, Lipschitz continuity, and whether $f$ is a $C^n$ function (see Definition~29.1 in \cite{Schikhof} for the precise definition of $C^n$ functions).} of $f$. We rephrase congruences as Lipschitz continuity of the interpolating functions and then further rephrase them in terms of the growth rate of the valuation of the coefficients.

The difference between $\mathbb{C}_p\langle t|\ge\lfloor\log_pn\rfloor+r\rangle$ and $\mathbb{C}_p\langle t|\ge\Theta_r\rangle$ is whether the valuations of the constant term of their elements are restricted or not. This distinction affects their algebraic structures. For example, $p^{-2}, p^{\lceil r\rceil}t\in\mathbb{C}_p\langle t|\ge\lfloor\log_pn\rfloor+r\rangle$, but their product $p^{\lceil r\rceil-2}t$ is not contained, while it is revealed that $\mathbb{C}_p\langle t|\ge\Theta_r\rangle$ is a subring of $\mathbb{C}_p\llbracket t\rrbracket$ if $r\ge0$ in Section~\ref{Change_of_basis_and_ring_stru_subsection}. Another important point here is that $\mathbb{C}_p\langle t|\ge\Theta\rangle$ does not depend on $\alpha$. This conciseness allows us to focus on larger algebraic structures in Section~4.

To describe the counterparts of $N(R)$ ($\mathcal{B}(R)$ in Figure~\ref{overview_figure}) and the direct sum decomposition of $\mathcal{M}_{r,s}^{la}(R)$, we define the following: for each function $\Theta:\mathbb{Z}_{\ge0}\rightarrow\mathbb{R}\cup\{\pm\infty\}$, let
\begin{equation}
\label{C_p<t|theta>_def}
    \mathbb{C}_p\langle t|\Theta\rangle\coloneqq\left\{\sum_{n=0}^{\infty}b_n\frac{t^n}{n!}\in\mathbb{C}_p\llbracket t\rrbracket\text{ }\middle|\text{ }v_p(b_n)-\Theta(n)\rightarrow\infty~(n\rightarrow\infty)\right\}.
\end{equation}

Note that neither $\mathbb{C}_p\langle t|\Theta\rangle\subset\mathbb{C}_p\langle t|\ge\Theta\rangle$ nor $\mathbb{C}_p\langle t|\ge\Theta\rangle\subset\mathbb{C}_p\langle t|\Theta\rangle$ holds in general.

A simple observation shows that there are one-to-one correspondences between
\begin{align*}
    N(\mathbb{C}_p)&\longleftrightarrow\mathbb{C}_p\langle t|0\rangle,\\
    N(\mathcal{O}_{\mathbb{C}_p})&\longleftrightarrow\mathbb{C}_p\langle t|0\rangle\cap\mathbb{C}_p\langle t|\ge0\rangle,\\
    \{d_n\}&\longmapsto\sum_{n=0}^{\infty}d_n\frac{t^n}{n!}.
\end{align*}

It is not difficult to see that $\mathbb{C}_p\langle t|\Theta\rangle$ and $\mathbb{C}_p\langle t|\ge\Theta\rangle$ are $\mathcal{O}_{\mathbb{C}_p}$-modules. In conclusion, we establish
\begin{align}
    &\mathcal{M}_r(\mathbb{C}_p)=\left(\bigoplus_{\alpha^{\ell}=1}e^{\alpha t}\mathbb{C}_p\langle t|\ge\lfloor\log_pn\rfloor+r\rangle\right)\oplus\mathbb{C}_p\langle t|0\rangle,\label{R_r(C_p)_decomposition}\\
    &\mathcal{M}_r(\mathcal{O}_{\mathbb{C}_p})=\left(\bigoplus_{\alpha^{\ell}=1}e^{\alpha t}\mathbb{C}_p\langle t|\ge\Theta_r\rangle\right)\oplus\Bigl(\mathbb{C}_p\langle t|0\rangle\cap\mathbb{C}_p\langle t|\ge0\rangle\Bigr)\label{R_r(O_C_p)_decomposition},\\
    &\mathcal{M}'_r(\mathcal{O}_{\mathbb{C}_p})\coloneqq\left\{\sum_{n=0}^{\infty}c'_n\frac{t^n}{n!}~\middle|~\{c'_n\}\subset\mathcal{O}_{\mathbb{C}_p}\text{ and }\{c'_n\}\text{ satisfies \eqref{Basic_equ_2}} \right\}=\bigoplus_{\alpha^{\ell}=1}e^{\alpha t}\mathbb{C}_p\langle t|\ge\Theta_r\rangle\label{(1.2)_gen_funcs_decomposition}
\end{align}
as $\mathcal{O}_{\mathbb{C}_p}$-modules.

In Section~\ref{Reformulation_via_power_series_subsection}, we also reformulate the smoothness of $p$-adic interpolating functions
\[f_i:\mathbb{Z}_p\rightarrow\mathbb{C'}_p:n\mapsto c_{n+\ell n}\quad(i=0,\cdots,\ell-1)\]
for $\{c'_n\}\in M'_r(R)$.

Again, by using Mahler expansion, the smoothness of $\{f_i\}$ are rephrased as conditions for the valuations of the coefficients, and we have
\begin{align}
    &\mathcal{M}_{r,s}^{la}(\mathbb{C}_p)=\bigoplus_{\alpha^{\ell}=1}e^{\alpha t}\Bigl(\mathbb{C}_p\langle t|\ge\lfloor\log_pn\rfloor+r\rangle\cap\mathbb{C}_p\langle t|v_p(\lfloor n/p^s\rfloor!)\rangle\Bigr),\\
    &\mathcal{M}_{r,s}^{la}(\mathcal{O}_{\mathbb{C}_p})=\bigoplus_{\alpha^{\ell}=1}e^{\alpha t}\Bigl(\mathbb{C}_p\langle t|\ge\Theta_r\rangle\cap\mathbb{C}_p\langle t|v_p(\lfloor n/p^s\rfloor!)\rangle\Bigr)\label{R_r,s^la(O_Cp)_decomposition}.
\end{align}
as $\mathcal{O}_{\mathbb{C}_p}$-modules.
(i.e., $\mathcal{A}_{r,s}(R)=\mathcal{A}_r(R)\cap\mathbb{C}_p\langle t|v_p(\lfloor n/p^s\rfloor!)\rangle$ in Figure~\ref{overview_figure}.)

\subsection{A convenient form}
In Section~\ref{Change_of_basis_and_ring_stru_subsection}, we prove that some specific $\Theta$, including $\Theta_r$ with $r\in[0,1]$ (resp. $\Theta(n)=v_p(\lfloor n/p^s\rfloor!)$ with $p+s\ge3$), satisfy the following:
\begin{enumerate}
    \item[(a)] $e^{pct}\in\mathbb{C}_p\langle t|\ge\Theta\rangle$ (resp. $\in\mathbb{C}_p\langle t| v_p(\lfloor n/p^s\rfloor!)\rangle$) for any $c\in\mathcal{O}_{\mathbb{C}_p}$,
    \item[(b)] $\mathbb{C}_p\langle t|\ge\Theta\rangle$ (resp. $\mathbb{C}_p\langle t|v_p(\lfloor n/p^s\rfloor!)\rangle$) is a subring of $\mathbb{C}_p\llbracket t\rrbracket$.
\end{enumerate}
These conditions enable us to change the basis of $\mathcal{M}_r(\mathcal{O}_{\mathbb{C}_p})$ and $\mathcal{M}_{r,s}^{la}(\mathcal{O}_{\mathbb{C}_p})$ as modules. That is, under the conditions $r\in[0,1]$ and $s+p\ge3$, we have 
\begin{align}
    \mathcal{M}_{p,\ell,r}(\mathcal{O}_{\mathbb{C}_p})&=\left(\bigoplus_{a\in A'}e^{at}\mathbb{C}_p\langle t|\ge\Theta_r\rangle\right)\oplus\Big(\mathbb{C}_p\langle t|0\rangle\cap\mathbb{C}_p\langle t|\ge0\rangle\Bigl)\label{R_r(O_Cp)_change_of_basis},\\
    \mathcal{M}_{p,\ell,r,s}^{la}(\mathcal{O}_{\mathbb{C}_p})&=\bigoplus_{a\in A'}e^{at}\Bigl(\mathbb{C}_p\langle t|\ge\Theta_r\rangle\cap\mathbb{C}_p\langle t|v_p(\lfloor n/p^s\rfloor!)\rangle\Bigr)\label{R_r,s(O_Cp)_change_of_basis}
\end{align}
where $A'\subset\mathcal{O}_{\mathbb{C}_p}$ is an arbitrary subset such that for each $\alpha\in\{x\in\mathcal{O}_h\mid x^\ell=1\}$, there is a unique $a\in A'$ satisfying $v_p(a-\alpha)\ge1$. Then, the descriptions of the right-hand sides of \eqref{R_r(O_Cp)_change_of_basis} and \eqref{R_r,s(O_Cp)_change_of_basis} are flexible and particularly convenient for applications.

A second benefit derived from (a) and (b) is that we can consider larger algebraic structures. Consider the case $\ell=p^h-1=q-1$ and let 
\begin{equation}
    \mathcal{R}_{p,h,r}=\mathcal{R}_r\coloneqq\sum_{\alpha^q=\alpha}e^{\alpha t}\mathbb{C}_p\langle t|\ge\Theta_r\rangle\label{R_r_def(R_r_alpha_decomposition_equ)}.
\end{equation}
Then, $\mathcal{R}_r$ is a subring of $\mathbb{C}_p\llbracket t\rrbracket$, and we have
\begin{align}
    \mathcal{R}_r&=\left\{\sum_{n=0}^{\infty}c_n\frac{t^n}{n!}~\middle|~
    \begin{array}{c}
        \{c_n\}\subset\mathcal{O}_{\mathbb{C}_p}\text{ and }\\
        \{c_n\}\text{ satisfies \eqref{Basic_equ} with at worst }n_k\le p^k
    \end{array}
    \right\}\label{R_r_rephrase_eq}\\
    &=\bigoplus_{a\in A}e^{at}\mathbb{C}_p\langle t|\ge\Theta_r\rangle\quad\quad\quad\text{(as a }\mathbb{C}_p\langle t|\ge\Theta_r\rangle\text{-module})\label{R_r_decomposition_equ}
\end{align}
where $A\subset\mathcal{O}_{\mathbb{C}_p}$ is an arbitrary subset such that for each $\alpha\in\{x\in\mathcal{O}_h\mid x^q=x\}$, there is a unique $a\in A$ satisfying $v_p(a-\alpha)\ge1$. In section~4.2, we determine the unit group of $\mathcal{R}_r$ as a corollary of a general result, and it is given as 
\begin{equation}
\label{unit_of_R_r}
    \mathcal{R}_r^{\times}=\{F(t)\in\mathcal{R}_r~|~v_p(F(0))=0\}.
\end{equation}

\subsection{Sequences admitting a $p$-adic measure}
The set 
\[\left\{\sum_{n=0}^{\infty}c_n\frac{t^n}{n!}~\middle|~
    \int_{\mathbb{Z}_p}x^nd\mu(x)=c_n~(\exists\mu)\right\}\]
in figure~\ref{overview_figure} denote the set of generating functions of $\{c_n\}$ for which there exists an $\mathcal{O}_{\mathbb{C}_p}$-valued measure $\mu$ such that 
\[\int_{\mathbb{Z}_p}x^nd\mu(x)=c_n.\]

In Section~\ref{Interpretation_of_integral_subsection}, we locate this class of sequences in a submodule of $\mathcal{M}_{r,s}^{la}(\mathcal{O}_{\mathbb{C}_p})$ when $r\le1$, $p-1\mid\ell$, and $p+s\ge3$. More precisely, 
\begin{equation}
\label{Integration_class_decomposition}
    \left\{\sum_{n=0}^{\infty}c_n\frac{t^n}{n!}~\middle|~
    \int_{\mathbb{Z}_p}x^nd\mu(x)=c_n~(\exists\mu)\right\}\subset\bigoplus_{\alpha^{p}=\alpha}e^{\alpha t}\mathcal{A}'=\left(\bigoplus_{\alpha^{p-1}=1}e^{\alpha t}\mathcal{A}'\right)\oplus\mathcal{A}'
\end{equation}
where $\mathcal{A}'=\mathbb{C}_p\langle t|\ge n\rangle$, and we have $\mathcal{A}'\subset\mathcal{A}_s(\mathcal{O}_{\mathbb{C}_p})\cap\mathcal{B}(\mathcal{O}_{\mathbb{C}_p})$ if $r\le1$, $p-1\mid\ell$, and $p+s\ge3$. While the conditions $r\le1$ and $p-1\mid\ell$ correspond to the fact that the existence of $\mu$ leads to congruence \eqref{Basic_equ} with $\ell=p-1,~r=1$, the condition $p+s\ge3$ is technical.

Therefore, the traditional argument of $p$-adic integration could be viewed as dealing a special class of sequences satisfying \eqref{Basic_equ}.

\subsection{Organization of the paper}
This paper is organized as follows. Section~2 provides the preliminaries on $p$-adic analysis that are needed in this paper. In~Section 3, we establish \eqref{Decomposition_of_M_r(R)_equ}, \eqref{R_r(C_p)_decomposition}, and \eqref{R_r(O_C_p)_decomposition}. In Section~4, we analyze the algebraic structures of the space $\mathbb{C}_p\langle t|\ge\Theta\rangle$, $\mathbb{C}_p\langle t|\Theta\rangle$, and their sums, and determine the unit group of $\mathcal{R}_r$ as a corollary of a general result. 
In Section~5, we explain some strong properties of sequences admitting $p$-adic measures and locate them within our framework by establishing \eqref{Integration_class_decomposition}.
In Section~6, we apply the results obtained in Sections~3 and~4 to Genocchi numbers, generalized Euler numbers, Fibonacci numbers, and Lucas numbers, and we also provide a $p$-adic analytical viewpoint for Wall-Sun-Sun primes.

\subsection{Acknowledgment}
I am profoundly grateful to my advisor, Nobuo Tsuzuki, for many valuable discussions and insightful advice that greatly guided this research. I would also like to thank Asuka Shiga for reading an earlier version of the manuscript and for providing helpful comments on the exposition.

\section{Preliminaries}
\noindent
Let $\mathscr{C}(\mathbb{Z}_p,\mathbb{C}_p)$ be a $\mathbb{C}_p$-vector space of continuous functions from $\mathbb{Z}_p$ to $\mathbb{C}_p$. We consider $\mathscr{C}(\mathbb{Z}_p,\mathbb{C}_p)$ as a $\mathbb{C}_p$-Banach space with valuation
\[v_{\mathscr{C}}(f)=\inf_{x\in\mathbb{Z}_p}v_p(f(x))\quad(f\in\mathscr{C}(\mathbb{Z}_p,\mathbb{C}_p)).\]

\begin{definition}
    For each integer $n\ge0$, we define the binomial polynomials by
    \[\binom{x}{n}=\begin{cases}
        1 & (\text{if }n=0), \\
        x(x-1)\cdots(x-n+1)/n! & (\text{if }n\ge1).
    \end{cases}\]
\end{definition}

\begin{theorem}[Mahler]
\label{Mahler_expansion_thm}
    The family of functions $\{\binom{x}{n}\}_{n=0}^{\infty}$ is a Banach basis of $\{\mathscr{C}(\mathbb{Z}_p,\mathbb{C}_p),v_{\mathscr{C}}\}$. That is, for every $f\in\mathscr{C}(\mathbb{Z}_p,\mathbb{C}_p)$, there exists a unique $\{a_n(f)\}_{n=0}^{\infty}\subset\mathbb{C}_p$ such that

    (i) $\displaystyle \lim_{n\rightarrow\infty}v_p(a_n(f))=\infty$

    (ii) $\displaystyle f(x)=\sum_{n=0}^{\infty}a_n(f)\binom{x}{n}$

    (iii) $\displaystyle v_{\mathscr{C}}(f)=\inf_{n\ge0}v_p(a_n(f))$.

    Moreover, we have
    \[a_n(f)=(\Delta^nf)(0)=\sum_{m=0}^{n}\binom{n}{m}(-1)^{n-m}f(m)\]
    for each $n\ge0$ where $\Delta f(x)=f(x+1)-f(x)$ is a difference operator, and $\Delta^n$ is its $n$-th iterate.
    
    (We call $a_n(f)$ \textit{the Mahler coefficient of degree $n$ of $f$}. )
\end{theorem}
\begin{proof}
    See Theorem~51.1 in \cite{Schikhof}, or Theorem~1.3.2 in \cite{Colmez}.
\end{proof}

\begin{remark}
\label{converse_rmk}
(1) If the image of $f$ is contained in a subring $R\subset\mathcal{O}_{\mathbb{C}_p}$ with $1\in R$, then the Mahler coefficients of $f$ are also contained in $R$.

(2) By the uniqueness of Mahler coefficients, if a function $f:\mathbb{Z}_{\ge0}\rightarrow\mathbb{C}_p$ satisfies
\[v_p\left(\sum_{m=0}^{n}\binom{n}{m}(-1)^{n-m}f(m)\right)\rightarrow\infty\quad(n\rightarrow\infty),\]
then $f$ can be continuously extended to $\mathbb{Z}_p$.
\end{remark}

\begin{definition}
    Let $f\in\mathscr{C}(\mathbb{Z}_p,\mathbb{C}_p)$. For each $r\in\mathbb{R}$, we say that $f$ is $p^{-r}$-Lipschitz continuous if 
    \[|f(x)-f(y)|_p\le p^{-r}|x-y|_p\]
    for all $x,y\in\mathbb{Z}_p$, or equivalently,
    \[v_p(f(x)-f(y))\ge v_p(x-y)+r\]
    for all $x,y\in\mathbb{Z}_p$.
\end{definition}

\begin{theorem}
\label{smooth_coeffi_thm}
    For a continuous function $f(x)=\sum_{n=0}^{\infty}a_n(f)\binom{x}{n}\in\mathscr{C}(\mathbb{Z}_p,\mathbb{C}_p)$, we have the following.

    (1) For each $r\in\mathbb{R}$, $f$ is $p^{-r}$-Lipschitz continuous if and only if \[v_p(a_n)\ge\lfloor\log_pn\rfloor+r \text{ for all }n\ge0\]
    where by convention, $\log_p0=-\infty$.

    (2) For each $s\in\mathbb{Z}_{\ge0}$, $f$ is locally analytic of order $s$ if and only if \[\lim_{n\rightarrow\infty}\left(v_p(a_n)-v_p\left(\left\lfloor\frac{n}{p^s}\right\rfloor!\right)\right)=\infty.\]
\end{theorem}
\begin{proof}
    (1) See Theorem~4.3 in \cite{Anashin} or Theorem~53.4(i) in \cite{Schikhof}. (2) See Corollary~2(b) in \cite[p. 162]{Amice}; see also Theorem~1.7.8 in \cite{Colmez} for an accessible treatment in English.
\end{proof}

\begin{remark}
\label{a_0(f)_rmk}
    If $r\ge0$ and $f:\mathbb{Z}_p\rightarrow\mathbb{C}_p$ is $p^{-r}$-Lipschitz continuous, the following conditions are equivalent:
    \begin{enumerate}
        \item[(i)] $f(\mathbb{Z}_p)\subset\mathcal{O}_{\mathbb{C}_p}$
        \item[(ii)] $v_p(a_0(f))\ge0$
    \end{enumerate}
    Indeed, since $a_0(f)=f(0)$, the implication (i) $\Rightarrow$ (ii) is clear. (ii) $\Rightarrow$  (i) follows from the strong triangle inequality and $|f(x)-f(0)|_p\le p^{-r}|x|_p$ for all $x\in\mathbb{Z}_p$.
\end{remark}

We collect some classical basic facts on the $p$-adic valuations of $n!$ and binomial coefficients, which will be used mainly in Section~4. In this paper, we adopt $0!=1$ and $\binom{0}{0}=1$.
\begin{proposition}[Legendre's formula]
\label{val_of_n!_prop}
    Let $n$ be a nonnegative integer, and let
    \[n=a_xp^x+\cdots+a_1p+a_0\quad(x\ge0,~0\le a_0,\cdots,a_x\le p-1)\]
    be a base-$p$ expansion of $n$, without requiring $a_x\ne0$. Then, we have 
    \[v_p(n!)=\frac{n-\sum_{i=0}^{x}a_i}{p-1}.\]
    In particular, we have
    \[\frac{n}{p-1}-\lfloor\log_pn\rfloor-1\le v_p(n!)<\frac{n}{p-1}\]
    for all $n\ge1$.
\end{proposition}
\begin{proof}
    The statement holds for $n=0$ since $0!=1$. 
    
    Let $n\ge1$. Since for each $1\le s\le x$, $a_xp^{x-s}+\cdots+a_s$ numbers of $\{1,2,\cdots,n\}$ are divisible by $p^s$, we have
    \begin{align*}
        v_p(n!)&=\sum_{s=1}^{x}(a_xp^{x-s}+\cdots+a_s)\\
        &=\sum_{i=1}^{x}a_i\sum_{j=0}^{i-1}p^j\\
        &=\frac{\sum_{i=1}^{x}a_ip^i-\sum_{i=1}^{x}a_i}{p-1}\\
        &=\frac{n-\sum_{i=0}^{x}a_i}{p-1}.
    \end{align*}
    Let $w$ be the maximum integer such that $a_w\ne0$. Then, $w\le\lfloor\log_pn\rfloor$. Thus,
    \[\frac{n}{p-1}-\lfloor\log_pn\rfloor-1\le\frac{n-(w+1)(p-1)}{p-1}\le v_p(n!)\le\frac{n-a_w}{p-1}<\frac{n}{p-1}.\]
\end{proof}

\begin{proposition}[Kummer's theorem]
\label{val_of_binom_coeff_prop}
    Let $n\ge m\ge0$ are positive integers, and let
    \begin{align*}
        n&=a_xp^x+\cdots+a_1p+a_0\quad(x\ge0,~0\le a_0,\cdots,a_x\le p-1),\\
        m&=b_yp^y+\cdots+b_1p+b_0\quad(y\ge0,~0\le b_0,\cdots,b_y\le p-1),\\
        n-m&=c_zp^z+\cdots+c_1p+c_0\quad(z\ge0,~0\le c_0,\cdots,c_z\le p-1)
    \end{align*}
    be base-$p$ expansions of $n,m,n-m$, respectively, without requiring $a_x\ne0$, $b_y\ne0$, or $c_z\ne0$. Then, we have
    \[v_p\left(\binom{n}{m}\right)=\frac{1}{p-1}\sum_{i\ge0}(b_i+c_i-a_i).\]
\end{proposition}
\begin{proof}
    It follows from Proposition~\ref{val_of_n!_prop} and $v_p(\binom{n}{m})=v_p(n!)-v_p(m!)-v_p((n-m)!)$.
\end{proof}

\section{A framework for studying congruence properties}
\noindent
In this section, let $R$ be a complete subring of $\mathbb{C}_p$, and suppose $\mathcal{O}_{h}\subset R$.

We first decompose a space of sequences $M_r(R)$ in subsection~\ref{Decomposition_of_sequences_subsection}. We then reformulate each direct sum component in terms of a space of generating functions, in particular for $R=\mathcal{O}_{\mathbb{C}_p},\mathbb{C}_p$ in subsection~\ref{Reformulation_via_power_series_subsection}.

\subsection{Decomposition of sequences}
\label{Decomposition_of_sequences_subsection}
The aim of this subsection is to prove the following direct sum decomposition \eqref{Decomposition_of_M_r(R)_equ}.

It is almost clear by definition that 
\[M'_r(R)=\Bigl\{\{c'_n\}_{n=0}^{\infty}\subset R\Bigm|\{c'_n\} \text{ satisfies \eqref{Basic_equ_2}}\Bigr\}\]
and $N(R)$ are submodules of $M_r(R)$. If $\alpha$ is a $\ell$-th root of unity and $\{c_n\}\in M_r^{(\alpha)}(R)$, we have
\[c_{n+\ell p^k}\equiv(\alpha^k)^{\ell}c_n=c_n\mod{p^{k+r}}\quad\text{for all }k,n\in\mathbb{Z}_{\ge0}.\]
Thus, $M_r^{(\alpha)}(R)$ is a submodule of $M'_r(R)$.

In this subsection, we can establish the following. 

\begin{proposition}
\label{internal_direct_sum_prop}
    Let $r\in\mathbb{R}$.
    Let $R$ be a complete subring of $\mathbb{C}_p$ containing $\mathcal{O}_h$. (See \eqref{q_def} and just before for the definition of $h,q$.) Then, we have
    \begin{align}
        M'_r(R)=\bigoplus_{\alpha^{\ell}=1}M_r^{(\alpha)}(R), \label{M'_r(R)_Decomposition_equ}\\
        M_r(R)=M'_r(R)\oplus N(R)\label{M_r(R)=M'_r(R)+N(R)_equ}.
    \end{align}
    In particular, we have \eqref{Decomposition_of_M_r(R)_equ}.
\end{proposition}

Since the decomposition itself will play an important role in later discussions, we first establish it in the following two lemmas. After proving these lemmas, we complete the proof of Proposition~\ref{internal_direct_sum_prop} by showing the uniqueness of this decomposition, i.e., that the sum is direct.
\begin{lemma}
\label{decomposition_lemma}
    Let $r\in\mathbb{R}$.
    
    Let $\{c_n\}_{n=0}^{\infty}\subset R$ be a sequence satisfying 
    \[c_{n+\ell p^k}\equiv c_n\mod{p^{k+r}}\text{ for all }k,n\in\mathbb{Z}_{\ge0}.\]

    Then, there exists a set of sequences $\{\{c_n^{(\alpha)}\}_{n=0}^{\infty}\subset R\}_{\alpha^{\ell}=1}$ such that
    \begin{enumerate}[label=(\roman*)]
        \item $c_n=\sum_{\alpha^{\ell}=1}c_n^{(\alpha)}$ for all $n\ge0$
        \item $c_{n+p^k}\equiv \alpha^{p^k} c_n\mod{p^{k+r}}$ for all $k,n\in\mathbb{Z}_{\ge0}$.
    \end{enumerate}
\end{lemma}
\begin{proof}
    Let $\zeta$ be a primitive $\ell$-th root of unity. Note that $\zeta\in\mathcal{O}_h\subset R$

    Define
    \[
    A=
    \begin{bmatrix}
       1        & 1        & \cdots & 1            \\
       \zeta& \zeta^2 & \cdots & \zeta^{\ell} \\
       \vdots   & \vdots   &        & \vdots       \\
       \zeta^{\ell-1} & \zeta^{2(\ell-1)} & \cdots & \zeta^{\ell(\ell-1)} \\
    \end{bmatrix}.
    \]
    Then, we have 
    \[\det A=\prod_{1\le i<j\le\ell}(\zeta^j-\zeta^i).\]
    The reduction map $\mathcal{O}_{h}\rightarrow\mathbb{F}_q$ restricts to an injection $\{\zeta^i\}\rightarrow\mathbb{F}_{q}^{\times}$.
    Thus, we have $\det A\in\mathcal{O}_h^\times\subset R^\times$, and so $B=A^{-1}$ has $\mathcal{O}_h$ entries.

    For each $n\ge0$ and $k\ge 0$, let us define
    \[
    \begin{bmatrix}
        c_{n,k}^{(1)} \\
        c_{n,k}^{(2)} \\
        \vdots \\
        c_{n,k}^{\ell} \\
    \end{bmatrix}
    =B
    \begin{bmatrix}
        c_n \\
        c_{n+p^{kh}} \\
        \vdots \\
        c_{n+(\ell-1)p^{kh}} \\
    \end{bmatrix}.
    \]
    Then, by the congruence assumption of $\{c_n\}$, we obtain
    \[
    \begin{bmatrix}
        c_{n,k+1}^{(1)}-c_{n,k}^{(1)} \\
        c_{n,k+1}^{(2)}-c_{n,k}^{(2)} \\
        \vdots \\
        c_{n,k+1}^{\ell}-c_{n,k}^{\ell} \\
    \end{bmatrix}
    =B
    \begin{bmatrix}
        0 \\
        c_{n+p^{kh}+(p^h-1)p^{kh}}-c_{n+p^{kh}} \\
        \vdots \\
        c_{n+(\ell-1)p^{kh}+(\ell-1)(p^h-1)p^{kh}}-c_{n+(\ell-1)p^{kh}} \\
    \end{bmatrix}
    \equiv
    \begin{bmatrix}
        0\\
        \vdots\\
        0\\
    \end{bmatrix}
    \mod{p^{kh+r}}
    \]
    where we used $\ell\mid p^h-1$, and the congruence of vectors is understood componentwise.

    We define the sequences $\{c_n^{(\zeta^i)}\}$ by 
    \[c_n^{(\zeta^i)}=\lim_{k\rightarrow\infty}c_{n,k}^{(i)}\quad(\text{for each }n\ge0\text{ and }i=1,\cdots,\ell).\]
    Since $R$ is complete, $\{c_n^{(\zeta^1)}\}_{n=0}^{\infty},\cdots,\{c_n^{(\zeta^\ell)}\}_{n=0}^{\infty}$ are sequences in $R$.
    
    We state that this is the required family of sequences. Comparing the first components of the vectors
    \[
    \begin{bmatrix}
        c_n \\
        c_{n+p^{kh}} \\
        \vdots \\
        c_{n+(\ell-1)p^{kh}} \\
    \end{bmatrix}
    =A
    \begin{bmatrix}
        c_{n,k}^{(1)} \\
        c_{n,k}^{(2)} \\
        \vdots \\
        c_{n,k}^{(\ell)} \\
    \end{bmatrix}
    \quad(\text{for all }k,n\ge0),
    \]
    we have
    $c_n=c_{n,k}^{(1)}+\cdots+c_{n,k}^{(\ell)}$.
    Taking the limit $k\rightarrow\infty$, we obtain (i).

    Let 
    \[
    S=
    \begin{bmatrix}
       0      & 1 & 0      & \cdots & 0      \\
       \vdots & 0 & 1      & \ddots & \vdots \\
       \vdots &   & \ddots & \ddots & 0      \\
       0      &   &        & 0      & 1      \\
       1      & 0 & \cdots & \cdots & 0      \\
    \end{bmatrix}.
    \]
    Then, since $\zeta^{i(m-1)}=\zeta^{i(m+\ell-1)}$ for all $m,i\ge0$, multiplying by $S$ gives
    \[
    S^mA=
    \begin{bmatrix}
       \zeta^m & \zeta^{2m}  & \cdots & \zeta^{\ell m}            \\
       \zeta^{m+1}   & \zeta^{2(m+1)} & \cdots & \zeta^{\ell(m+1)} \\
       \vdots   & \vdots   &        & \vdots       \\
       \zeta^{m+\ell-1} & \zeta^{2(m+\ell-1)} & \cdots & \zeta^{\ell(m+\ell-1)} \\
    \end{bmatrix}\quad\text{for all }m\ge0.
    \]

    Thus, for all $k\ge0$, we have
    \begin{align*}
    A\begin{bmatrix}
        \zeta^{p^k}c_{n,k}^{(1)} \\
        \zeta^{2p^k}c_{n,k}^{(2)} \\
        \vdots \\
        \zeta^{\ell p^k}c_{n,k}^{(\ell)} \\
    \end{bmatrix}
    &=
    \begin{bmatrix}
       \zeta^{p^k} & \zeta^{2p^k}  & \cdots & \zeta^{\ell p^k}            \\
       \zeta^{p^k+1}   & \zeta^{2(p^k+1)} & \cdots & \zeta^{\ell(p^k+1)} \\
       \vdots   & \vdots   &        & \vdots       \\
       \zeta^{p^k+\ell-1} & \zeta^{2(p^k+\ell-1)} & \cdots & \zeta^{\ell(p^k+\ell-1)} \\
    \end{bmatrix}
    \begin{bmatrix}
        c_{n,k}^{(1)} \\
        c_{n,k}^{(2)} \\
        \vdots \\
        c_{n,k}^{(\ell)} \\
    \end{bmatrix}\\
    &=S^{p^k}A
    \begin{bmatrix}
        c_{n,k}^{(1)} \\
        c_{n,k}^{(2)} \\
        \vdots \\
        c_{n,k}^{(\ell)} \\
    \end{bmatrix}
    =S^{p^k}
    \begin{bmatrix}
        c_{n} \\
        c_{n+p^{kh}} \\
        c_{n+2p^{kh}} \\
        \vdots \\
        c_{n+(\ell-1)p^{kh}} \\
    \end{bmatrix}.
    \end{align*}
    Now, we have
    \begin{equation*}
        S
    \begin{bmatrix}
        c_{n+mp^{kh}} \\
        c_{n+(m+1)p^{kh}} \\
        \vdots \\
        c_{n+(m+\ell-1)p^{kh}} \\
    \end{bmatrix}
    =\begin{bmatrix}
        c_{n+(m+1)p^{kh}} \\
        \vdots \\
        c_{n+(m+\ell-1)p^{kh}} \\
        c_{n+mp^{kh}} \\
    \end{bmatrix}
    \equiv\begin{bmatrix}
        c_{n+(m+1)p^{kh}} \\
        \vdots \\
        c_{n+(m+\ell-1)p^{kh}} \\
        c_{n+(m+\ell)p^{kh}} \\
    \end{bmatrix}\mod{p^{k+r}}
    \end{equation*}
    for each $m\ge0$, and we obtain
    \[S^m\begin{bmatrix}
        c_{n} \\
        c_{n+p^{kh}} \\
        \vdots \\
        c_{n+(\ell-1)p^{kh}} \\
    \end{bmatrix}
    \equiv\begin{bmatrix}
        c_{n+mp^{kh}} \\
        c_{n+(m+1)p^{kh}} \\
        \vdots \\
        c_{n+(m+\ell-1)p^{kh}} \\
    \end{bmatrix}\mod{p^{k+r}}\]
    by induction. Thus, 
    \begin{equation*}
    A\begin{bmatrix}
        \zeta^{p^k}c_{n,k}^{(1)} \\
        \zeta^{2p^k}c_{n,k}^{(2)} \\
        \vdots \\
        \zeta^{\ell p^k}c_{n,k}^{(\ell)} \\
    \end{bmatrix}
    \equiv
    \begin{bmatrix}
        c_{n+p^kp^{kh}} \\
        c_{n+(p^k+1)p^{kh}} \\
        c_{n+(p^k+2)p^{kh}} \\
        \vdots \\
        c_{n+(p^k+\ell-1)p^{kh}} \\
    \end{bmatrix}
    \equiv
    \begin{bmatrix}
        c_{n+p^k} \\
        c_{n+p^k+p^{kh}} \\
        c_{n+p^k+2p^{kh}} \\
        \vdots \\
        c_{n+p^k+(\ell-1)p^{kh}} \\
    \end{bmatrix}
    \mod{p^{k+r}}.
    \end{equation*}
    For the last congruences, we used the fact that $p^kp^{kh}\equiv p^{k}\mod{\ell p^k}$.

    Therefore, for all $k\ge0$, we obtain
    \[
    \begin{bmatrix}
        \zeta^{p^k}c_{n,k}^{(1)} \\
        \zeta^{2p^k}c_{n,k}^{(2)} \\
        \vdots \\
        \zeta^{\ell p^k}c_{n,k}^{(\ell)} \\
    \end{bmatrix}
    \equiv B
    \begin{bmatrix}
        c_{n+p^k} \\
        c_{n+p^k+p^{kh}} \\
        c_{n+p^k+2p^{kh}} \\
        \vdots \\
        c_{n+p^k+(\ell-1)p^{kh}} \\
    \end{bmatrix}
    =\begin{bmatrix}
        c_{n+p^k,k}^{(1)} \\
        c_{n+p^k,k}^{(2)} \\
        \vdots \\
        c_{n+p^k,k}^{(\ell)} \\
    \end{bmatrix}
    \mod{p^{k+r}}.
    \]
    Since $c_n^{(\zeta^i)}\equiv c_{n,k}^{(i)}\mod{p^{kh+r}}$ by definition, (ii) also holds.
\end{proof}

\begin{lemma}
\label{c'_n_and_d_n_lemma}
    Let $r\in\mathbb{R}$ and $R\subset\mathbb{C}_p$ be a complete subring.
    
    Let $\{c_n\}_{n=0}^{\infty}\subset R$ be a sequence satisfying the condition: for all $k\in\mathbb{Z}_{\ge0}$, there exists $n_k\in\mathbb{Z}_{\ge0}$ such that
    \[c_{n+\ell p^k}\equiv c_n\mod{p^{k+r}}\quad\text{for all }n\ge n_k.\]

    Then, there exists a sequence $\{d_n\}_{n=0}^{\infty}\subset R$ such that 

    (i) For any $k\in\mathbb{Z}_{\ge0}$ and $n\ge n_k$, $d_n\equiv 0\mod{p^{k+r}}$. In particular, $v_p(d_n)\rightarrow\infty~(n\rightarrow\infty)$.

    (ii) $c_{n+\ell p^k}-d_{n+\ell p^k}\equiv c_n-d_n\mod{p^{k+r}}$ for all $k,n\in\mathbb{Z}_{\ge0}$.
\end{lemma}
\begin{proof}
    For each $n\ge0$, consider the sequence $\{c_{n+\ell p^{k}}\}_{k=0}^{\infty}$. Then, for any $k\in\mathbb{Z}_{\ge0}$, we have
    \begin{equation}
    \label{c_n_Cauchy_cong}
    c_{n+\ell p^{L+1}}=c_{n+\ell p^L+\ell(p-1)p^L}\equiv c_{n+\ell p^L}\mod{p^{k+r}}
    \end{equation}
    whenever $L$ is sufficiently large, say $L\ge \min\{k,\log_p(n_k/\ell)\}$. 
    
    Therefore, $\{c_{n+\ell p^k}\}_{k=0}^{\infty}$ is a Cauchy sequence in $R$.

    Define 
    \[c'_n=\lim_{k\rightarrow\infty}c_{n+\ell p^k}\quad\quad\text{and}\quad\quad d_n=c_n-c'_n.\]
    
    Then, for each $k\in\mathbb{Z}_{\ge0}$, using $L\in\mathbb{Z}_{\ge0}$ equal to or greater than $\min\{k,\log_p(n_k/\ell)\}$, we have
    \begin{align*}
        c_n-d_n&=c'_n\\
        &\equiv c_{n+\ell p^L}\\
        &\equiv c_{n+\ell p^k+\ell p^L}\\
        &\equiv c'_{n+\ell p^k}\\
        &=c_{n+\ell p^k}-d_{n+\ell p^k}\mod{p^{k+r}}.
    \end{align*}

    Moreover, if $n\ge n_k$, the congruence \eqref{c_n_Cauchy_cong} holds for all $L\ge k$, and we have
    \[c'_n\equiv c_{n+\ell p^k}\equiv c_n\mod{p^{k+r}}\]
    by the definition of $c'_n$. So $\{d_n\}$ satisfies (i) too.
\end{proof}

\begin{proof}[Proof of Proposition~\ref{internal_direct_sum_prop}]
    By Lemma~\ref{decomposition_lemma} and Lemma~\ref{c'_n_and_d_n_lemma}, we have
    \[M'_r(R)=\sum_{\alpha^{\ell}=1}M_r^{(\alpha)}(R),\quad M_r(R)=M'_r(R)+N(R).\]

    It remains to show the uniqueness of the decomposition.
    Let $\{\alpha_1,\cdots,\alpha_{\ell}\}=\{x\in\mathcal{O}_h|x^{\ell}=1\}$.
    Let $\{d_n\}\in N(R)$, $\{c_n^{(\alpha_i)}\}\in M_r^{(\alpha_i)}(R)$ for each $i=1,\cdots,\ell$, and assume
    \[\{c_n^{(\alpha_1)}\}+\cdots+\{c_n^{(\alpha_{\ell})}\}+\{d_n\}=\{0\}_{n=0}^{\infty}\quad(\text{as a sequence}).\]

    Since $v_p(d_n)\rightarrow\infty~(n\rightarrow\infty)$, for each $k\in\mathbb{Z}_{\ge0}$, we have 
    \begin{equation}
    \label{c_n_1+cdots+c_n_q-1equiv0_equ}
        c_n^{(\alpha_1)}+\cdots+c_n^{(\alpha_{\ell})}\equiv 0\mod{p^{k+r}}
    \end{equation}
    for sufficiently large $n$.
    Since $\{c_n^{(\alpha_1)}\}+\cdots+\{c_n^{(\alpha_{\ell})}\}\in M'_r(R)$ is periodic modulo $p^{k+r}$, the congruence \eqref{c_n_1+cdots+c_n_q-1equiv0_equ} holds for all $n\in\mathbb{Z}_{\ge0}$. Since $k$ was arbitrary, we obtain
    \[\{c_n^{(\alpha_1)}\}+\cdots+\{c_n^{(\alpha_{\ell})}\}=\{0\},\quad\{d_n\}=\{0\}.\]

    Let $S(m)$ be a shift operator defined by $S(m)(\{c_n\}_{n=0}^{\infty})=\{c_{n+m}\}_{n=0}^{\infty}$.
    Then, applying $S(p^{mh})-\alpha_1$, we have
    \[(\alpha_2-\alpha_1)\{c_n^{(\alpha_2)}\}+\cdots+(\alpha_{\ell}-\alpha_1)\{c_n^{(\alpha_{\ell})}\}\equiv\{0\}\mod{p^{mh+r}}\]
    for all $m\ge0$, where we used the fact that $\alpha_i^{p^{mh}}=\alpha_i$ because $\ell\mid p^h-1$. Thus, we obtain
    \[(\alpha_2-\alpha_1)\{c_n^{(\alpha_2)}\}+\cdots+(\alpha_{\ell}-\alpha_1)\{c_n^{(\alpha_{\ell})}\}=\{0\}.\]
    Repeating this for each $S(p^{mh})-\alpha_i$, and we finally have
    \[(\alpha_{\ell}-\alpha_1)\cdots(\alpha_{\ell}-\alpha_{\ell-1})\{c_n^{(\alpha_{\ell})}\}=\{0\}.\]
    Since $\alpha_i$'s are distinct, we have $\{c_n^{(\alpha_{\ell})}\}=\{0\}$.
    We have $\{c_n^{(\alpha_1)}\}=\cdots=\{c_n^{(\alpha_{\ell-1})}\}=\{0\}$ by induction.
\end{proof}

\subsection{Reformulation via power series}
\label{Reformulation_via_power_series_subsection}
As we mentioned in Section~1, there is no reason to restrict $R$ for applications. However, since a technical distinction arises depending on whether or not $R$ contains an $x$ satisfying $v_p(x)<0$, we consider below the cases $R=\mathbb C_p$ and $R=\mathcal O_{\mathbb C_p}$.

We obtained 
\[M_r(R)=\left(\bigoplus_{\alpha^{\ell}=1}M_r^{(\alpha)}(R)\right)\oplus N(R)\]
in the previous subsection. It is easy to see that 
\begin{align*}
    \{c_n\}_{n=0}^{\infty}\in N(\mathbb{C}_p)&\Leftrightarrow \sum_{n=0}^{\infty}c_n\frac{t^n}{n!}\in\mathbb{C}_p\langle t|0\rangle,\\
    \{c_n\}_{n=0}^{\infty}\in N(\mathcal{O}_{\mathbb{C}_p})&\Leftrightarrow \sum_{n=0}^{\infty}c_n\frac{t^n}{n!}\in\mathbb{C}_p\langle t|0\rangle\cap\mathbb{C}_p\langle t|\ge0\rangle.
\end{align*}

In this subsection, we establish 
\begin{align}
    \{c_n\}_{n=0}^{\infty}\in M_r^{(\alpha)}(\mathbb{C}_p)&\Leftrightarrow \sum_{n=0}^{\infty}c_n\frac{t^n}{n!}\in e^{\alpha t}\mathbb{C}_p\langle t|\ge\lfloor\log_pn\rfloor+r\rangle,\label{M_r^alpha(Cp)_gen_func_equ}\\
    \{c_n\}_{n=0}^{\infty}\in M_r^{(\alpha)}(\mathcal{O}_{\mathbb{C}_p})&\Leftrightarrow \sum_{n=0}^{\infty}c_n\frac{t^n}{n!}\in e^{\alpha t}\mathbb{C}_p\langle t|\ge\Theta_r\rangle\label{M_r^alpha(O_Cp)_gen_func_equ}
\end{align}
and lead \eqref{R_r(C_p)_decomposition} and \eqref{R_r(O_C_p)_decomposition} as consequences.

\begin{proposition}
\label{Basic_prop}
    Let 
    \[F(t)=\sum_{n=0}^{\infty}c_n\frac{t^n}{n!}\]
    be the generating function of a sequence $\{c_n\}_{n=0}^{\infty}\subset\mathbb{C}_p$.

    Let $\alpha\in\mathbb{C}_p$ such that $v_p(\alpha)=0$.

    (1) For any $r\in\mathbb{R}$, the following conditions are equivalent.

    \begin{enumerate}[label=(\alph*)]
      \item For any integer $r\ge0$ and $n\ge0$, $c_{n+p^{k}}\equiv\alpha^{p^k} c_n\pmod{p^{k+r}}$. 
      
      \item The function $f:\mathbb{Z}_{\ge0}\rightarrow\mathbb{C}_p;n\mapsto\alpha^{-n}c_n$ can be extended to a $p^{-r}$-Lipschitz continuous function on $\mathbb{Z}_p$.
      
      \item $e^{-\alpha t}F(t)\in\mathbb{C}_p\langle t|\ge\lfloor\log_pn\rfloor+r\rangle$.
    \end{enumerate}

    (2) For any $s\in\mathbb{Z}_{\ge0}$, the function $f:\mathbb{Z}_{\ge0}\rightarrow\mathbb{C}_p;n\mapsto\alpha^{-n}c_n$ can be extended to a function on $\mathbb{Z}_p$, which is locally analytic of order $s$, if and only if $e^{-\alpha t}F(t)\in\mathbb{C}_p\langle t|v_p(\lfloor n/p^s\rfloor!)\rangle$.
\end{proposition}
\begin{proof}
    Let
    \[e^{-\alpha t}F(t)=\sum_{n=0}^{\infty}b_n\frac{t^n}{n!}.\]
    Then, since
    \begin{align*}
        e^{-\alpha t}F(t)&=\left(\sum_{n=0}^\infty(-\alpha)^n\frac{t^n}{n!}\right)\left(\sum_{n=0}^\infty c_n\frac{t^n}{n!}\right)\\
        &=\sum_{n=0}^{\infty}\left(\sum_{m=0}^{n}\binom{n}{m}(-\alpha)^{n-m}c_m\right)\frac{t^n}{n!},
    \end{align*}
    we have
    \[\alpha^{-n}b_n=\sum_{m=0}^{n}\binom{n}{m}(-1)^{n-m}\frac{c_m}{\alpha^m}.\]
    Therefore, the sequence $\{\alpha^{-n}b_n\}$ is the sequence of Mahler coefficients of the function $f:n\mapsto\alpha^{-n}c_n$.
    Because $v_p(\alpha^{-n}b_n)=v_p(b_n)$, (2) and the equivalence of (b) and (c) of (1) follow from Theorem~\ref{smooth_coeffi_thm} and Remark~\ref{converse_rmk}.

    We show the equivalence of (a) and (b) of (1). Since $v_p(\alpha)=0$, condition (a) is equivalent to the congruence
    \[\frac{c_{n+p^k}}{\alpha^{n+p^k}}\equiv\frac{c_n}{\alpha^n}\mod{p^{k+r}}\quad\text{for all }r,n\ge0.\]
    Because $\mathbb{Z}_{\ge0}$ is dense in $\mathbb{Z}_p$, this is equivalent to condition (b).
\end{proof}

Proposition~\ref{Basic_prop}(1) is precisely the statement \eqref{M_r^alpha(Cp)_gen_func_equ}. 

Assume $\{c_n\}\in M_r^{(\alpha)}(\mathbb{C}_p)$ with $\alpha\in\mathcal{O}_{\mathbb{C}_p}^{\times}$. Then, the Mahler coefficient of degree $0$ of $f$ in condition (b) of Proposition~\ref{Basic_prop}(1) is $f(0)=c_0$. Therefore, by Remark \ref{a_0(f)_rmk}, we have
\[\{c_n\}\subset\mathcal{O}_{\mathbb{C}_p}\Leftrightarrow\{\alpha^{-n}c_n\}\subset\mathcal{O}_{\mathbb{C}_p}\Leftrightarrow f(\mathbb{Z}_p)\subset\mathcal{O}_{\mathbb{C}_p}\Leftrightarrow v_p(c_0)\ge 0.\]
Note that 
\[\Bigl\{F(t)\in\mathbb{C}_p\langle t|\ge\lfloor\log_pn\rfloor+r\rangle~\Bigm|~v_p(F(0))\ge0~\Bigr\}=\mathbb{C}_p\langle t|\ge\Theta_r\rangle\]
by the definition of $\Theta_r$. Thus, we obtain \eqref{M_r^alpha(O_Cp)_gen_func_equ} too. 

Combining them with \eqref{M_r(R)=M'_r(R)+N(R)_equ} and \eqref{M'_r(R)_Decomposition_equ}, we have \eqref{R_r(C_p)_decomposition}, \eqref{R_r(O_C_p)_decomposition}, and \eqref{(1.2)_gen_funcs_decomposition}; that is, 
\begin{align*}
    &\mathcal{M}_r(\mathbb{C}_p)=\left(\bigoplus_{\alpha^{\ell}=1}e^{\alpha t}\mathbb{C}_p\langle t|\ge\lfloor\log_pn\rfloor+r\rangle\right)\oplus\mathbb{C}_p\langle t|0\rangle, \\
    &\mathcal{M}_r(\mathcal{O}_{\mathbb{C}_p})=\left(\bigoplus_{\alpha^{\ell}=1}e^{\alpha t}\mathbb{C}_p\langle t|\ge\Theta_r\rangle\right)\oplus\Bigl(\mathbb{C}_p\langle t|0\rangle\cap\mathbb{C}_p\langle t|\ge0\rangle\Bigr)\\
    &\mathcal{M}'_r(\mathcal{O}_{\mathbb{C}_p})=\bigoplus_{\alpha^{\ell}=1}e^{\alpha t}\mathbb{C}_p\langle t|\ge\Theta_r\rangle
\end{align*}
as $\mathcal{O}_{\mathbb{C}_p}$-modules. (Note that $\mathbb{C}_p\langle t|\Theta\rangle$ and $\mathbb{C}_p\langle t|\ge\Theta\rangle$ are $\mathcal{O}_{\mathbb{C}_p}$-modules. Indeed, they contain $0$ by $v_p(0)=\infty$, and they are closed under addition because $v_p$ is a non-Archimedean valuation, and they are also closed under scalar multiplication because $v_p(x)\ge0$ for all $x\in\mathcal{O}_{\mathbb{C}_p}$.)

Another consequence of Proposition~\ref{Basic_prop} concerns the smoothness of $p$-adic interpolating functions for the sequences $\{c'_n\}\in M'_r(R)~(R=\mathbb{C}_p, \mathcal{O}_{\mathbb{C}_p})$.

Let $\{c'_n\}\in M'_r(R)$. Then, for each $i\in\{0,1,\cdots,\ell-1\}$, the sequence $\{c'_{i,n}\}=\{c'_{i+\ell n}\}$ satisfies
\[c'_{i,n+p^k}\equiv c'_{i,n}\mod{p^{k+r}}\]
for all $k,n\in\mathbb{Z}_{\ge0}$. Thus, we can consider $p$-adic interpolation
\begin{equation}
\label{f_i_def}
    f_i:\mathbb{Z}_p\rightarrow\mathbb{C}_p;n\mapsto c'_{i,n}=c'_{i+\ell n}.
\end{equation}

By using Proposition~\ref{Basic_prop}, we have the following proposition.
\begin{proposition}
\label{smoothness_of_interpolation_prop}
    Let $r\in\mathbb{R}, s\in\mathbb{Z}_{\ge0}$ and $\{c'_n\}\in M'_r(\mathbb{C}_p)$, write
    \[\{c'_n\}=\sum_{\alpha^{\ell}=1}\{c_n^{(\alpha)}\}\quad\text{with }\{c_n^{(\alpha)}\}\in M_r^{(\alpha)}(\mathbb{C}_p).\]

    Set $f_\alpha:\mathbb{Z}_p\rightarrow\mathbb{C}_p$ to be a $p$-adic interpolation of 
    \[\mathbb{Z}_{\ge0}\rightarrow\mathbb{C}_p;n\mapsto \alpha^{-n}c_n^{(\alpha)}.\]

    Then, all of $\{f_i\}_{i=0,1,\cdots,\ell-1}$ (defined by \eqref{f_i_def}) are locally analytic of order $s$ if and only if all of $\{f_{\alpha}\}$ are locally analytic of order $s$.
\end{proposition}
\begin{proof}
Fix $i\in\{0,1,\cdots,\ell-1\}$ and consider $p$-adic interpolations 
\[f_{i,\alpha}:\mathbb{Z}_p\rightarrow\mathbb{C}_p; n\mapsto c_{i+\ell n}^{(\alpha)}.\]
Then, since $\alpha^{\ell}=1$, and $\mathbb{Z}_{\ge0}$ is dense in $\mathbb{Z}_p$, we have
\[f_{i,\alpha}(x)=\alpha^if_{\alpha}(i+\ell x)\quad(x\in\mathbb{Z}_p).\]
Therefore,
\[f_i(x)=\sum_{\alpha}\alpha^if_\alpha(i+\ell x).\]
Since $\ell\in\mathbb{Z}_p^\times$, this means that all of $\{f_i\}_i$ are locally analytic of order $s$ if all of $\{f_\alpha\}_\alpha$ are locally analytic of order $s$.

Conversely, we have
\[f_i\left(\frac{x-i}{\ell}\right)=\sum_{\alpha}\alpha^if_\alpha(x).\] 
Since the matrix $A=(\alpha^i)_{\alpha,i}$ is invertible (recall the argument in the proof of Lemma~\ref{decomposition_lemma}), we can write as 
\[f_\alpha(x)=\sum_{i}b_{\alpha,i}f_i\left(\frac{x-i}{\ell}\right)\]
with some $\{b_{\alpha,i}\}\subset\mathcal{O}_h$. Hence, all of $\{f_\alpha\}_\alpha$ are locally analytic of order $s$ if all of $\{f_i\}_i$ are locally analytic of order $s$.
\end{proof}

By Proposition~\ref{Basic_prop}, for each $\ell$-th root of unity $\alpha$, $f_\alpha~;n\mapsto c_n^{(\alpha)}$ with $\{c_n^{(\alpha)}\}\in M_r^{(\alpha)}(\mathbb{C}_p)$ is locally analytic of order $s$ if and only if
\[\sum_{n=0}^{\infty}c_n^{(\alpha)}\frac{t^n}{n!}\in e^{\alpha t}\Bigl(\mathbb{C}_p\langle t|\ge\lfloor\log_pn\rfloor+r\rangle\cap\mathbb{C}_p\langle t|v_p(\lfloor n/p^s\rfloor!)\rangle\Bigr).\]

Thus, Proposition~\ref{smoothness_of_interpolation_prop} gives 
\[\mathcal{M}_{r,s}^{la}(\mathbb{C}_p)=\bigoplus_{\alpha^{\ell}=1}e^{\alpha t}\Bigl(\mathbb{C}_p\langle t|\ge\lfloor\log_pn\rfloor+r\rangle\cap\mathbb{C}_p\langle t|v_p(\lfloor n/p^s\rfloor!)\rangle\Bigr)\]
as a $\mathcal{O}_{\mathbb{C}_p}$-module.

Again, because for $\{c_n^{(\alpha)}\}\in\mathcal{M}_r^{(\alpha)}(\mathbb{C}_p)$, $\{c_n^{(\alpha)}\}\subset\mathcal{O}_{\mathbb{C}_p}$ if and only if 
\[\sum_{n=0}^{\infty}c_n^{(\alpha)}\frac{t^n}{n!}\in\mathbb{C}_p\langle t|\ge\Theta_r\rangle,\]
we also have
\[\mathcal{M}_{r,s}^{la}(\mathcal{O}_{\mathbb{C}_p})=\bigoplus_{\alpha^{\ell}=1}e^{\alpha t}\Bigl(\mathbb{C}_p\langle t|\ge\Theta_r\rangle\cap\mathbb{C}_p\langle t|v_p(\lfloor n/p^s\rfloor!)\rangle\Bigr)\]
as a $\mathcal{O}_{\mathbb{C}_p}$-module.

\section{Structures of spaces of generating functions}
\noindent
In this section, we investigate properties of $\mathbb{C}_p\langle t|\Theta\rangle$ and $\mathbb{C}_p\langle t|\ge\Theta\rangle$ with some specific $\Theta:\mathbb{Z}_{\ge0}\rightarrow\mathbb{R}\cup\{\pm\infty\}$ that are of particular interest to us. This ultimately allows us to show that $\mathcal{R}_r$ (defined by \eqref{R_r_def(R_r_alpha_decomposition_equ)}) forms a ring and to establish \eqref{R_r_rephrase_eq} and \eqref{R_r_decomposition_equ}.

\subsection{Properties of specific $\Theta$}
The propositions in the subsequent subsection require $\Theta:\mathbb{Z}_{\ge0}\rightarrow\mathbb{R}\cup\{\pm\infty\}$ to satisfy certain conditions so that $\mathbb{C}_p\langle t|\Theta\rangle$ or $\mathbb{C}_p\langle t|\ge\Theta\rangle$ will have good properties. We begin by showing that the specific $\Theta$s we are interested in (for example, $\Theta_r$ or $\Theta(n)=v_p(\lfloor n/p^s\rfloor!)$) satisfy the conditions required for the result below. 

\begin{lemma}
\label{concrete_Theta_satisfying_condition_lemma}
    (1) Let $\Theta(n)=v_p(\lfloor n/p^s\rfloor!)$ with some $s\in\mathbb{Z}_{\ge0}$. Then, we have
    \begin{enumerate}
        \item[(A)] $\Theta(m)+\Theta(n-m)+v_p(\binom{n}{m})\ge\Theta(n)$ for all $n\ge m\ge1$,
    \end{enumerate}
    Moreover, if $p+s\ge3$,
    \begin{enumerate}
        \item[(C')] $\mathbb{C}_p\langle t|\ge n\rangle\subset\mathbb{C}_p\langle t|\Theta\rangle$.
    \end{enumerate}
    
    (2) Let $\Theta:\mathbb{Z}_{\ge0}\rightarrow\mathbb{R}\cup\{\pm\infty\}$ be one of the following:
    \begin{enumerate}[label=(\alph*)]
        \item $\Theta=\Theta_r$ with some $r\ge0$;
        \item $\Theta(n)=Cn$ with a constant $0<C$;
        \item $\Theta(n)=v_p(n!)$.
    \end{enumerate}
    Then, we have
    \begin{enumerate}
        \item[(A)] $\Theta(m)+\Theta(n-m)+v_p(\binom{n}{m})\ge\Theta(n)$ for all $n\ge m\ge1$,
        \item[(B)] $\Theta(0)=0$,
        \item[(D)] $\lim_{n\rightarrow\infty}\Theta(n)=\infty$.
    \end{enumerate}
    Moreover, if $r\le1$ in (a) or $C\le1$ in (b),
    \begin{enumerate}
        \item[(C)] $\mathbb{C}_p\langle t|\ge n\rangle\subset\mathbb{C}_p\langle t|\ge\Theta\rangle$.
    \end{enumerate}
    (The labels (A) through (D) are assigned based on their intended subsequent use.)
\end{lemma}
\begin{proof}
    (1) By Proposition~\ref{val_of_n!_prop}, we have 
    \[v_p\left(\left\lfloor\frac{n}{p^s}\right\rfloor!\right)\le\left\lfloor\frac{n}{p^s}\right\rfloor\frac{1}{p-1}\le\frac{n}{p^s(p-1)}\]
    for all $n\ge0$. Thus, $\lim_{n\rightarrow\infty}(n-v_p(\lfloor n/p^s\rfloor!))=\infty$ if $p+s\ge3$. Thus, we have (C') in this case. 

    We prove (A). Let $n\ge m\ge1$ and let
    \begin{align*}
        n&=a_xp^x+\cdots+a_sp^s+\cdots+a_1p+a_0\quad(x\ge s,~0\le a_0,\cdots,a_x\le p-1),\\
        m&=b_yp^y+\cdots+b_sp^s+\cdots+b_1p+b_0\quad(y\ge s,~0\le b_0,\cdots,b_y\le p-1),\\
        n-m&=c_zp^z+\cdots+c_sp^s+\cdots+c_1p+b_0\quad(z\ge s,~0\le c_0,\cdots,c_z\le p-1)
    \end{align*}
    be base-$p$ expansions of $n,m,n-m$, respectively. Then, by Proposition~\ref{val_of_n!_prop} and Proposition~\ref{val_of_binom_coeff_prop}, we have
    \begin{align*}
        v_p\left(\left\lfloor\frac{n}{p^s}\right\rfloor!\right)&=\frac{1}{p-1}\left(a_xp^{x-s}+\cdots+a_{s+1}p+a_s-\sum_{i\ge s}a_i\right),\\
        v_p\left(\left\lfloor\frac{m}{p^s}\right\rfloor!\right)&=\frac{1}{p-1}\left(b_yp^{y-s}+\cdots+b_{s+1}p+b_s-\sum_{i\ge s}b_i\right),\\
        v_p\left(\left\lfloor\frac{n-m}{p^s}\right\rfloor!\right)&=\frac{1}{p-1}\left(c_zp^{z-s}+\cdots+c_{s+1}p+c_s-\sum_{i\ge s}c_i\right),\\
        v_p\left(\binom{n}{m}\right)&=\frac{1}{p-1}\sum_{i\ge0}(b_i+c_i-a_i).
    \end{align*}
    Therefore, we obtain
    \begin{align*}
        &\Theta(m)+\Theta(n-m)+v_p(\binom{n}{m})-\Theta(n)\\
        =&\frac{1}{p-1}\left(\sum_{i=0}^{s-1}(b_i+c_i-a_i)+(b_yp^{y-s}+\cdots+b_0+c_zp^{z-s}+\cdots+c_0-a_xp^{x-s}-\cdots-a_0)\right)\\
        =&\frac{1}{p-1}\left(\sum_{i=0}^{s-1}(b_i+c_i-a_i)+\frac{1}{p^s}\left(m+(n-m)-n+\sum_{i=0}^{s-1}p^i(b_i+c_i-a_i)\right)\right)\\
        =&\frac{1}{p-1}\sum_{i=0}^{s-1}\left(1+\frac{p^i}{p^s}\right)(b_i+c_i-a_i)\\
        \ge&\frac{1}{p-1}\sum_{i=0}^{s-1}\left(\frac{p^i}{p^s}+\frac{p^i}{p^s}\right)(b_i+c_i-a_i)\\
        =&\frac{2}{p^s(p-1)}\left(\sum_{i=0}^{s-1}p^ib_i+\sum_{i=0}^{s-1}p^ic_i-\sum_{i=0}^{s-1}p^ia_i\right).
    \end{align*}
    Since $m+(n-m)=n$ and $a_i,b_i,c_i$ were base-$p$ digits of $n,m,n-m$ respectively, 
    \[\sum_{i=0}^{s-1}p^ib_i+\sum_{i=0}^{s-1}p^ic_i-\sum_{i=0}^{s-1}p^ia_i\ge0.\]
    Thus, (A) also holds.
    
    (2) \textbf{(a):} (B) and (D) are clear by definition. If $r\le1$, we have (C) because $\Theta_r(n)\le n$ for all $n\ge0$.

    We prove (A). Since $r\ge0$, it is sufficient to show in the case of $r=0$. By setting $n=i+j$, the problem is reduced to showing 
    \[v_p(\binom{i+j}{j})\ge\lfloor\log_p(i+j)\rfloor-\lfloor\log_pi\rfloor-\lfloor\log_pj\rfloor\quad\text{for all }i\ge j\ge1.\]
    Since $\lfloor\log_p(i+j)\rfloor\le\lfloor\log_p2i\rfloor\le\lfloor\log_pi\rfloor+1$, when $j\ge p$, we have
    \[v_p(\binom{i+j}{j})\ge0\ge1-\lfloor\log_pj\rfloor\ge\lfloor\log_p(i+j)\rfloor-\lfloor\log_pi\rfloor-\lfloor\log_pj\rfloor.\]
    
    So consider the case of $1\le j\le p-1$. Then, we have 
    \[\lfloor\log_p(i+j)\rfloor-\lfloor\log_pi\rfloor-\lfloor\log_pj\rfloor=\lfloor\log_p(i+j)\rfloor-\lfloor\log_pi\rfloor\le1.\]
    When none of $i+1,\cdots,i+j$ are multiples of $p$, 
    \[\lfloor\log_p(i+j)\rfloor-\lfloor\log_pi\rfloor=0\]
    If there exists a positive integer $m$ such that $i<mp\le i+j$, we have $v_p(\binom{i+j}{j})\ge1$ because $j!$ is not a multiple of $p$. In any case, we obtain
    \[v_p(\binom{i+j}{j})\ge\lfloor\log_p(i+j)\rfloor-\lfloor\log_pi\rfloor-\lfloor\log_pj\rfloor.\]

    \textbf{(b):} (B) is clear. (A) holds because $v_p(\binom{n}{m})\ge0$. The condition $0<C$ leads (D), and $C\le1$ leads (C).

    \textbf{(c):} (B) and (D) are clear. In this case, $\Theta(n)-\Theta(m)-\Theta(n-m)$ is precisely $v_p(\binom{n}{m})$. Thus, we have (A). Since $v_p(n!)\le n/(p-1)$ for all $n\ge0$ by Proposition~\ref{val_of_n!_prop}, we have (C) too.
\end{proof}

Note that $\mathbb{C}_p\langle t|\ge n\rangle\subset\mathbb{C}_p\langle t|v_p(\lfloor n/p^s\rfloor!)\rangle$ does not hold when $p+s<3$, that is, $p=2$ and $s=1$. This is because $v_2(2^k)=2^k-1$ for all $k\ge0$ and thus, $\liminf_{n\rightarrow\infty}(n-v_2(\lfloor n/2^0\rfloor!))\ne\infty$.

\subsection{Change of basis and ring structure}
\label{Change_of_basis_and_ring_stru_subsection}

\begin{proposition}
\label{subring_prop}
    (1) Let $\Theta:\mathbb{Z}_{\ge0}\rightarrow\mathbb{R}\cup\{\pm\infty\}$ be a function satisfying
    \begin{enumerate}[label=(\Alph*)]
        \item $\Theta(m)+\Theta(n-m)+v_p(\binom{n}{m})\ge\Theta(n)$ for all $n\ge m\ge1$.
    \end{enumerate}
    Then, $\mathbb{C}_p\langle t|\Theta\rangle$ is a subring of $\mathbb{C}_p\llbracket t\rrbracket$.
    
    (2) Let $\Theta:\mathbb{Z}_{\ge0}\rightarrow\mathbb{R}\cup\{\pm\infty\}$ be a function satisfying
    \begin{enumerate}[label=(\Alph*)]
        \item $\Theta(m)+\Theta(n-m)+v_p(\binom{n}{m})\ge\Theta(n)$ for all $n\ge m\ge1$, 
        \item $\Theta(0)=0$.
    \end{enumerate}
    Then, $\mathbb{C}_p\langle t|\ge\Theta\rangle$ is a subring of $\mathbb{C}_p\llbracket t\rrbracket$. Moreover, $\mathbb{C}_p\langle t|\ge\Theta\rangle^{\times}=\{F(t)\in\mathbb{C}_p\langle t|\ge\Theta\rangle\mid v_p(F(0))=0\}$.
\end{proposition}
\begin{proof}
    Let $F(t)=\sum_{n=0}^{\infty}b_nt^n/n!,G(t)=\sum_{n=0}^{\infty}c_nt^n/n!\in \mathbb{C}_p\llbracket t\rrbracket$. Then,
    \begin{equation}
    \label{F(T)G(T)_expand_equ}
        F(t)G(t)=\sum_{n=0}^{\infty}\left(\sum_{m=0}^{n}\binom{n}{m}b_mc_{n-m}\right)\frac{t^n}{n!}.
    \end{equation}

    (1) It is immediate that $0,1\in\mathbb{C}_p\langle t|\Theta\rangle$ and $\mathbb{C}_p\langle t|\Theta\rangle$ is closed under addition.
    
    Suppose $F(t),G(t)\in\mathbb{C}_p\langle t|\Theta\rangle$.
    Then, there exists $K\in\mathbb{R}$ such that 
    \[-K\le\min\left\{v_p(b_0),\inf_{m\ge1}(v_p(b_m)-\Theta(m)),v_p(c_0),\inf_{m\ge1}(v_p(c_m)-\Theta(m))\right\}.\]

    For each $L\in\mathbb{R}$, there exists $n_L\in\mathbb{Z}_{\ge0}$ such that
    \[v_p(b_n)-\Theta(n)\ge L+K,\quad v_p(c_n)-\Theta(n)\ge L+K\quad\text{for all }n\ge n_L.\]

    Then, for any $n\ge2n_L$ and $0\le m\le n$, we have
    \[v_p\left(\binom{n}{m}b_mc_{n-m}\right)\ge\Theta(n)+L.\]
    by the definition of $K$ and the condition (A). 
    Hence,
    \[v_p\left(\sum_{m=0}^{n}\binom{n}{m}b_mc_{n-m}\right)\ge\Theta(n)+L,\]
    for all $n\ge 2n_L$. This means $F(t)G(t)\in\mathbb{C}_p\langle t|\Theta\rangle$. Thus, $\mathbb{C}_p\langle t|\Theta\rangle$ is a subring.

    (2) It is also immediate that $0,1\in\mathbb{C}_p\langle t|\ge\Theta\rangle$ and $\mathbb{C}_p\langle t|\ge\Theta\rangle$ is closed under addition.

    Suppose $F(t),G(t)\in \mathbb{C}_p\langle t|\ge\Theta\rangle$. Then, we again have \eqref{F(T)G(T)_expand_equ} and 
    \[v_p\left(\sum_{m=0}^{n}\binom{n}{m}b_mc_{n-m}\right)\ge\min\left\{v_p(b_m)+v_p(c_{n-m})+v_p\left(\binom{n}{m}\right)\right\}\ge\Theta(n)\]
    by conditions (A) and (B). So we have $F(t)G(t)\in \mathbb{C}_p\langle t|\ge\Theta\rangle$, and thus $\mathbb{C}_p\langle t|\ge\Theta\rangle$ is a subring.

    We determine its unit group. Suppose $F(t)\in \mathbb{C}_p\langle t|\ge\Theta\rangle$, $v_p(b_0)=0$, and $F(t)G(t)=1$. Then, by (B),
    \[v_p(c_0)=v_p\left(\frac{1}{b_0}\right)=0=\Theta(0).\]
    Let $n\ge1$ and assume $v_p(c_k)\ge\Theta(k)$ for all $0\le k\le n-1$. Then, by (A),
    \[v_p(c_n)=v_p\left(-\frac{1}{b_0}\sum_{m=1}^{n}\binom{n}{m}b_mc_{n-m}\right)\ge\Theta(n).\]
    Therefore, we have $v_p(c_n)\ge\Theta(n)$ by induction and $G(t)\in\mathbb{C}_p\langle t|\ge\Theta\rangle$. Thus, $\mathbb{C}_p\langle t|\ge\Theta\rangle^{\times}$ is contained in $\{F(t)\in\mathbb{C}_p\langle t|\ge\Theta\rangle\mid v_p(F(0))=0\}$. The reverse inclusion is obvious by (B).
\end{proof}

\begin{corollary}
    Let $p$ be an odd prime, $r\ge0$ be a real number, and $F(t)$ be the exponential generating function of a sequence $\{c_n\}_{n=0}^{\infty}\subset\mathcal{O}_{\mathbb{C}_p}$ with $v_p(c_0)=0$. Let $\{d_n\}_{n=0}^{\infty}$ be the sequence generated by 
    \[\frac{1}{F(t)}=\sum_{n=0}^{\infty}d_n\frac{t^n}{n!}.\]
    Then, we have 
    \[c_{n+p^k}\equiv c_n\mod{p^{k+r}}\quad\text{for all }k,n\ge0\]
    if and only if 
    \[d_{n+p^k}\equiv -d_n\mod{p^{k+r}}\quad\text{for all }k,n\ge0.\]
\end{corollary}
\begin{proof}
    By Proposition~\ref{Basic_prop}, Proposition~\ref{subring_prop}, and $v_p(c_0)=0$, we have
    \begin{align*}
        &~c_{n+p^k}\equiv c_n\mod{p^{k+r}}\quad\text{for all }k,n\ge0\\
        \Leftrightarrow&~F(t)\in e^t\mathbb{C}_p\langle t|\ge\Theta_r\rangle\Leftrightarrow\frac{1}{F(t)}\in e^{-t}\mathbb{C}_p\langle t|\ge\Theta_r\rangle\\
        \Leftrightarrow&~d_{n+p^k}\equiv (-1)^{p^k}d_n\equiv-d_n\mod{p^{k+r}}\quad\text{for all }k,n\ge0.
    \end{align*}
\end{proof}

\begin{remark}
\label{why_<ge>_only_remark}
    We determined the unit group of $\mathbb{C}_p\langle t|\ge\Theta\rangle$ in Proposition~\ref{subring_prop}, as it will be used in the next subsection. While the unit elements of $\mathbb{C}_p\langle t|\ge\Theta\rangle$ can be determined only by constant terms, those of $\mathbb{C}_p\langle t|\Theta\rangle$ cannot. For example, for any $a\in\mathbb{C}_p^\times$ and $s\in\mathbb{Z}_{\ge0}$,
    \[F(T)=a-\frac{aT}{p}\in\mathbb{C}_p\langle t|v_p(\lfloor n/p^s\rfloor!)\rangle\quad\text{but}\quad F(T)^{-1}\notin\mathbb{C}_p\langle t|v_p(\lfloor n/p^s\rfloor!)\rangle.\] In view of this distinction, we restrict our attention to the type of $\mathbb{C}_p\langle t|\ge\Theta\rangle$, but not $\mathbb{C}_p\langle t|\Theta\rangle$ in the next subsection. 
    
    Instead, we can utilize the following inclusion relation; Since $v_p(\lfloor n/p^s\rfloor!)$ is evaluated as
    \[v_p\left(\left\lfloor\frac{n}{p^s}\right\rfloor!\right)\le\frac{n}{p^s(p-1)}\quad\text{for all }n\ge0\]
    by Proposition~\ref{val_of_n!_prop}, we have
    \[\mathbb{C}_p\langle t|\ge Cn\rangle\subset\mathbb{C}_p\langle t|v_p(\lfloor n/p^s\rfloor!)\rangle\quad\text{for }C>\frac{1}{p^s(p-1)}.\] 
    Of course, $\mathbb{C}_p\langle t|\ge Cn\rangle$ is not a ``good approximation" because $\mathbb{C}_p\langle t|\ge Cn\rangle$ restricts the valuations of all coefficients while $\mathbb{C}_p\langle t|v_p(\lfloor n/p^s\rfloor!)\rangle$ imposes no conditions on the initial finite terms. However, given the trade-off between the range of applications and the technical difficulty, this seems to be a useful substitute.
\end{remark}

In the introduction, we provided \eqref{R_r(O_Cp)_change_of_basis} and \eqref{R_r,s(O_Cp)_change_of_basis}, and we also stated that $\mathcal{R}_r$ with $r\in[0,1]$ is a ring and that \eqref{R_r_decomposition_equ} holds. The following theorem contains these statements as special cases and will also be used in the next section.
\begin{theorem}
\label{A,M',R_properties_thm}
    Let $\Theta^{(1)},\cdots,\Theta^{(u)}:\mathbb{Z}_{\ge0}\rightarrow\mathbb{R}\cup\{\pm\infty\}$ be functions satisfying
    \begin{enumerate}
        \item[(A)] $\Theta^{(i)}(m)+\Theta^{(i)}(n-m)\ge\Theta^{(i)}(n)$ for all $n\ge m\ge1$,
        \item[(C')] $\mathbb{C}_p\langle t|\ge n\rangle\subset\mathbb{C}_p\langle t|\Theta^{(i)}\rangle$,
    \end{enumerate}
    for all $1\le i\le u$ and let $\Theta^{(u+1)},\cdots,\Theta^{(u+v)}:\mathbb{Z}_{\ge0}\rightarrow\mathbb{R}\cup\{\pm\infty\}$ be functions satisfying 
    \begin{enumerate}
        \item[(A)] $\Theta^{(u+i)}(m)+\Theta^{(u+i)}(n-m)\ge\Theta^{(u+i)}(n)$ for all $n\ge m\ge1$,
        \item[(B)] $\Theta^{(u+i)}(0)=0$,
        \item[(C)] $\mathbb{C}_p\langle t|\ge n\rangle\subset\mathbb{C}_p\langle t|\ge\Theta^{(u+i)}\rangle$
    \end{enumerate}
    for all $1\le i\le v$.

    Define 
    \begin{align*}
        \mathcal{A}&\coloneqq\left(\bigcap_{i=1}^{u}\mathbb{C}_p\langle t|\Theta^{(i)}\rangle\right)\cap\left(\bigcap_{i=1}^{v}\mathbb{C}_p\langle t|\ge\Theta^{(u+i)}\rangle\right),\\
        \mathcal{M'}&\coloneqq\sum_{\alpha^{\ell}=1}e^{\alpha t}\mathcal{A}=\left\{\sum_{\alpha^{\ell}=1}F_\alpha(t)~\middle|~F_{\alpha}(t)\in\mathcal{A}\right\}\subset\mathbb{C}_p\llbracket t\rrbracket,\\
        \mathcal{R}&\coloneqq\sum_{\alpha^q=\alpha}e^{\alpha t}\mathcal{A}=\left\{\sum_{\alpha^q=\alpha}F_\alpha(t)~\middle|~F_{\alpha}(t)\in\mathcal{A}\right\}\subset\mathbb{C}_p\llbracket t\rrbracket\quad(\text{see \eqref{q_def} for }q).
    \end{align*}

    Then, 

    (1) For any $a_1,\cdots,a_m\in\mathbb{C}_p$ such that 
    $a_1\equiv\cdots\equiv a_m\mod{p}$, 
    \[e^{a_1 t}\mathcal{A}=\cdots=e^{a_m t}\mathcal{A}=\sum_{i=1}^{m}e^{a_i t}\mathcal{A}.\]
    
    In particular, for any subset $A\subset\mathcal{O}_{\mathbb{C}_p}$ such that for each $\alpha\in\{x\in\mathcal{O}_h\mid x^q=x\}$, there is a unique $a\in A$ satisfying $v_p(a-\alpha)\ge1$,
    \[\mathcal{M'}=\sum_{\substack{a\in A\\a^\ell\equiv1}}e^{at}\mathcal{A},\quad\mathcal{R}=\sum_{a\in A}e^{at}\mathcal{A}\]

    (2) $\mathcal{R}$ is a subring of $\mathbb{C}_p\llbracket t\rrbracket$.

    (3) $\mathcal{M'}$ is an $\mathcal{A}$-module by addition and multiplication in $\mathbb{C}_p\llbracket t\rrbracket$, and $\mathcal{R}$ is an $\mathcal{A}$-algebra by inclusion map $\mathcal{A}\hookrightarrow\mathcal{R}$.

    (4) If at least one of the $\Theta^{(u+i)}$ ($1\le i\le v$) is equal to $\Theta_r$ with some $r\in[0,1]$, for any subset $A\subset\mathcal{O}_{\mathbb{C}_p}$ such that for each $\alpha\in\{x\in\mathcal{O}_h\mid x^q=x\}$, there is a unique $a\in A$ satisfying $v_p(a-\alpha)\ge1$, 
    \[\mathcal{M'}=\bigoplus_{\substack{a\in A\\a^\ell\equiv1}}e^{at}\mathcal{A},\quad\mathcal{R}=\bigoplus_{a\in A}e^{at}\mathcal{A}\]
    as an $\mathcal{A}$-module.
\end{theorem}
\begin{proof}
    (1) It is enough to show
    \[e^{a_1t}\mathcal{A}=e^{a_2t}\mathcal{A}=\sum_{i=1}^{m}e^{a_it}\mathcal{A}\]
    Let $c=(a_1-a_2)/p\in\mathcal{O}_{\mathbb{C}_p}$. Then,
    \[e^{(a_1-a_2)t}=\sum_{n=0}^{\infty}\frac{(pc)^n}{n!}t^n\in\mathbb{C}_p\langle t|\ge n\rangle\subset\mathcal{A}.\]
    Since $\mathcal{A}$ is a subring of $\mathbb{C}_p\llbracket t\rrbracket$ by Proposition~\ref{subring_prop}, we have $e^{a_1t}\mathcal{A}=e^{a_2t}e^{(a_1-a_2)t}\mathcal{A}\subset e^{a_2t}\mathcal{A}$, and thus, we obtain $e^{a_1t}\mathcal{A}=e^{a_2t}\mathcal{A}$ by symmetry.

    This also leads
    \[\sum_{i=1}^{m}e^{a_it}\mathcal{A}=\sum_{i=1}^{m}e^{a_1t}\mathcal{A}\subset e^{a_1t}\mathcal{A}\]
    because $\mathcal{A}$ is closed under addition. The reverse inclusion is clear since $0\in\mathcal{A}$.

    (2) Let $B=\{\alpha\in\mathcal{O}_h~|~\alpha^{q}=\alpha\}$. 
    Since $B$ contains $1$, $\mathcal{A}\subset\mathcal{R}$. By Proposition~\ref{subring_prop}, $\mathcal{A}$ is a ring, so $0,1\in\mathcal{R}$ and $\mathcal{R}$ is closed under addition.

    Because $B$ is a complete set of representatives for $\mathcal{O}_h/\mathfrak{m}_h$, for each $\alpha, \beta\in B$, there exists a unique $\gamma_{\alpha\beta}\in B$ such that $\gamma_{\alpha\beta}\equiv\alpha+\beta\mod{p}$.
    Thus, 
    \[\left(\sum_{\alpha^q=\alpha}e^{\alpha t}\mathcal{A}\right)\left(\sum_{\alpha^q=\alpha}e^{\alpha t}\mathcal{A}\right)=\sum_{\alpha,\beta\in B}e^{(\alpha+\beta)t}\mathcal{A}=\sum_{\gamma\in B}e^{\gamma t}\mathcal{A},\]
    and $\mathcal{R}$ is closed under multiplication.

    (3) It is immediate from (2) and the fact that $\mathcal{A}$ is a ring.

    (4) Assume that at least one of the $\Theta^{(u+i)}$ ($1\le i\le v$) is equal to $\Theta_r$ with some $r\in[0,1]$.
    
    Let $\{C_\alpha(t)\}_{\alpha^q=\alpha}\subset\mathcal{A}$ and assume
    \[\sum_{\alpha^q=\alpha}C_\alpha(t)e^{\alpha t}=\sum_{\alpha^{\ell}=1}C_\alpha(t)e^{\alpha (t)}+C_0(t)=0.\]
    Because We already have \eqref{R_r(O_C_p)_decomposition} and since
    \[C_0(t)\in\mathcal{A}\subset\mathbb{C}_p\langle t|\ge\Theta_r\rangle\subset\mathbb{C}_p\langle t|0\rangle\cap\mathbb{C}_p\langle t|\ge0\rangle,\]
    it follows that $C_\alpha(t)=0$ for all $\alpha$. Therefore, 
    \[\mathcal{R}=\bigoplus_{\alpha^q=\alpha}e^{\alpha t}\mathcal{A}.\]

    Thus, the statement follows from (1).
\end{proof}

Since $K_h/\mathbb{Q}_p$ is an unramified extension, $v_p(\mathcal{O}_h)\subset\mathbb{Z}$. Thus, when $A\subset\mathcal{O}_h$ is a complete set of representatives for $\mathcal{O}_h/\mathfrak{m}_h$, it satisfies the condition: for each $\alpha\in\{x\in\mathcal{O}_h\mid x^q=x\}$, there is a unique $a\in A$ satisfying $v_p(a-\alpha)\ge1$. Therefore, we have the following corollaries as a special case of Theorem~\ref{A,M',R_properties_thm}.

\begin{corollary}
\label{change_of_basis_cor}
    Let $r\in[0,1]$ and let $A\subset\mathcal{O}_h$ be a complete set of representatives in $\mathcal{O}_h$ for $\mathcal{O}_h/\mathfrak{m}_h$. Then, 
    \[\mathcal{M}_{p,\ell,r}(\mathcal{O}_{\mathbb{C}_p})=\left(\bigoplus_{a\in A,~a^\ell\equiv1}e^{at}\mathbb{C}_p\langle t|\ge\Theta_r\rangle\right)\oplus\Bigl(\mathbb{C}_p\langle t|0\rangle\cap\mathbb{C}_p\langle t|\ge0\rangle\Bigr).\]
\end{corollary}

\begin{corollary}
\label{change_of_basis_la_cor}
    Let $r\in[0,1]$, $s\in\mathbb{Z}_{\ge0}$, and $A$ be a complete set of representatives in $\mathcal{O}_h$ for $\mathcal{O}_h/\mathfrak{m}_h$. Suppose $p+s\ge3$. Then, 
    \begin{align*}
        \mathcal{M}_{p,\ell,r,s}^{la}(\mathcal{O}_{\mathbb{C}_p})=\bigoplus_{a\in A,~a^\ell\equiv1}e^{at}\Bigl(\mathbb{C}_p\langle t|\ge\Theta_r\rangle\cap\mathbb{C}_p\langle t|v_p(\lfloor n/p^s\rfloor!)\rangle\Bigr)
    \end{align*}
\end{corollary}

The following Corollary is easy but useful in applications.
\begin{theorem}
\label{application_thm}
    Let $r\in[0,1]$ and let $\{c_n\}_{n=0}^{\infty}\subset\mathbb{C}_p$.
    Consider a function $\Theta:\mathbb{Z}_{\ge0}\rightarrow\mathbb{R}\cup\{\pm\infty\}$ such that $\Theta(n)\rightarrow\infty$ when $n\rightarrow\infty$, and set for each $k\in\mathbb{Z}_{\ge0}$, 
    \[n_k\coloneqq\min\{n'\in\mathbb{Z}_{\ge0}~|~\Theta(n)\ge k+r\text{ for all }n\ge n'\}.\]

    If there exist $a_1,\cdots,a_m\in\mathcal{O}_h^{\times}$, $F_1(t),\cdots,F_m(t)\in\mathbb{C}_p\langle t|\ge\Theta_r\rangle$, and $G(t)\in\mathbb{C}_p\langle t|\ge\Theta\rangle$ such that 
    \[a_i^\ell\equiv1\mod{p}\]
    for each $i=1,\cdots,m$ and
    \[\sum_{n=0}^{\infty}c_n\frac{t^n}{n!}=\sum_{i=1}^{m}\exp{(a_it)}F_i(t)+G(t),\]
    then, for each $k\in\mathbb{Z}_{\ge0}$,
    \[c_{n+\ell p^k}\equiv c_n\mod{p^{k+r}}\quad\text{for all }n\ge n_k.\]
\end{theorem}
\begin{proof}
    Let $\{c'_n\},\{d_n\}\subset\mathbb{C}_p$ be the sequence defined by
    \[\sum_{n=0}^{\infty}c'_n\frac{t^n}{n!}=\sum_{i=1}^{m}\exp{(a_it)}F_i(t),\quad\sum_{n=0}^{\infty}d_n\frac{t^n}{n!}=G(t).\]

    Then, by Theorem~\ref{A,M',R_properties_thm}(1), \eqref{M_r^alpha(O_Cp)_gen_func_equ}, and \eqref{M'_r(R)_Decomposition_equ}, we have $\{c'_n\}\in M'_r(\mathbb{C}_p)$.
    
    For each $k\in\mathbb{Z}_{\ge0}$, by the definition of $n_k$ and the assumption $G(t)\in\mathbb{C}_p\langle t|\ge\Theta\rangle$, it follows that
    \[v_p(d_n)\ge\Theta(n)\ge k+r\quad\text{for all }n\ge n_k.\]

    Therefore, $\{c_n\}=\{c'_n\}+\{d_n\}\in M_r(\mathbb{C}_p)$ satisfies 
    \[c_{n+\ell p^k}\equiv c_n\mod{p^{k+r}}\]
    for any $k\in\mathbb{Z}_{\ge0}$ and $n\ge n_k$.
\end{proof}

Now, we can derive \eqref{R_r_rephrase_eq}. Suppose $\ell=p^h-1$.

Let $r\in[0,1]$ and assume $\sum_{n=0}^{\infty}c_nt^n/n!\in\mathcal{R}_r$.

Since $\Theta_r(n)\ge0$ for all $n\in\mathbb{Z}_{\ge0}$ and $\lim_{n\rightarrow\infty}\Theta_r(n)=\infty$, we have
\[\mathbb{C}_p\langle t|\ge\Theta_r\rangle\subset\mathbb{C}_p\langle t|0\rangle\cap\mathbb{C}_p\langle t|\ge0\rangle,\]
and thus, $\mathcal{R}_r\subset\mathcal{M}_r(\mathcal{O}_{\mathbb{C}_p})$.

For each $k\in\mathbb{Z}_{\ge0}$, 
\[\min\{n'\in\mathbb{Z}_{\ge0}~|~\Theta_r(n)\ge k+r\text{ for all }n\ge n'\}=\begin{cases}
    0 & (\text{if }k=r=0), \\
    p^k & (\text{otherwise}).
\end{cases}\]

Therefore, by Theorem~\ref{application_thm}, we have 
\[\mathcal{R}_r\subset\left\{\sum_{n=0}^{\infty}c_n\frac{t^n}{n!}~\middle|~
    \begin{array}{c}
        \{c_n\}\subset\mathcal{O}_{\mathbb{C}_p}\text{ and }\\
        \{c_n\}\text{ satisfies \eqref{Basic_equ} with }n_k\le p^k
    \end{array}
\right\}.\]

Conversely, if a sequence $\{c_n\}_{n=0}^{\infty}\subset\mathcal{O}_{\mathbb{C}_p}$ satisfies
\[c_{n+(p^h-1)p^k}\equiv c_n\mod{p^{k+r}}\text{ for all }k\in\mathbb{Z}_{\ge0}\text{ and }n\ge p^k,\]
by Lemma~\ref{c'_n_and_d_n_lemma}, there exist $\{c'_n\},\{d_n\}\subset\mathcal{O}_{\mathbb{C}_p}$ such that 
\begin{align*}
    &\{c_n\}=\{c'_n\}+\{d_n\},\\
    &\{c'_n\}\in M'_r(\mathcal{O}_{\mathbb{C}_p}), \\
    & v_p(d_n)\ge k+r\quad\text{for all }k\in\mathbb{Z}_{\ge0}\text{ and }n\ge n_k.
\end{align*}
Then, by \eqref{M'_r(R)_Decomposition_equ} and \eqref{M_r^alpha(O_Cp)_gen_func_equ}, 
\[\sum_{n=0}^{\infty}c'_n\frac{t^n}{n!}\in\bigoplus_{\alpha^{q-1}=1}e^{\alpha t}\mathbb{C}_p\langle t|\ge\Theta_r\rangle.\]
Moreover, 
\[v_p(d_0)\ge0=\Theta_r(0)\]
because $d_n\in\mathcal{O}_{\mathbb{C}_p}$, and for any $n\ge1$, 
\[v_p(d_n)\ge k_n+r=\lfloor\log_pn\rfloor+r=\Theta_r(n)\]
where $k_n$ is a unique nonnegative integer such that $p^{k_n}\le n<p^{k_n+1}$.
Therefore, we have
\[\sum_{n=0}^{\infty}d_n\frac{t^n}{n!}\in\mathbb{C}_p\langle t|\ge\Theta_r\rangle,\]
and hence,
\[\sum_{n=0}^{\infty}c_n\frac{t^n}{n!}\in\mathcal{R}_r.\]
So we also obtain the reverse inclusion relationship
\[\left\{\sum_{n=0}^{\infty}c_n\frac{t^n}{n!}~\middle|~
    \begin{array}{c}
        \{c_n\}\subset\mathcal{O}_{\mathbb{C}_p}\text{ and }\\
        \{c_n\}\text{ satisfies \eqref{Basic_equ} with }n_k\le p^k
    \end{array}
\right\}\subset\mathcal{R}_r.\]

\subsection{Preservation of congruences under reciprocation}
\label{Reciprocation_subsection}
In this subsection, we consider the case of $\ell=p^h-1$ to further investigate the structure of $\mathcal{R}_r$ as a ring.

Let $\Theta_1,\cdots,\Theta_v:\mathbb{Z}_{\ge0}\rightarrow\mathbb{R}\cup\{\pm\infty\}$ be functions satisfying the following conditions: for all $1\le i\le v$,
\begin{enumerate}[label=(\Alph*)]
    \item $\Theta_i(m)+\Theta_i(n-m)\ge\Theta_i(n)$ for all $n\ge m\ge1$;
    \item $\Theta_i(0)=0$;
    \item $\mathbb{C}_p\langle t|\ge n\rangle\subset\mathbb{C}_p\langle t|\ge\Theta_i\rangle$,
\end{enumerate}
and for at least one $i$, which we may assume without loss of generality to be $i=1$,
\begin{enumerate}[label=(\Alph*), start=4]
    \item $\lim_{n\rightarrow\infty}\Theta_1(n)=\infty$.
\end{enumerate}

Let 
\[\mathcal{A}=\bigcap_{i=1}^{v}\mathbb{C}_p\langle T|\ge\Theta_i\rangle\quad\text{and}\quad\mathcal{R}=\sum_{\alpha^q=\alpha}e^{\alpha t}\mathcal{A}.\]
By Proposition~\ref{subring_prop} and Theorem~\ref{A,M',R_properties_thm}, $\mathcal{A}$ and $\mathcal{R}$ are subrings of $\mathbb{C}_p\llbracket t\rrbracket$ and, in particular, are integral domains. Let
\[
\mathcal{K}=\operatorname{Frac}(\mathcal{A})\quad\text{and}\quad\mathcal{L}=\operatorname{Frac}(\mathcal{R})
\]
be their fields of fractions.

Fix a $x\in\mathcal{O}_h$ such that $\{1,\overline{x},\cdots,\overline{x}^{h-1}\}$ forms a basis of the $\mathbb{F}_{p}$-vector space $\mathcal{O}_h/\mathfrak{m}_h\cong\mathbb{F}_{q}$, where $\overline{x}$ denotes the reduction of $x$ in $\mathcal{O}_h/\mathfrak{m}_h$.

The goal of this subsection is to determine $\mathcal{R}^{\times}$ by using the norm map of the field extension $\mathcal{L}/\mathcal{K}$.

\begin{lemma}
\label{(e^ut_notin_K_lemma}
    If $u\in\mathcal{O}_{\mathbb{C}_p}^{\times}$, then $e^{ut}\notin\mathcal{K}$.
\end{lemma}
\begin{proof}
    Assume that there exists $A(t),B(t)\in\mathcal{A}\setminus\{0\}$ such that $e^{ut}=A(t)/B(t)$
    Let
    \[A(t)=\sum_{n=0}^{\infty}a_n\frac{t^n}{n!},\quad B(t)=\sum_{n=0}^{\infty}b_n\frac{t^n}{n!}.\]

    Let $j\ge0$ be the minimum nonnegative integer such that $v_p(b_j)\ne0$.
    
    Then, we have
    \begin{equation}
    \label{an_equ}
        a_{n+j}=\sum_{m=0}^{n}\binom{n+j}{m+j}b_{m+j}u^{n-m}\quad\text{for all }n\ge0.
    \end{equation}
    
    Take a positive integer $N$ such that $v_p(b_j)<v_p(j!)+N$.
    By the Assumption (D) $\lim_{n\rightarrow\infty}\Theta_1(n)=\infty$, there are some $M$ such that 
    \[\Theta_1(m)>v_p(j!)+N\quad\text{for all }m\ge M.\]
    Since $v_p(b_j)\ge\Theta_1(j)$, we have $M>j$.
    Take a positive integer $K$ such that $\min_{1\le i<M}(v_p(b_i)-v_p(i!))>-K$, and 
    take a positive integer $L$ such that $j\le p^L$.
    
    Let $n=Mp^{\min\{N+K,L\}}$.

    Then, since $v_p(n)\ge L$ and $p^L\ge j$, by Kummer's theorem (Proposition~\ref{val_of_binom_coeff_prop}), we have
    \begin{equation}
    \label{v_p((n+j)!)=v_p(n!)+v_p(j!)_equ}
        v_p\left(\binom{n+j}{j}\right)=v_p((n+j)!)-v_p(n!)-v_p(j!)=0.
    \end{equation}

    We have 
    \[v_p(a_{n+j})\ge\Theta_1(n+j)>v_p(j!)+N\quad\quad(\text{since }n+j\ge n\ge M)\]
    and for all $1\le m< M-j$,
    \begin{align*}
        v_p\left(\binom{n+j}{m+j}b_{m+j}u^{n-m}\right)&\ge\min_{1\le i< M}(v_p(b_i)-v_p(i!))+v_p((n+j)!)-v_p((n-m)!)\\
        &>-K+v_p((n-j)!)-v_p(n!)+v_p(n)\quad(\text{since }n-m<n)\\
        &\ge v_p((n-j)!)-v_p(n!)+N\quad\quad\quad(\text{by the definition of }n)\\
        &=v_p(j!)+N\qquad\qquad\qquad(\text{by \eqref{v_p((n+j)!)=v_p(n!)+v_p(j!)_equ}}),
    \end{align*}
    and for all $M-j\le m\le n$,
    \[
        v_p\left(\binom{n+j}{m+j}b_{m+j}u^{n-m}\right)\ge v_p(b_{m+j})\ge\Theta_1(m+j)\ge v_p(j!)+N.
    \]
    On the other hand, we have 
    \[v_p\left(\binom{n+j}{j}b_ju^n\right)=v_p(b_j)<v_p(j!)+N\]
    by \eqref{v_p((n+j)!)=v_p(n!)+v_p(j!)_equ}, and this contradicts the equation \eqref{an_equ}.
\end{proof}

\begin{lemma}
\label{L=K((1+pT)^x^i/p)_lem}
    As a field, $\mathcal{L}=\mathcal{K}[\exp(t),\exp(xt),\cdots,\exp(x^{h-1}t)]$.
\end{lemma}
\begin{proof}
    Note that
    \[\mathcal{A}[\exp(t),\exp(xt),\cdots,\exp(x^{h-1}t)]=\mathcal{R}\]
    is the smallest subring of $\mathbb{C}_p\llbracket t\rrbracket$ that contains $\mathcal{A}$ and $\exp{(t)},\cdots,\exp(x^{h-1}t)$. Therefore, 
    \[\mathcal{R}\subset\mathcal{K}[\exp(t),\exp(xt),\cdots,\exp(x^{h-1}t)].\] Since $\mathcal{L}$ is the smallest field containing $\mathcal{R}$, we have 
    \[\mathcal{L}\subset\mathcal{K}[\exp(t),\exp(xt),\cdots,\exp(x^{h-1}t)].\]

    Now, $\exp(t),\cdots,\exp(x^{h-1}t)\in\mathcal{R}\subset\mathcal{L}$. 
    Since $\mathcal{A}\subset\mathcal{R}\subset\mathcal{L}$ and $\mathcal{K}$ is the smallest field containing $\mathcal{A}$, $\mathcal{K}\subset\mathcal{L}$.
    Thus, \[\mathcal{K}[\exp(t),\exp(xt),\cdots,\exp(x^{h-1}t)]\subset\mathcal{L}.\]
\end{proof}

\begin{proposition}
\label{Galois_ext_prop}
    The field extension $\mathcal{L}/\mathcal{K}$ is a Galois extension of degree $p^h$, and its Galois group is \[\operatorname{Gal}(\mathcal{L}/\mathcal{K})\coloneqq\left\langle\sigma_i:\mathcal{L}\rightarrow\mathcal{L}\text{ }\middle|\text{ }\sigma_i(\exp(x^jt))=\zeta_p^{i\cdot\delta_{ij}}\exp{(x^jt)},\text{ }0\le i,j\le h-1\right\rangle\cong(\mathbb{Z}/p\mathbb{Z})^h\]
    where $\zeta_p$ is a primitive $p$-th root of unity, and $\delta_{ij}$ is the Kronecker delta.
\end{proposition}
\begin{proof}
    By Kummer theory, it suffices to show that $\{\exp(pt),\exp(pxt),\cdots,\exp(px^{h-1}t)\}$ is linearly independent in the multiplicative group $\mathcal{K}^{\times}/(\mathcal{K}^{\times})^p$ as a $\mathbb{F}_p$-vector space.

    Let $a_0,\cdots,a_{h-1}\in\{0,1,\cdots,p-1\}$ and assume 
    \[\exp(put)\in(\mathcal{K}^{\times})^p\quad\text{where }u=a_0+a_1x+\cdots+a_{h-1}x^{h-1}.\]
    Then, we have $\exp(ut)\in\mathcal{K}$ because $\mathbb{C}_p\subset\mathcal{K}$, but by Lemma~\ref{(e^ut_notin_K_lemma}, this implies $v_p(u)>0$. Then, since $x$ is chosen so that $\{1,\overline{x},\cdots,\overline{x}^{h-1}\}$ is a basis of $\mathcal{O}_h/\mathfrak{m}_h$, we have $a_0=\cdots=a_{h-1}=0$.
\end{proof}

\begin{proposition}
\label{norm_and_reciprocal_prop}
    $\mathcal{R}^\times=\{F(t)\in\mathcal{R}~|~v_p(F(0))=0\}$.
\end{proposition}
\begin{proof}
    Denote by $\textbf{Nr}_{\mathcal{L}/\mathcal{K}}$ the norm map of the field extension $\mathcal{L}/\mathcal{K}$.
    Define
    \[A=\{\mathbf{a}=(a_0,\cdots,a_{h-1})~|~a_0,\cdots,a_{h-1}\in\{0,1,\cdots,p-1\}\}\]
    and set
    \[\mathbf{a}(x)=a_0+a_1x+\cdots+a_{h-1}x^{h-1}\]
    for each $\mathbf{a}=(a_0,\cdots,a_{h-1})\in A$.
    
    Let $F(t)\in\mathcal{R}$. We show the equivalence of conditions
    \[F(t)\in\mathcal{R}^\times\Leftrightarrow \textbf{Nr}_{\mathcal{L}/\mathcal{K}}(F(t))\in\mathcal{A}^{\times}\Leftrightarrow v_p(F(0))=0.\]
    
    By Theorem~\ref{A,M',R_properties_thm}(1), we can write
    \[F(t)=\sum_{\mathbf{a}\in A}e^{\mathbf{a}(x)}F_{\mathbf{a}}(t)\]
    with some $\{F_{\mathbf{a}}(t)\}_{\mathbf{a}\in A}\subset\mathcal{A}$.

    Because the degree of the minimal polynomial of $\exp{(x^it)}$ is $p$ and $[\mathcal{L}:\mathcal{K}]=p^h$, $\{\exp(\mathbf{a}(x)t)~|~\mathbf{a}\in A\}$ is a basis of the $\mathcal{K}$-linear space $\mathcal{L}$.
    Since $\textbf{Nr}_{\mathcal{L}/\mathcal{K}}(F(t))$ is the determinant of the matrix of a linear map $\mathcal{L}\rightarrow\mathcal{L};X\mapsto F(t)X$, considering the matrix with respect to the basis $\{\exp(\mathbf{a}(x)t)~|~\mathbf{a}\in A\}$, we obtain $\textbf{Nr}_{\mathcal{L}/\mathcal{K}}(F(t))\in\mathcal{A}$.

    Also, by Proposition~\ref{Galois_ext_prop}, 
    \[\prod_{\substack{\sigma\in\operatorname{Gal}(\mathcal{L}/\mathcal{K})}\\\sigma\ne 1}\sigma(F(t))\in\mathcal{R}\]
    because $\zeta_p\in\mathcal{R}$. (Note that $\mathcal{O}_{\mathbb{C}_p}\subset\mathcal{A}$ by the assumption (B) at the beginning of this subsection.)

    Assume $F(t)\in\mathcal{R}^\times$. Then, since $1/F(t)\in\mathcal{A}$, we have $\textbf{Nr}_{\mathcal{L}/\mathcal{K}}(1/F(t))\in\mathcal{A}$ too. Thus
    \[\textbf{Nr}_{\mathcal{L}/\mathcal{K}}(F(t))=\frac{1}{\textbf{Nr}_{\mathcal{L}/\mathcal{K}}(1/F(t))}\in\mathcal{A}^\times.\]

    Conversely, if $\textbf{Nr}_{\mathcal{L}/\mathcal{K}}(F(t))\in\mathcal{A}$, it follows that
    \[\frac{1}{F(t)}=\frac{1}{A(t)}\prod_{\substack{\sigma\in\operatorname{Gal}(\mathcal{L}/\mathcal{K})}\\\sigma\ne 1}\sigma(F(t))\in\mathcal{R}.\]

    Let $\pi:\mathcal{O}_{\mathbb{C}_p}\rightarrow\overline{\mathbb{F}_p}$ be the reduction map. Note that $F_\mathbf{a}(0)\in\mathcal{O}_{\mathbb{C}_p}$ by the assumption (B), so
    \begin{align*}
        \textbf{Nr}_{\mathcal{L}/\mathcal{K}}(F(t))\in\mathcal{A}^\times&\Leftrightarrow v_p\Bigl(\textbf{Nr}_{\mathcal{L}/\mathcal{K}}(F(t))\Bigm|_{t=0}\Bigr)=0\quad(\text{by Proposition~\ref{subring_prop}})\\
        &\Leftrightarrow v_p\left(\prod_{i_0=0}^{p-1}\cdots\prod_{i_{h-1}=0}^{p-1}\sum_{\mathbf{a}=(a_0,\cdots,a_{h-1})\in A}F_{\mathbf{a}}(0)\zeta_p^{i_0a_0+\cdots+i_{h-1}a_{h-1}}\right)=0\\
        &\Leftrightarrow\pi\left(\left(\sum_{\mathbf{a}\in A}F_{\mathbf{a}}(0)\right)^q\right)\in\overline{\mathbb{F}_p}^\times\quad(\text{since }\pi(\zeta_p^i)=1\text{ for any }i\in\mathbb{Z})\\
        &\Leftrightarrow\pi\left(\sum_{\mathbf{a}\in A}F_\mathbf{a}(0)\right)=\pi(F(0))\in\overline{\mathbb{F}_p}^\times\\
        &\Leftrightarrow v_p(F(0))=0.
    \end{align*}
\end{proof}

Combining Theorem~\ref{norm_and_reciprocal_prop} with the result obtained in the last subsection, we have the following corollaries. 

\begin{corollary}
\label{Reciplocal_cor}
    Let $r\in[0,1]$, and
    \[
    F(t)=\sum_{n=0}^{\infty}c_n\frac{t^n}{n!}
    \]
    be the generating function of a sequence $\{c_n\}_{n=0}^{\infty}\subset\mathcal{O}_{\mathbb{C}_p}$ with $v_p(c_0)=0$. Let $\{d_n\}_{n=0}^{\infty}$ be the sequence whose generating function is
    \[
    \frac{1}{F(t)}=\sum_{n=0}^{\infty}d_n\frac{t^n}{n!}.
    \]

    Then the following conditions are equivalent.
    \begin{enumerate}[label=(\alph*)]
        \item $\{c_n\}$ satisfies the congruence
        \[c_{n+\ell p^k}\equiv c_n\mod{p^{k+r}}\quad\text{for all }k\in\mathbb{Z}_{\ge0}\text{ and }n\ge p^k.\]
        \item $\{d_n\}$ satisfies the congruence
        \[d_{n+\ell p^k}\equiv d_n\mod{p^{k+r}}\quad\text{for all }k\in\mathbb{Z}_{\ge0}\text{ and }n\ge p^k.\]
    \end{enumerate}
\end{corollary}

\begin{remark}
    Compare Corollary~\ref{Reciplocal_cor} with Theorem~8.1 in \cite{Gessel} by I. M. Gessel. The theorem states that if integer sequences $\{c_n\}$ and $\{d_n\}$ satisfy
    \[\left(\sum_{n=0}^{\infty}c_n\frac{t^n}{n!}\right)\left(\sum_{n=0}^{\infty}d_n\frac{t^n}{n!}\right)=1,\]
    $\{d_n\}$ satisfies "Kummer's congruence" if and only if so does $\{c_n\}$, but his definition for the term \textit{Kummer's congruence} is to satisfy
    \[\sum_{i=0}^{n}(-1)^{n-i}\binom{n}{i}c_{i(p-1)+j}\equiv0\mod{p^n}\]
    for all integer $n$ and $j\ge n$.
\end{remark}

\begin{corollary}
\label{Reciplocal_la_cor}
    Let $r\in[0,1]$, $s\in\mathbb{Z}_{\ge0}$, and
    \[
    F(t)=\sum_{n=0}^{\infty}c_n\frac{t^n}{n!}
    \]
    be the generating function of a sequence $\{c_n\}_{n=0}^{\infty}\subset\mathcal{O}_{\mathbb{C}_p}$ with $v_p(c_0)=0$. Let $\{d_n\}_{n=0}^{\infty}$ be the sequence whose generating function is
    \[
    \frac{1}{F(t)}=\sum_{n=0}^{\infty}d_n\frac{t^n}{n!}.
    \]

    Suppose $p+s\ge3$. If there exists some $C\in(\frac{1}{p^s(p-1)},1]$ such that 
    \[F(t)\in\bigoplus_{\alpha^q=\alpha}e^{\alpha t}\Bigl(\mathbb{C}_p\langle t|\ge\Theta_r\rangle\cap\mathbb{C}_p\langle t|\ge Cn\rangle\Bigr),\]
    then $\{d_n\}$ satisfies
    \[d_{n+(p^h-1)p^k}\equiv d_n\mod{p^{k+r}}\quad\text{for all }k\in\mathbb{Z}_{\ge0}\text{ and }n\ge p^k,\]
    and moreover, for each $i=0,1,\cdots,p^h-2$, the map
    \[\mathbb{Z}_{\ge0}\rightarrow\mathcal{O}_{\mathbb{C}_p};n\mapsto \lim_{k\rightarrow\infty}d_{i+(p^h-1)(n+p^k)}\]
    is continuously extended to a function on $\mathbb{Z}_p$ that is locally analytic of order $s$.
\end{corollary}

\section{The case a $p$-adic measure exists}
\noindent
\label{Comparison_with_integral_section}
The purpose of this section is to explain how sequences admitting suitable $p$-adic measures can be categorized within our framework. 

\subsection{An advantage and a strong condition}
Let $\{c_n\}_{n=0}^{\infty}$ be a sequence that we are now interested in. Consider the following assumption;
\begin{assumption}
\label{measure_assumption}
    We have a $\mathcal{O}_{\mathbb{C}_p}$-valued $p$-adic measure $\mu$ on $\mathbb{Z}_p$ such that
    \begin{equation}
    \label{c_n_x^n_int_measure_equ}
        c_n=\int_{\mathbb{Z}_p}x^nd\mu(x)
    \end{equation}
    for all $n\in\mathbb{Z}_{\ge0}$. 
\end{assumption}
\begin{remark}
(1) For any complete subring $R\subset\mathbb{C}_p$, we say a $p$-adic measure $\mu$ on $\mathbb{Z}_p$ is $R$-valued if
\[\int_{\mathbb{Z}_p}\binom{x}{n}d\mu(x)\in R\]
for all $n\in\mathbb{Z}_{\ge0}$. It is not difficult to see that this is equivalent to the condition that for any compact open subset $X\subset\mathbb{Z}_p$,
\[\mu(X)=\int_Xd\mu=\int_{\mathbb{Z}_p}\mathds{1}_{X}d\mu\in R\]
where $\mathds{1}_{X}$ is a characteristic function of $X$.

(2) Since the $p$-adic measure $\mu$ is determined by the values of
\[\int_{\mathbb{Z}_p}\binom{x}{n}d\mu(x),\]
if a measure satisfying \eqref{c_n_x^n_int_measure_equ} exists, it is unique.

(3) \eqref{c_n_x^n_int_measure_equ} forces $\{c_n\}\subset\mathcal{O}_{\mathbb{C}_p}$.
\end{remark}

The existence of a measure satisfying \eqref{c_n_x^n_int_measure_equ} leads, for example, to the following advantages.

\begin{enumerate}
    \item[(i)] In Proposition~\ref{Assump_leads_la_prop} below, we show that under the Assumption~\ref{measure_assumption}, the problem of whether $\{c_n\}$ satisfies certain congruences can be reduced to determining the support of $\mu$. To investigate or control the support of $\mu$, we can use a well known one to one correspondence (Amice transform):
    \begin{align*}
    \{R\text{-valued $p$-adic measures on }\mathbb{Z}_p\}&\leftrightarrow\{A(T)=\sum_{n=0}^{\infty}b_nT^n\in R\llbracket T\rrbracket \mid \inf_{n\in\mathbb{Z}_{\ge0}}v_p(b_n)>-\infty\}\\
    \mu&\mapsto A_{\mu}(T)\coloneqq\sum_{n=0}^{\infty}\left(\int_{\mathbb{Z}_p}\binom{x}{n}d\mu(x)\right)T^n,
    \end{align*}
    where $R$ is a complete subring of $\mathbb{C}_p$. (Since $p$-adic measure $\mu$ is a bounded linear functional of $\mathscr{C}(\mathbb{Z}_p,\mathbb{C}_p)$, the valuations of the coefficients of $A_\mu$ are bounded below.) This correspondence enables us to use actions on the set of power series. For example, an action $\Phi$ defined as follows on a power series corresponds to the restriction of the support of the corresponding measure to $\mathbb{Z}_p^\times$:
    \[\Phi(A(T))=A(T)-\frac{1}{p}\sum_{\zeta^p=1}A((1+T)\zeta-1).\]
    That is, for a $p$-adic measure $\mu$, let $\nu$ be the restriction of $\mu$ to $\mathbb{Z}_p^\times$, defined by
    \[\int_{\mathbb{Z}_p}fd\nu=\int_{\mathbb{Z}_p^\times}fd\mu=\int_{\mathbb{Z}_p}f(x)\mathds{1}_{\mathbb{Z}_p^\times}(x)d\mu(x).\]
    Then, we have $\Phi(A_\mu(T))=A_\nu(T)$.
    
    \item[(ii)] The existence of a measure justifies a kind of ``umbral calculus". For example, we have
    \begin{equation*}
    \label{umbral_calculus_equ}
        \sum_{k=0}^{n}\binom{n}{k}(-1)^{n-k}c_{n+k(p-1)}=\int_{\mathbb{Z}_p}x^n(x^{p-1}-1)^nd\mu(x),
    \end{equation*}
    for any $n\ge0$. Since $\inf_{x\in\mathbb{Z}_p}v_p(x^n(x^{p-1}-1)^n)\ge n$, We know that the valuation of the left-hand side of \eqref{umbral_calculus_equ} is also equal to or greater than $n$, and we have
    \[\sum_{k=0}^{n}\binom{n}{k}(-1)^{n-k}c_{n+k(p-1)}\equiv0\mod{p^n}.\]
    In this way, under the Assumption~\ref{measure_assumption}, we obtain the abstract Kummer congruence: For any $N\ge0$ and any $\{b_n\}\subset\mathbb{C}_p$ such that $b_n\ne0$ for finitely many $n$, we have
    \[\left(\forall y\in\mathbb{Z}_p,~\sum_{n=0}^{\infty}b_ny^n\equiv0\mod{p^N}\right)\Rightarrow\left(\sum_{n=0}^{\infty}b_nc_n\equiv0\mod{p^N}\right).\]
    Setting $b_{n+(p-1)p^k}=1,~b_n=-1,~b_{n'}=0~(\text{for the other }n')$ yields 
    \[c_{n+(p-1)p^k}\equiv c_n\mod{p^{k+1}}.\]
    So the abstract Kummer congruence is a stronger condition than the congruence we consider in this paper. Katz proved in \cite[Proposition~4.0.6]{Katz} that the abstract Kummer congruence is also a sufficient condition for the existence of a measure $\mu$ satisfying \eqref{c_n_x^n_int_measure_equ}. 
\end{enumerate}

In \cite{Young}, Young exploited these advantages (i), (ii) to prove his main result concerning the congruences of sequences (Theorem 1.1 and Theorem 1.2 in \cite{Young}).

As is immediate from (ii) above, Assumption~\ref{measure_assumption} does not hold for any sequence $\{c_n\}\subset\mathcal{O}_{\mathbb{C}_p}$. From one point of view, this is due to the abstract Kummer congruence, but from another point of view, this is because setting the value of $\mu$ to satisfy \eqref{c_n_x^n_int_measure_equ} may result in a failure to satisfy the boundedness condition. That is, \eqref{c_n_x^n_int_measure_equ} forces
\[b_n\coloneqq\int_{\mathbb{Z}_p}\binom{x}{n}d\mu=\frac{1}{n!}\sum_{k=0}^{n}(-1)^{n-k}\genfrac{[}{]}{0pt}{}{n}{k}c_k\]
where $\genfrac{[}{]}{0pt}{}{n}{k}$ are the (unsigned) first kind of Stirling numbers, but for $\mu$ to be a measure, $v_p(b_n)$ should be bounded below since $\{\binom{x}{n}\}$ is a Banach basis of the set of continuous functions $\mathscr{C}(\mathbb{Z}_p,\mathbb{C}_p)$. (Recall that $p$-adic measures are defined as linear bounded functionals of $\mathscr{C}(\mathbb{Z}_p,\mathbb{C}_p)$.) On the other hand, in the worst case, $v_p(b_n)=-v_p(n!)+\lfloor\log_pn\rfloor\sim -\frac{n}{p-1}+O(\log_pn)$, even if $\{c_n\}\subset\mathcal{O}_{\mathbb{C}_p}$ satisfies congruence
\[c_{n+p^{k}}\equiv c_n\mod{p^{k+r}}\quad\text{for all }n,k\ge0\] 
for any arbitrarily large but fixed $r$. (See Appendix~A for the precise statement and a proof.)

If one allows the use of locally analytic distributions or some extension thereof, this problem may be addressed to some extent. That is, instead of restricting the class of integrable functions, we relax the condition of boundedness for measures. For example, we may define
\[\int_{\mathbb{Z}_p}fd\mu=\sum_{n=0}^{\infty}a_n(f)b_n\]
for any function $f$ in the restricted class
\[\mathscr{C}(\{c_n\})\coloneqq\left\{f\in\mathscr{C}(\mathbb{Z}_p,\mathbb{C}_p)~\middle|~f(x)=\sum_{n=0}^{\infty}a_n(f)\binom{x}{n}\text{ with }\lim_{n\rightarrow\infty}a_n(f)b_n=0\right\}.\]
However, in the worst case $v_p(b_n)\sim -\frac{n}{p-1}$, only the functions that are globally analytic on $\mathbb{Z}_p$ belong to $\mathscr{C}(\{c_n\})$ by Theorem~\ref{smooth_coeffi_thm}(2). In particular, even the characteristic functions $\mathds{1}_{a+p\mathbb{Z}_p}$ fail to belong to $\mathscr{C}(\{c_n\})$.

\subsection{An interpretation of existence of measure in our framework}
\label{Interpretation_of_integral_subsection}
In this subsection, we clarify what the cases where Assumption~\ref{measure_assumption} holds correspond to within our framework. This reveals how the results presented in Sections~3 and~4 extend the methods of $p$-adic integration. More precisely, we prove the following in this section.

\begin{proposition}
\label{Assump_leads_la_prop}
    Let 
    \[F(t)=\sum_{n=0}^{\infty}c_n\frac{t^n}{n!}\]
    be the generating function of a sequence $\{c_n\}_{n=0}^{\infty}\subset\mathcal{O}_{\mathbb{C}_p}$, and suppose a $p$-adic measure $\mu$ satisfies \eqref{c_n_x^n_int_measure_equ}. Then, 
    
    (1) we have
    \[F(t)\in\sum_{\alpha^{p}=\alpha}e^{\alpha t}\mathbb{C}_p\langle t|\ge n\rangle.\]
    
    (2) the support of $\mu$ is contained in $\mathbb{Z}_p^\times$ if and only if
    \[F(t)\in\sum_{\alpha^{p-1}=1}e^{\alpha t}\mathbb{C}_p\langle t|\ge n\rangle.\]
    
    (3) for each $(p-1)$-th root of unity $\alpha$, the support of $\mu$ is contained in $\alpha+p\mathbb{Z}_p^\times$ if and only if
    \[F(t)\in e^{\alpha t}\mathbb{C}_p\langle t|\ge n\rangle.\]
\end{proposition}
\begin{remark}
\label{Assump_leads_cong_rmk}
Recall that
\[\mathbb{C}_p\langle t|\ge n\rangle\subset\mathbb{C}_p\langle t|\ge\Theta_1\rangle\]
We already have \eqref{R_r(O_C_p)_decomposition} and Theorem~\ref{application_thm}. So, if Assumption~\ref{measure_assumption} holds, $\{c_n\}$ automatically satisfies the congruence
\[c_{n+(p-1)p^k}\equiv c_n\mod{p^{k+1}}\]
for all $k\in\mathbb{Z}_{\ge0}$ and $n\ge k+1$, and moreover, if the support of $\mu$ is contained in $\mathbb{Z}_p^\times$ (resp. $\alpha+p\mathbb{Z}_p$), $\{c_n\}$ satisfies
\begin{align*}
    c_{n+(p-1)p^k}&\equiv c_n\mod{p^{k+1}}\quad\text{for all }k,n\in\mathbb{Z}_{\ge0}\\
    (\text{resp. }c_{n+p^k}&\equiv \alpha c_n\mod{p^{k+1}}\quad\text{for all }k,n\in\mathbb{Z}_{\ge0})
\end{align*}
by Proposition~\ref{Basic_prop}. Furthermore, if $p+s\ge3$, since
\[\mathbb{C}_p\langle t|\ge n\rangle\subset\mathbb{C}_p\langle t|v_p(\lfloor n/p^s\rfloor!)\rangle,\]
the $p$-adic interpolation of the map
\[\mathbb{Z}_{\ge0}\rightarrow\mathbb{C}_p;n\rightarrow\lim_{k\rightarrow\infty}c_{i+(p-1)(n+p^k)}\]
is locally analytic of order $0$, i.e., globally analytic on $\mathbb{Z}_p$ for each $i=0,1,\cdots,p-2$ by \eqref{R_r,s^la(O_Cp)_decomposition}.

Therefore, for example, if we set
\[c_n=\begin{cases}
    p & (\text{if }n\equiv0\mod{p}),\\
    0 & (\text{if }n\not\equiv0\mod{p})
\end{cases}\]
with odd prime $p$, it satisfies
\[c_{n+(p-1)p^k}\equiv c_n\mod{p^{k+1}}\]
for all $k,n\in\mathbb{Z}_{\ge0}$, but Assumption~\ref{measure_assumption} does not hold because the map
\[\mathbb{Z}_p\rightarrow\mathbb{C}_p;n\mapsto c_{(p-1)n}\]
is locally analytic of order $1$ but not of order $0$.
\end{remark}

To prove Proposition~\ref{Assump_leads_la_prop}, we prepare two facts.
\begin{lemma}
\label{measure_lem}
    Let $\mu$ be a $p$-adic measure.
    For $t\in\mathbb{Z}_p$,
    \[A_\mu(e^{pt}-1)=\int_{\mathbb{Z}_p}e^{ptx}d\mu(x).\]
\end{lemma}
\begin{proof}
    By definition, 
    \[A_\mu(e^{pt}-1)=\sum_{n=0}^{\infty}\left(\int_{\mathbb{Z}_p}\binom{x}{n}d\mu(x)\right)(e^{pt}-1)^n.\]
    Since the valuations
    \[v_p\left(\int_{\mathbb{Z}_p}\binom{x}{n}d\mu(x)\right)\]
    are bounded below and $e^{pt}-1\in p\mathbb{Z}_p$, this infinite sum converges.

    For all $x\in\mathbb{Z}_p$ and $n\in\mathbb{Z}_{\ge0}$, $v_p(\binom{x}{n})\ge0$. Thus, for fixed $t\in\mathbb{Z}_p$, the infinite sum
    \[\sum_{n=0}^{\infty}\binom{x}{n}(e^{pt}-1)^n\]
    uniformly converges with respect to $x\in\mathbb{Z}_p$. Hence, this is a continuous function on $\mathbb{Z}_p$, and we have
    \[\int_{\mathbb{Z}_p}\sum_{n=0}^{\infty}\binom{x}{n}(e^{pt}-1)^nd\mu(x)=\sum_{n=0}^{\infty}\left(\int_{\mathbb{Z}_p}\binom{x}{n}d\mu(x)\right)(e^{pt}-1)^n.\]

    On the other hand, by the generalized binomial theorem, 
    \[\sum_{n=0}^{\infty}\binom{x}{n}(e^{pt}-1)^n=(1+e^{pt}-1)^{x}=e^{ptx}.\]
    Therefore, we have
    \[A_\mu(e^{pt}-1)=\int_{\mathbb{Z}_p}e^{ptx}d\mu(x).\]
\end{proof}

\begin{lemma}
\label{poly_approxi_lem}
    Let $c_0,\cdots,c_{p-1}\in\mathbb{Z}_p$ and set 
    \[f:\mathbb{Z}_p\rightarrow\mathbb{Z}_p;x\mapsto c_i\quad(\text{if }x\in i+p\mathbb{Z}_p).\]
    Then, there exists a family of polynomials $\{f_m(x)\}\in\mathbb{Z}_p[x]$ that uniformly converges to $f(x)$ on $\mathbb{Z}_p$.
\end{lemma}
\begin{proof}
    Let
    \[f(x)=\sum_{n=0}^{\infty}a_n\binom{x}{n}\]
    be the Mahler expansion of $f$. Then, we have $v_p(a_n)\rightarrow\infty$.
    Since $v_p(\binom{x}{n})\ge0$ for all $x\in\mathbb{Z}_p$, this convergence is uniform on $\mathbb{Z}_p$.

    Since $n!\binom{x}{n}\in\mathbb{Z}_p[x]$, it is enough to show that $v_p(a_n)\ge v_p(n!)$. Let $\zeta_p$ be a primitive $p$-th root of unity. By Theorem~\ref{Mahler_expansion_thm}, 
    \begin{align*}
        a_n&=\sum_{m=0}^{n}\binom{n}{m}(-1)^{n-m}f(m)\\
        &=\sum_{i=0}^{p-1}c_i\sum_{m=0}^{n}\binom{n}{m}(-1)^{n-m}\frac{1}{p}\sum_{k=0}^{p-1}\zeta_p^{(m-i)k}\\
        &=\sum_{i=0}^{p-1}\frac{c_i}{p}\sum_{k=0}^{p-1}\zeta_p^{-ik}(\zeta_p^k-1)^n.
    \end{align*}
    Now, $v_p(\zeta_p^k-1)\ge\frac{n}{p-1}$ for all $n$, and $v_p(a_n)\in\mathbb{Z}_p$. Thus,
    \[v_p(a_n)\ge\left\lceil\frac{n}{p-1}\right\rceil-1.\]
    On the other hand, since $v_p(n!)\in\mathbb{Z}_{\ge0}$ and $v_p(n!)\le\frac{n-1}{p-1}$ for all $n\ge1$, we have
    \[v_p(n!)\le\left\lfloor\frac{n-1}{p-1}\right\rfloor\]
    for all $n\ge1$.
    Hence, we obtain $v_p(a_n)\ge v_p(n!)$ and $v_p(a_0)=v_p(c_0)\ge0=v_p(0!)$.
\end{proof}

\begin{proof}[proof of Proposition~\ref{Assump_leads_la_prop}]
    Take an $\mathcal{O}_{\mathbb{C}_p}$-valued measure $\mu$ such that 
    \[\int_{\mathbb{Z}_p}x^nd\mu(x)=c_n\]
    for all $n\in\mathbb{Z}_{\ge0}$. 
    
    We first show (1) and ``only if" part of (2) and (3).
    Let $S\subset\{\alpha\in\mathbb{Z}_p~|~\alpha^p=\alpha\}$ and assume the support of $\mu$ is contained in $\bigcup_{\alpha\in S}\alpha+p\mathbb{Z}_p$.
    
    We have
    \begin{align*}
        \sum_{n=0}^{\infty}c_n\frac{t^n}{n!}&=\sum_{n=0}^{\infty}\int_{\mathbb{Z}_p}x^n\frac{t^n}{n!}d\mu(x)\\
        &=\int_{\mathbb{Z}_p}e^{tx}d\mu(x)\\
        &=\sum_{\alpha\in S}\int_{\alpha+p\mathbb{Z}_p}e^{tx}d\mu(x)\\
        &=\sum_{\alpha\in S}\int_{\mathbb{Z}_p}e^{t(\alpha+py)}d\nu_\alpha(y)\\
        &=\sum_{\alpha\in S}e^{\alpha t}A_{\nu_\alpha}(e^{pt}-1)\quad\quad(\text{by Lemma~\ref{measure_lem}})
    \end{align*}
    where $\nu_\alpha$ are the measures obtained from $\mu$ by the change of variables $x=\alpha+py$, and $A_{\nu_\alpha}$ are their corresponding power series by the Amice transform. 
    
    Therefore, the problem is reduced to showing the following:
    \[\text{For any }A(T)\in\mathcal{O}_{\mathbb{C}_p}\llbracket T\rrbracket,\quad A(e^{pt}-1)\in\mathbb{C}_p\langle t|\ge n\rangle.\]
    This is just a calculation. Let $A(T)=\sum_{n=0}^{\infty}a_nT^n\in\mathcal{O}_{\mathbb{C}_p}\llbracket T\rrbracket$. Then,
    \begin{align*}
        A(e^{pt}-1)&=\sum_{n=0}^{\infty}a_n\left(\sum_{k=1}^{\infty}\frac{(pt)^k}{k!}\right)^n\\
        &=a_0+\sum_{n=1}^{\infty}\sum_{m=1}^{\infty}a_n\sum_{\substack{k_1,\cdots,k_n\ge1\\k_1+\cdots+k_n=m}}\frac{(pt)^m}{k_1!\cdots k_n!}\\
        &=a_0+\sum_{m=1}^{\infty}\frac{p^m}{m!}\left(\sum_{n=1}^{m}a_n\sum_{\substack{k_1,\cdots,k_n\ge1\\k_1+\cdots+k_n=m}}\binom{m}{k_1}\binom{m-k_1}{k_2}\cdots\binom{m-k_1-\cdots-k_{n-1}}{k_n}\right)t^m\\
        &\in\mathbb{C}_p\langle t|\ge n\rangle.
    \end{align*}

    Next, we show ``if" part of (2) and (3).
    Assume 
    \[F(t)\in\sum_{\alpha^{p-1}=1}e^{\alpha t}\mathbb{C}_p\langle t|\ge n\rangle\quad(\text{resp. }F(t)\in e^{\alpha t}\mathbb{C}_p\langle t|\ge n\rangle).\]
    Then, by Proposition~\ref{Basic_prop} and $\mathbb{C}_p\langle t|\ge n\rangle\subset\mathbb{C}_p\langle t|\ge\Theta_1\rangle$, congruence 
    \[c_{n+(p-1)p^k}\equiv c_n\mod{p^{k+1}}\quad(\text{resp. }c_{n+p^k}\equiv \alpha c_n\mod{p^{k+1}})\]
    holds for all $k,n\in\mathbb{Z}_{\ge0}$. Thus,
    \begin{equation}
    \label{int_in_p^k+1_O_Cp_equ}
        \int_{\mathbb{Z}_p}x^n(x^{(p-1)p^k}-1)d\mu(x)\in p^{k+1}\mathcal{O}_{\mathbb{C}_p}\quad\left(\text{resp. }\int_{\mathbb{Z}_p}x^n(x^{p^k}-\alpha)d\mu(x)\in p^{k+1}\mathcal{O}_{\mathbb{C}_p}\right).
    \end{equation}
    Note that
    \[
    \lim_{k\rightarrow\infty}(x^{(p-1)p^k}-1)=\begin{cases}
        0 & (\text{if }x\in\mathbb{Z}_p^\times), \\
        -1 & (\text{if }x\in p\mathbb{Z}_p),
    \end{cases}\qquad
    \left(\text{resp. }\lim_{k\rightarrow\infty}(x^{p^k}-\alpha)=\beta_x-\alpha\right)
    \]
    where $\beta_x$ is the unique element satisfying $\beta_x^p=\beta_x$ and $\beta_x\equiv x\mod{p}$.
    Let 
    \[
    f(x)=\begin{cases}
        0 & (\text{if }x\in\mathbb{Z}_p^\times), \\
        -1 & (\text{if }x\in p\mathbb{Z}_p),
    \end{cases}\quad\left(\text{resp. }
    g(x)=\begin{cases}
        0 & (\text{if }x\in\alpha+p\mathbb{Z}_p), \\
        \frac{1}{\beta-\alpha} & (\text{if }x\in\beta+p\mathbb{Z}_p\text{ with }\beta\ne\alpha,\beta^{p}=\beta).
    \end{cases}\right)
    \]
    Then, by Lemma~\ref{poly_approxi_lem} below, there exists a family of polynomials $\{f_m(x)\},\{g_m(x)\}\subset\mathbb{Z}_p[x]$ that uniformly converge to $f(x),g(x)$ respectively on $\mathbb{Z}_p$.

    By \eqref{int_in_p^k+1_O_Cp_equ}, we have
    \[\int_{\mathbb{Z}_p}x^nf_m(x)(x^{(p-1)p^k}-1)d\mu(x)\in p^{k+1}\mathcal{O}_{\mathbb{C}_p}\quad\left(\text{resp. }\int_{\mathbb{Z}_p}x^ng_m(x)(x^{p^k}-\alpha)d\mu(x)\in p^{k+1}\mathcal{O}_{\mathbb{C}_p}\right)\]
    for all $n,k,m\in\mathbb{Z}_{\ge0}$. Taking a limit of $m\rightarrow\infty$, we obtain
    \[\int_{\mathbb{Z}_p}x^nf(x)(x^{(p-1)p^k}-1)d\mu(x)\in p^{k+1}\mathcal{O}_{\mathbb{C}_p}\quad\left(\text{resp. }\int_{\mathbb{Z}_p}x^ng(x)(x^{p^k}-\alpha)d\mu(x)\in p^{k+1}\mathcal{O}_{\mathbb{C}_p}\right).\]
    Again taking a limit of $k\rightarrow\infty$,
    \[\int_{\mathbb{Z}_p}x^n\mathds{1}_{p\mathbb{Z}_p}(x)d\mu(x)=0\quad\left(\text{resp. }\int_{\mathbb{Z}_p}x^n\mathds{1}_{\mathbb{Z}_p\setminus(\alpha+p\mathbb{Z}_p)}(x)d\mu(x)=0\right).\]
    It follows that 
    \[\int_{p\mathbb{Z}_p}\binom{x}{n}d\mu(x)=0\quad\left(\text{resp. }\int_{\mathbb{Z}_p\setminus(\alpha+p\mathbb{Z}_p)}\binom{x}{n}d\mu(x)=0\right),\]
    and since $\{\binom{x}{n}\}$ is a Banach basis of $\mathscr{C}(\mathbb{Z}_p,\mathbb{C}_p)$, this means that the support of $\mu$ is contained in $\mathbb{Z}_p^\times$ (resp. $\alpha+p\mathbb{Z}_p$).
\end{proof}

\begin{remark}
    Note that even if the generating function $F(t)$ of a sequence $\{c_n\}$ satisfies 
    \[F(t)\in\sum_{\alpha^p=\alpha}e^{\alpha t}\mathbb{C}_p\langle t|\ge n\rangle,\]
    there does not necessarily exist a $p$-adic measure $\mu$ such that 
    \[c_n=\int_{\mathbb{Z}_p}x^nd\mu(x).\]
    For example, consider $F(t)=e^tt\in e^t\mathbb{C}_p\langle t|\ge n\rangle$. If such $\mu$ exists, we should have
    \[F(t)=A_\mu(e^t-1)\]
    where $A_\mu$ is the corresponding formal power series by the Amice transform. However, the valuation of the coefficients of
    \[A_\mu(T)=F(\log(1+T))=(1+T)\sum_{n=1}^{\infty}\frac{(-1)^{n-1}}{n}T^n\]
    is not bounded below, and this is a contradiction.
\end{remark}

Recall that we have a canonical projection
\begin{align*}
    \mathcal{M}_r(\mathcal{O}_{\mathbb{C}_p})=\left(\bigoplus_{\alpha^{\ell}=1}e^{\alpha t}\mathbb{C}_p\langle t|\ge\Theta_r\rangle\right)\oplus\Bigl(\mathbb{C}_p\langle t|0\rangle\cap\mathbb{C}_p\langle t|\ge0\rangle\Bigr)\label{R_r(O_C_p)_decomposition}\rightarrow\bigoplus_{\alpha^{\ell}=1}e^{\alpha t}\mathbb{C}_p\langle t|\ge\Theta_r\rangle=\mathcal{M}'_r(\mathcal{O}_{\mathbb{C}_p})
\end{align*}

In the case that $\{c_n\}_{n=0}^{\infty}$ is given as 
\[c_n=\int_{\mathbb{Z}_p}x^nd\mu(x)\]
for some $\mathcal{O}_{\mathbb{C}_p}$-valued measure $\mu$, by the proof of Proposition~\ref{Assump_leads_la_prop}, we have
\[\sum_{n=0}^{\infty}c_n\frac{t^n}{n!}=\sum_{\alpha^p=\alpha}\int_{\alpha+p\mathbb{Z}_p}e^{tx}d\mu(x)=\sum_{\alpha^p=\alpha}e^{\alpha t}A_{\nu_\alpha}(e^{pt}-1)\]
where $\nu_\alpha$ are the measures obtained from $\mu$ by change of variables $x=\alpha+py$, and $A_{\nu_\alpha}$ are their corresponding power series by the Amice transform. Since 
\[A_{\nu_\alpha}(e^{pt}-1)\in\mathbb{C}_p\langle t|\ge\Theta_1\rangle\]
for each $\alpha$, the canonical projection $\mathcal{M}_r(\mathcal{O}_{\mathbb{C}_p})\rightarrow\mathcal{M}'_r(\mathcal{O}_{\mathbb{C}_p})$ maps
\[\sum_{n=0}^{\infty}c_n\frac{t^n}{n!}\mapsto\sum_{\alpha^{p-1}=1}e^{\alpha t}A_{\nu_\alpha}(e^{pt}-1)=\sum_{\alpha^{p-1}=1}\int_{\alpha+p\mathbb{Z}_p}e^{xt}d\mu(x)=\sum_{n=0}^{\infty}\left(\int_{\mathbb{Z}_p^\times}x^nd\mu(x)\right)\frac{t^n}{n!}.\]
That is, as sequences, the canonical projection $M_r(\mathcal{O}_{\mathbb{C}_p})\rightarrow M'_r(\mathcal{O}_{\mathbb{C}_p})$ maps 
\[\left\{\int_{\mathbb{Z}_p}x^nd\mu(x)\right\}_{n=0}^{\infty}\mapsto\left\{\int_{\mathbb{Z}_p^\times}x^nd\mu(x)\right\}_{n=0}^{\infty}\]
for each $\mathcal{O}_{\mathbb{C}_p}$-valued measure $\mu$.

\begin{proposition}
\label{can_projedtion_Euler_factor_prop}
    Let $p$ be an odd prime. Let $c$ be an integer such that $c\ne1$ and $c\not\equiv0\mod{p}$. Let $\{B_n\}_{n=0}^{\infty}$ be the Bernoulli numbers with $B_1=-\frac{1}{2}$. Set
    \[c_n=(1-c^{-(n+1)})\frac{B_{n+1}}{n+1}.\]
    Then, $\{c_n\}_{n=0}^{\infty}\in M_1(\mathbb{Z}_p)$, and the canonical projection $M_1(\mathbb{Z}_p)\rightarrow M'_1(\mathbb{Z}_p)$ maps
    \[\{c_n\}\mapsto\left\{(1-c^{-(n+1)})(1-p^n)\frac{B_{n+1}}{n+1}\right\}.\]
\end{proposition}
That is, the canonical projection $M_r(R)\rightarrow M'_r(R)$ can be viewed as a generalization of multiplying Euler factors in a sense.
\begin{proof}
    Let $\mu_c$ be a $p$-adic measure defined by 
    \[\mu_c(a+p^N\mathbb{Z}_p)=\frac{1}{c}\left\lfloor\frac{ca}{p^N}\right\rfloor\frac{c^{-1}-1}{2}\]
    for each open ball $a+p^N\mathbb{Z}_p$ with $a,N\in\mathbb{Z}_{\ge0}$. (This is certainly an $\mathcal{O}_{\mathbb{C}_p}$-valued measure by the proposition on page~36 of \cite{Koblitz}.)

    Then, we have
    \[\int_{\mathbb{Z}_p^\times}x^{n}d\mu_c(x)=(1-c^{-(n+1)})(1-p^n)\frac{B_{n+1}}{n+1}\]
    for all $n\ge0$. (See the equation at the top of page~44 \cite{Koblitz}.)

    It is not difficult to see $\mu_c(p(a+p^N\mathbb{Z}_p))=\mu_c(a+p^N\mathbb{Z}_p)$. Thus, we have
    \begin{align*}
        \int_{p\mathbb{Z}_p}f(x)d\mu_c(x)&=\lim_{N\rightarrow\infty}\sum_{a=0}^{p^N-1}f(a)\mathds{1}_{p\mathbb{Z}_p}(a)\mu_c(a+p^N\mathbb{Z}_p)\\
        &=\lim_{N\rightarrow\infty}\sum_{a=0}^{p^{N-1}-1}f(pa)\mu_c(p(a+p^N\mathbb{Z}_p))\\
        &=\int_{\mathbb{Z}_p}f(px)d\mu_c(x)
    \end{align*}
    for all $f\in\mathscr{C}(\mathbb{Z}_p,\mathbb{C}_p)$. Therefore, we have
    \[\int_{\mathbb{Z}_p^{\times}}x^nd\mu_c(x)=\int_{\mathbb{Z}_p}x^nd\mu_c(x)-\int_{p\mathbb{Z}_p}x^nd\mu_c(x)=(1-p^n)\int_{\mathbb{Z}_p}x^nd\mu_c(x).\]
    Hence, for $n\ge1$, we have
    \[\int_{\mathbb{Z}_p}x^nd\mu(x)=(1-c^{-(n+1)})\frac{B_{n+1}}{n+1}\]
    and this also holds for $n=0$ because $\mu_c(\mathbb{Z}_p)=\frac{c^{-1}-1}{2}=(1-c^{-1})B_1$. This finishes the proof.
\end{proof}

\section{Applications}
\noindent
\label{Application_section}
In this section, we apply the above results to sequences arising in combinatorics and number theory. Some of the following applications are not particularly novel and might be established by other methods. Rather, the point is that congruences (and properties) previously handled for each sequence separately, as well as cases excluded in earlier work, can all be derived in a uniform manner. Although we will not provide detailed proofs here, arguments of the same type as the following examples hold for many sequences, such as tangent numbers, Fubini numbers, and derangement numbers.

\subsection{Genocchi numbers}
\label{Genocchi_subsection}
In this subsection, using the Genocchi numbers as an example, we present two aspects of our framework in terms of its applications:

\begin{enumerate}[label=(\roman*)]
    \item Our framework can address the case where Assumption~\ref{measure_assumption} does not hold. (See Remark~\ref{Genocchi_no_measure_rmk} below.)

    \item Our framework can also be applied to prove that ''a congruence does not hold" without explicitly computing the values of the sequence.
\end{enumerate}

Genocchi numbers $\{G_n\}_{n=0}^{\infty}$ are defined by 
\[\frac{2t}{e^t+1}=\sum_{n=0}^{\infty}G_n\frac{t^n}{n!}.\]

For $p$-adic congruence of Genocchi numbers, M. Cenkci showed the following congruence in \cite{Cenkci} by a method using $p$-adic integration presented in \cite{Young}.
\begin{theorem}[A part of Theorem 3.1 in \cite{Cenkci}]
\label{Cenkci's_thm}
    Let $p$ be an odd prime.
    If $c\equiv0\mod{(p-1)p^k}$ with $k\ge0$ and $p-1$ does not divide $n>0$, then,
    \[\Delta_c^N\left\{(1-p^{n-1})\frac{G_n}{n}\right\}\equiv0\mod{p^{N(k+1)}}\]
    for all $N\ge0$, where $\Delta_c$ is a difference operator such that $\Delta_c(c_n)=c_{n+c}-c_n$
\end{theorem}

Cenkci used the $p$-adic measure $\mu$ such that
\begin{equation}
\label{Genocchi_measure_equ}
    (1-p^n)\frac{G_{n+1}}{n+1}=\int_{\mathbb{Z}_p^\times}x^nd\mu(x),
\end{equation} 
and to do so, he restricted the case to odd $p$.

Here, we dispose of the case $p=2$.
\begin{corollary}
\label{Genocchi_cor}
    (1) The congruence
    \[G_{n+2^k}\equiv G_n\mod{2^k}\]
    holds for all $k\ge0$ and $n\ge n_k$ where $n_k$ is the minimum positive integer such that $v_2(n_k!)\ge k$.

    (2) The following statement does not hold:
    \begin{equation}
    \label{false_congruence}
        \frac{G_{n+2^k}}{n+2^k}\equiv\frac{G_n}{n}\mod{2^k}\quad\text{for all }k\ge0\text{ and }n\ge2^k.
    \end{equation}

    Consequently, the following statement does not hold either.
    \[\Delta_{2^k}\left\{(1-2^{n-1})\frac{G_n}{n}\right\}\equiv0\mod{2^k}\quad\text{for all }k\ge0\text{ and }n\ge2^k.\]
\end{corollary}
\begin{proof}
    (1) First, the generating function of Genocchi numbers can be rewritten as
    \[\frac{2t}{e^t+1}=\frac{2t(e^t-1)}{e^{2t}-1}=\frac{2t}{e^{2t}-1}e^t-\frac{2t}{e^{2t}-1}e^{0\cdot t}.\]
    We have
    \[\frac{e^{2t}-1}{2t}=\sum_{n=0}^{\infty}\frac{2^n}{n+1}\frac{t^n}{n!}\in\mathbb{C}_2\langle t|\ge 0\rangle\]
    and its constant term is $1$. Thus, we obtain $2t/(e^{2t}-1)\in\mathbb{C}_2\langle t|\ge v_2(n!)\rangle$ by Proposition~\ref{subring_prop} and Lemma~\ref{concrete_Theta_satisfying_condition_lemma}.

    By Theorem~\ref{application_thm} and the fact $\mathbb{C}_2\langle t|\ge v_2(n!)\rangle\subset\mathbb{C}_2\langle t|\ge\Theta_0\rangle$, 
    \[G_{n+2^k}\equiv G_n\mod{2^k}\]
    for all $n,k\in\mathbb{Z}_{\ge0}$ such that $v_2(n!)\ge k$.

    (2) We have
    \[\sum_{n=0}^{\infty}\frac{G_{n+1}}{n+1}\frac{t^n}{n!}=\frac{2}{e^t+1}=\frac{2t}{e^{2t}-1}\frac{e^t-1}{t}.\]

    Assume that the congruence \eqref{false_congruence} holds. 

    Then, 
    \[\frac{G_n}{n}\equiv\begin{cases}
        \frac{G_3}{3} & (\text{if }n\ge3\text{ and }n\text{ is odd}) \\
        \frac{G_2}{2} & (\text{if }n\ge2\text{ and }n\text{ is even}) 
    \end{cases}\quad\mod2\]
    Thus, the values $v_2(G_n/n)$ are bounded below. Hence, there are some $M\in\mathbb{R}$ such that 
    \[\frac{2}{e^t+1}\cdot 2^M\in\mathcal{M}_{2,1,0}(\mathcal{O}_{\mathbb{C}_2}).\]
    (For the notation $\mathcal{M}_{2,1,0}(\mathcal{O}_{\mathbb{C}_2})$, see \eqref{R_r(R)_def}.)
    
    As we see in (1), $2t/(e^{2t}-1)$ is a unit in the ring $\mathbb{C}_2\langle t|\ge v_2(n!)\rangle$, and thus it is also a unit in the ring $\mathcal{M}_{2,1,0}(\mathcal{O}_{\mathbb{C}_2})$ because $\mathbb{C}_2\langle t|\ge v_2(n!)\rangle\subset\mathbb{C}_2\langle t|\ge\Theta_0\rangle\subset\mathcal{M}_{2,1,0}(\mathcal{O}_{\mathbb{C}_2})$.

    Hence, we have
    \[\frac{e^t-1}{t}\cdot2^M=\frac{2}{e^t+1}\cdot2^M\cdot\frac{e^{2t}-1}{2t}\in\mathcal{M}_{2,1,0}(\mathcal{O}_{\mathbb{C}_2})\]
    
    However, 
    \[\frac{e^t-1}{t}\cdot 2^M=\sum_{n=0}^{\infty}\frac{2^M}{n+1}\frac{t^n}{n!}\notin\mathcal{M}_{2,1,0}(\mathcal{O}_{\mathbb{C}_2})\]
    because the values $v_2(2^M/(n+1))$ are not bounded below. This is a contradiction.
\end{proof}

\begin{remark}
\label{Genocchi_no_measure_rmk}
    Corollary~\ref{Genocchi_cor}(2) implies that there are no $p$-adic measure $\mu$ satisfying \eqref{Genocchi_measure_equ} if $p=2$. Indeed, if such a $\mu$ exists, we should have
    \[\Delta_{2^k}\left\{(1-2^n)\frac{G_{n+1}}{n+1}\right\}\equiv0\mod{2^{k+1}}\]
    for all $k\ge0$ and $n\ge k+1$ by Remark~\ref{Assump_leads_cong_rmk}.
\end{remark}

\subsection{Generalized Euler numbers}
The sequence of generalized Euler numbers $\{E_n^{(d)}\}$ is defined for each integer $d\ge2$ by
\[\left(\sum_{n=0}^{\infty}\frac{t^{dn}}{(dn)!}\right)^{-1}=\left(\sum_{k=0}^{d-1}\exp{(\zeta_d^kt)}\right)^{-1}=\sum_{n=0}^{\infty}E_n^{(d)}\frac{t^n}{n!}\]
where $\zeta_d$ is a $d$-th primitive root of unity. 

This was first introduced in \cite{Leeming_and_MacLeod} by D. J. Leeming and R. A. MacLeod. Recently, a combinatorial interpretation for generalized Euler numbers was given by Sagan in \cite{Sagan}, and for $d=4$, relations to Bernoulli numbers were shown in \cite{Nishibuchi}.

Leeming and MacLeod also investigated some congruences modulo a prime power in \cite{Leeming_and_MacLeod}, and subsequent studies on generalized Euler numbers, such as \cite{Gessel, Sagan}, have also provided results concerning their congruences. Now, we obtain other congruences and results on the smoothness of $p$-adic interpolation of generalized Euler numbers.
\begin{corollary}
\label{generalized_Euler_cor}
    (1) Let $p$ be a prime, and let $d\ge2$ be an integer that is coprime to $p$. Then, there exists a positive integer $h$ such that
    \[E_{n+(p^h-1)p^k}^{(d)}\equiv E_n^{(d)}\mod{p^{k+1}}\quad\text{for all }k\in\mathbb{Z}_{\ge0}\text{ and }n\ge p^k.\]
    Moreover, such $h$ can be taken as a divisor of $\varphi(d)$, where $\varphi$ is Euler's totient function.

    (2) Let $p$ be a prime. Then, 
    \begin{equation}
    \label{d=p_cong}
        (-1)^{n+p^k}E_{n+p^k}^{(p)}\equiv(-1)^nE_n^{(p)}\mod{p^{k}}
    \end{equation}
    holds for all $k,n\ge0$ and the map $f:\mathbb{Z}_{\ge0}\rightarrow\mathbb{Z}_p;n\mapsto (-1)^nE_n^{(p)}$ is continuously extended to a function on $\mathbb{Z}_p$, which is locally analytic, at least as strong as of order $1$.

    (3) Let $p$ be a prime, and $m$ be a positive integer that is coprime to $p$. Set $d=mp$. Then, there exists a positive integer $h$ dividing $\varphi(m)$ such that
    \[E_{n+(p^h-1)p^k}^{(d)}\equiv E_n^{(d)}\mod{p^{k}}\quad\text{for all }k\in\mathbb{Z}_{\ge0}\text{ and }n\ge n_k,\]
    where $n_k$ is the minimum positive integer such that $v_p(n_k!)\ge k$
    
    Moreover, for each $i=0,1,\cdots,p^h-2$, the map 
    \[\mathbb{Z}_{\ge0}\rightarrow\mathcal{O}_{\mathbb{C}_p};n\mapsto \lim_{k\rightarrow\infty}E_{i+(p^h-1)(n+p^k)}^{(d)}\]
    is continuously extended to a function on $\mathbb{Z}_p$ which is locally analytic, at least as strong as of order $1$.
\end{corollary}
\begin{proof}
    Let 
    \[F_d(t)\coloneqq\sum_{n=0}^{\infty}\frac{t^{dn}}{(dn)!}\]
    for each $d\ge2$.
    
    (1) Since $p^{\varphi(d)}\equiv1\mod{d}$, there are some $h$ such that $d\mid p^h-1$ and $h\mid \varphi(d)$. For such $h$, $F_d(t)\in\mathcal{R}_{p,h,1}$ by the sequence generated by $F_d$ has a period of length $d$. (For the notation of $\mathcal{R}_{p,h,1}$, see definition~\eqref{R_r_def(R_r_alpha_decomposition_equ)}.) Since $F_d(0)=1$, we obtain the conclusion by Corollary~\ref{Reciplocal_cor}.
    
    (2) Since the map
    \[f:\mathbb{Z}_{\ge0}\rightarrow\mathcal{O}_{\mathbb{C}_p};n\mapsto\begin{cases}
        1 & \text{if }n\equiv0\mod{p}\\
        0 & \text{if }n\not\equiv0\mod{p}\\
    \end{cases}\]
    can be continuously extended to a $1$-Lipschitz continuous function on $\mathbb{Z}_p$, by Proposition~\ref{Basic_prop}, there is a $G(t)\in\mathbb{C}_p\langle t|\ge\Theta_0\rangle$ such that $F_p(t)=e^tG(t)$.

    We can calculate the coefficients of $G(t)$ explicitly as
    \begin{align*}
        G(t)&=e^{-t}F_p(t)\\
        &=\left(\sum_{n=0}^{\infty}(-1)^n\frac{t^n}{n!}\right)\left(\sum_{n=0}^{\infty}\left(\frac{1}{p}\sum_{\nu=0}^{p-1}\zeta_p^{n\nu}\right)\frac{t^n}{n!}\right)\\
        &=\sum_{n=0}^{\infty}\left(\frac{1}{p}\sum_{i=0}^{n}\binom{n}{i}(-1)^{n-1}\sum_{\nu=0}^{p-1}\zeta_p^{i\nu}\right)\frac{t^n}{n!}\\
        &=\sum_{n=0}^{\infty}\frac{1}{p}\sum_{\nu=0}^{p-1}(\zeta_p^{\nu}-1)^n\frac{t^n}{n!}
    \end{align*}
    Where $\zeta_p$ is a primitive $p$-th root of unity.
    
    Since $v_p(\zeta_p-1)\ge1/(p-1)$, we have
    \[v_p\left(\sum_{\nu=0}^{p-1}(\zeta_p-1)^n\right)\ge\frac{n}{p-1}>\frac{n-1}{p-1}\ge v_p(n!).\]
    Because $\sum_{\nu=0}^{p-1}(\zeta_p-1)^n\in\mathbb{Q}$, we have $v_p(\sum_{\nu=0}^{p-1}(\zeta_p-1)^n)\in\mathbb{Z}$, and thus,
    \[v_p\left(\sum_{\nu=0}^{p-1}(\zeta_p-1)^n\right)-1\ge v_p(n!).\]

    Hence, we have $G(t)\in\mathbb{C}_p\langle t|\ge v_p(n!)\rangle$.

    Since $G(0)=1,$ we have $1/G(t)\in\mathbb{C}_p\langle t|\ge v_p(n!)\rangle$ by Proposition~\ref{subring_prop}.

    Note that $\mathbb{C}_p\langle t|\ge v_p(n!)\rangle\subset\mathbb{C}_p\langle t|\ge\Theta_0\rangle\cap\mathbb{C}_p\langle t|v_p(\lfloor n/p^s\rfloor!)\rangle$. Now, since
    \[\sum_{n=0}^{\infty}E_n^{(p)}\frac{t^n}{n!}=\frac{1}{F_p(t)}=e^{-t}\cdot\frac{1}{G(t)}\]
    we obtain the result for the case of $p\ne2$ by Proposition~\ref{Basic_prop}. 
    
    For the case of $p=2$,
    \[\sum_{n=0}^{\infty}E_n^{(2)}\frac{t^n}{n!}=e^{-t}\cdot\frac{1}{G(t)}=e^t\cdot\frac{e^{-2t}}{G(t)}\]
    and $e^{-2t}\in\mathbb{C}_2\langle t|\ge v_2(n!)\rangle$ give
    \[E_{n+2^k}^{(2)}\equiv E_n^{(2)}\mod{2^{k}}\]
    holds for all $k,n\ge0$. This is equivalent to \eqref{d=p_cong}.

    (3) Let $h$ be a positive integer such that $m\mid p^h-1$.
    
    Let $\zeta_m$ be a primitive $m$-th root of unity. Since 
    \[\sum_{i=0}^{m}\zeta_m^{in'}=\begin{cases}
        1 & (\text{if }m\mid n')\\
        0 & (\text{if }m\nmid n'),
    \end{cases}\]
    we have 
    \[F_d(t)=\sum_{n'=0}^{\infty}\left(\frac{1}{m}\sum_{i=1}^{m}\zeta_m^{in'}\right)\frac{t^{pn'}}{(pn')!}.\]
    Now $p$ is coprime to $m$ so 
    \[\sum_{i=0}^{m}\zeta_m^{in'}=\sum_{i=0}^{m}\zeta_m^{ipn'}.\]
    Therefore,
    \begin{align*}
        F_d(t)
        &=\frac{1}{m}\sum_{i=1}^{m}\sum_{n'=0}^{\infty}\frac{(\zeta_m^it)^{pn'}}{(pn')!}\\
        &=\frac{1}{m}\sum_{i=1}^{m}F_p(\zeta_m^it)\\
        &=\frac{1}{m}\sum_{i=1}^{m}e^{-\zeta_m^it}G(\zeta_m^it)
    \end{align*}
    where $G(t)\in\mathbb{C}_p\langle t|\ge v_p(n!)\rangle$ is the same one appearing in the proof of (2).

    Since $m\mid \ell$, $\zeta_m\in\mathcal{O}_h$, where $\mathcal{O}_h$ is the valuation ring of the unramified extension of $\mathbb{Q}_p$ of degree $h$. 

    Because $F_d(0)=1$, we can write
    \[\frac{1}{F_d(t)}=\sum_{\alpha^{p^h}=\alpha}e^{\alpha t}G_\alpha(t)\]
    with some $\{G_\alpha(t)\}\subset\mathbb{C}_p\langle t|\ge v_p(n!)\rangle$ by Proposition~\ref{norm_and_reciprocal_prop}. Thus, we obtain the conclusion by Corollary~\ref{change_of_basis_la_cor} and Theorem~\ref{application_thm}.
\end{proof}

\subsection{Fibonacci numbers and Lucas numbers}
Fibonacci numbers $\{F_n\}_{n=0}^{\infty}$ are defined by the recurrence formula
\[F_{0}=0,\quad F_1=1,\quad F_{n+2}=F_n+F_{n+1}.\]
There is a well-known explicit formula for Fibonacci numbers:
\begin{equation}
\label{Fibnacci_equ}
    F_n=\frac{1}{\sqrt{5}}\left(\left(\frac{1+\sqrt{
5}}{2}\right)^n-\left(\frac{1-\sqrt{5}}{2}\right)^n\right)=\frac{1}{-\sqrt{5}}\left(\left(\frac{1-\sqrt{
5}}{2}\right)^n-\left(\frac{1-(-\sqrt{5})}{2}\right)^n\right).
\end{equation}

The sequence of Lucas numbers $\{L_n\}_{n=0}^\infty$ has a similar definition:
\[L_0=2,\quad L_1=1,\quad L_{n+2}=L_n+L_{n+1}.\]
The explicit formula for the Lucas numbers is 
\begin{equation}
\label{Lucas_equ}
    L_n=\left(\frac{1+\sqrt{
5}}{2}\right)^n+\left(\frac{1-\sqrt{5}}{2}\right)^n.
\end{equation}

For these numbers, we have the following congruences. For the case of $p\ne2,5$, the proof is essentially given by D. D. Wall in \cite{Wall}.
\begin{theorem}[Combined result of Lemma~3 of \cite{Jacobson} ($p=2$)\footnote{Precisely, Lemma~3 in \cite{Jacobson} is only for the case $k\ge3$, but the validity for $k=0, 1,$ and $2$ follows immediately from a finite calculation.} and Theorem~4.1 of \cite{Lengyel} ($p\ne2,5$)]
\label{Fibonacci_thm}
    Let $p\ne5$ be a prime.

    (1) If $p\equiv 1,4\mod{5}$, then, we have 
    \begin{align*}
        &F_{n+(p-1)p^k}\equiv F_n\mod{p^{k+1}},\\
        &L_{n+(p-1)p^k}\equiv L_n\mod{p^{k+1}}
    \end{align*}
    for all $k\ge0$ and $n\ge0$.

    (2) If $p\equiv 2,3\mod{5}$ and $p\ne2$, then, we have 
    \begin{align*}
        &F_{n+2(p+1)p^k}\equiv F_n\mod{p^{k+1}},\\
        &L_{n+2(p+1)p^k}\equiv L_n\mod{p^{k+1}}
    \end{align*}
    for all $k\ge0$ and $n\ge0$. 

    (3) We have
    \begin{align*}
        &F_{n+(2^2-1)2^k}\equiv F_n\mod{2^{k+1}},\\
        &L_{n+(2^2-1)2^k}\equiv L_n\mod{2^{k+1}}
    \end{align*}
    for all $k\ge0$ and $n\ge0$. 
\end{theorem}
\begin{proof}
    By the formula \eqref{Fibnacci_equ} and \eqref{Lucas_equ}, we have
    \begin{align*}
        &\sum_{n=0}^{\infty}F_n\frac{t^n}{n!}=\frac{1}{\alpha}\left(\exp{(\frac{1+\alpha}{2}t)-\exp{(\frac{1-\alpha}{2}t)}}\right)\quad\text{and}\\
        &\sum_{n=0}^{\infty}L_n\frac{t^n}{n!}=\exp{(\frac{1+\alpha}{2}t)+\exp{(\frac{1-\alpha}{2}t)}}.
    \end{align*}
    where $\alpha$ is a root of $x^2-5$. Now $p\ne5$, so $\pm1,\pm1/\sqrt{5}\in\mathbb{C}_p\langle t|\ge\Theta_1\rangle$. Therefore, by Theorem~\ref{application_thm}, the problem is reduced to determine the order $\ell$ satisfying $((1\pm\alpha)/2)^\ell\equiv1\mod{p}$.
    
    First, consider the case of $p\ne2$. By the law of quadratic reciprocity, Hensel's lifting lemma, and an easy observation that $\alpha\not\equiv0,\pm1\mod{p}$, we have
    \[\frac{1\pm\alpha}{2}\in\begin{cases}
        \mathbb{Z}_p^{\times} & (\text{if }p\equiv1,4\mod{5}),\\
        \mathcal{O}_2^{\times}\setminus\mathbb{Z}_p^{\times} & (\text{if }p\equiv2,3\mod{5}\text{ and }p\ne2),
    \end{cases}\]
    where $\mathcal{O}_2$ denotes the ring of integers of the unramified quadratic extension of $\mathbb{Q}_p$, as defined in Subsection~1.1. 

    Since all of $x\in\mathbb{Z}_p^\times$ satisfies $x^{p-1}\equiv 1\mod{p}$, we obtain (1) by Theorem~\ref{application_thm}.

    In the case of $p\equiv2,3\mod{5}$ and $p\ne2$, $\beta_1\coloneqq(1+\alpha)/2$ and $\beta_2\coloneqq(1-\alpha)/2$ satisfies $\beta_i^{2(p+1)}\equiv1\mod{p}$. Indeed, they are the roots of $x^2-x-1$, and in the field $\mathbb{F}_{p^2}\cong\mathbb{F}_p[x]/(x^2-x-1)$, we have
    \[x^{2(p+1)}=(x+1)^{p+1}=(x^p+1)(x+1)=(2-x)(x+1)=1\]
    where we used the fact that $x^p=1-x$ because the Frobenius map swaps the two roots of $x^2-x-1$ and the sum of the roots is $1$.
    
    Next consider the case $p=2$, since $5\not\equiv1\mod{8}$, $\alpha\notin\mathbb{Z}_2$. Since $\mathbb{Q}_2(\alpha)/\mathbb{Q}_2$ is an unramified extension of degree $2$ and $v_2((1\pm\alpha)/2)=0$, we obtain $(1\pm\alpha)/2\in\mathcal{O}_2^{\times}$. All of the element of $\mathcal{O}_2^\times$ are mapped in $\mathbb{F}_4$ by the reduction map $\mathcal{O}_{\mathbb{C}_2}\rightarrow\overline{\mathbb{F}_2}$. Thus, $((1\pm\alpha)/2)^3\equiv1\mod{p}$. So we obtain (3) by Theorem~\ref{application_thm}.
\end{proof}

We can obtain the periodicity of Fibonacci numbers and Lucas numbers modulo $5^k$ by using $\mathbb{C}_p\langle t|\ge\Theta\rangle_r$ with $r\notin\mathbb{Z}$.
\begin{theorem}
    We have 
    \begin{align*}
        &F_{n+4\cdot5^k}\equiv F_n\mod{5^k},\\
        &L_{n+4\cdot5^k}\equiv L_n\mod{5^{k+1}}
    \end{align*}
    for all $k\ge0$ and $n\ge0$.
\end{theorem}
\begin{proof}
    We have
    \begin{align*}
        \sum_{n=0}^{\infty}F_n\frac{t^n}{n!}&=\frac{1}{\alpha}\left(\exp(\frac{1}{2}t)\exp(\frac{\alpha}{2}t)-\exp(\frac{1}{2}t)\exp(-\frac{\alpha}{2}t)\right),\\
        \sum_{n=0}^{\infty}L_n\frac{t^n}{n!}&=\exp(\frac{1}{2}t)\exp(\frac{\alpha}{2}t)+\exp(\frac{1}{2}t)\exp(-\frac{\alpha}{2}t)
    \end{align*}
    where $\alpha$ is a root of $x^2-5$.
    Now, $\pm\frac{1}{2}\in\mathbb{Z}_5^{\times}$ and $(\pm\frac{1}{2})^4\equiv1\mod{5}$. Moreover, since $v_5(\pm\frac{\alpha}{2})=\frac{1}{2}$, we have
    \[\exp(\pm\frac{\alpha}{2}t)\in\mathbb{C}_5\langle t|\ge n/2\rangle\subset\mathbb{C}_5\langle t|\ge\Theta_{\frac{1}{2}}\rangle.\]
    Therefore, 
    \[\sum_{n=0}^{\infty}\alpha F_n\frac{t^n}{n!},~\sum_{n=0}^{\infty}L_n\frac{t^n}{n!}\in\bigoplus_{\alpha^4=1}e^{\alpha t}\mathbb{C}_5\langle t|\ge\Theta_{\frac{1}{2}}\rangle\]
    and it follows that
    \[\alpha F_{n+4\cdot5^k}\equiv\alpha F_n,\quad L_{n+4\cdot 5^k}\equiv L_n\mod{p^{k+\frac{1}{2}}}\]
    for all $k,n\in\mathbb{Z}_{\ge0}$.
    Because $v_5(\alpha)=\frac{1}{2}$ and $v_p(L_{n+4\cdot 5^k}-L_n)\in\mathbb{Z}$, we obtain 
    \begin{align*}
        &F_{n+4\cdot5^k}\equiv F_n\mod{5^k},\\
        &L_{n+4\cdot5^k}\equiv L_n\mod{5^{k+1}}
    \end{align*}
    for all $k,n\in\mathbb{Z}_{\ge0}$.
\end{proof}

We can generalize this argument to some extent. Indeed,
if a sequence$\{c_n\}_{n=0}^{\infty}\subset\mathbb{C}_p$ satisfies a recurrence relation
\[c_n=a_1c_{n-1}+\cdots+a_mc_{n-m}\]
with some $a_1,\cdots,a_{m}\in\mathbb{Z}_p$ ($a_m\ne0$), the general term of $\{c_n\}$ can be written as
\[c_n=\sum_{i=1}^s\sum_{j=0}^{m_i-1}A_{i,j}n^j\lambda_i^n\]
where $\lambda_1,\cdots,\lambda_s$ are distinct roots of 
\[x^m-a_1x^{m-1}-\cdots-a_{m-1}x-a_m=0\]
with multiplicities $m_1,\cdots,m_s$, respectively, and $A_{i,j}$ are constants determined by $c_0,\cdots,c_{m-1}$.\footnote{For example, see Theorem~4 in 8.2.2 of \cite{Rosen}.} If we set 
\[F_{0,i}(t)=\exp(\lambda_i t),\quad F_{k,i}(t)=t\frac{d}{dt}F_{k-1,i}(t)\]
for each $i=1,\cdots,s$ and $k\ge0$. Then, 
\[\sum_{n=0}^{\infty}c_n\frac{t^n}{n!}=\sum_{i=1}^s\sum_{j=0}^{m_i-1}A_{i,j}F_{j,i}(t).\]
Because $F_{k,i}(t)$ can be written as
\[F_{k,i}(t)=\exp(\lambda_it)P(t)\]
with some polynomial $P(t)$, we have 
\[\sum_{n=0}^{\infty}c_n\frac{t^n}{n!}=\sum_{i=1}^s\exp(\lambda_it)F_i(t)\]
with some polynomial $F_i(t)$.

Therefore, depending on the values of $\lambda_i$ and the coefficients of $F_i$, the generating function of $\{c_n\}$ may belong to $\mathcal{M}_r(\mathbb{C}_p)$ or $\mathcal{M}_{r,s}^{la}(\mathbb{C}_p)$. Since the valuations of $\lambda_i$ can be estimated, and $A_{i,j}$ and the coefficients of $F_i$ are computable, a general formulation might be possible in principle. However, the conditions and conclusions would require a rather cumbersome formulation. We omit such a formulation here.

\subsection{Wall-Sun-Sun primes}
We say a prime number $p$ is a Wall-Sun-Sun prime, also called a Fibonacci-Wieferich prime, if $p^2$ divides $F_{p-\left(\frac{5}{p}\right)}$ where $\{F_n\}$ is the Fibonacci sequence and $\left(\frac{5}{p}\right)$ is the Legendre symbol. This class of primes has been of interest since Z.-H. Sun and Z.-W. Sun proved the following theorem in 1992 \cite{Sun}:
\begin{theorem}[Theorem~4 in \cite{Sun}]
\label{Sun_thm}
    If the first case of Fermat's last theorem fails for odd prime $p$, that is, if there exists an integer solution for 
    \[x^p+y^p=z^p\quad\text{and}\quad p\nmid xyz,\]
    then $p$ is a Wall-Sun-Sun prime.
\end{theorem}
Due to the celebrated proof of Fermat's last theorem by A. Wiles in 1995 \cite{Wiles}, One of the reasons for the interest in Wall-Sun-Sun primes may have been lost, but the existence of Wall-Sun-Sun primes, questioned by Wall in 1960 \cite{Wall}, and even the infinitude of non Wall-Sun-Sun primes remain open problems. 

Using the results and arguments in Section~3 and Section~4, we now have some equivalent conditions for $p\ne2,5$ to be a Wall-Sun-Sun prime. The reformulations (b) and (c) are essentially known previously, but (d) seems novel. It is unclear, however, whether (d) will be useful for solving the aforementioned problems.

For each positive integer $m$, let $k(m)$ denote the length of the period of $\{F_n\pmod{m}\}_{n=0}^{\infty}$, that is, the least positive integer such that 
\[F_{n+k(m)}\equiv F_n\mod{m}\quad\text{for all }n\ge0.\]
\begin{theorem}
    Let $p\ne2,5$ be a prime and $r$ be a positive integer. The following conditions are equivalent.
    \begin{enumerate}[label=(\alph*)]
        \item $p^{1+r}\mid F_{p-\left(\frac{5}{p}\right)}$
        \item $\alpha^{q-1}\equiv1\mod{p^{1+r}}$ and $\beta^{q-1}\equiv1\mod{p^{1+r}}$ where $\alpha,\beta$ are roots of $x^2-x-1=0$ and 
        \[q=
        \begin{cases}
            p & (\text{if }p\equiv1,4\mod{5})\\
            p^2 & (\text{if }p\equiv2,3\mod{5})
        \end{cases}\]
        \item For all $n,k\ge0$,
            \[F_{n+k(p)p^k}\equiv F_n\mod{p^{k+1+r}}.\]
        \item For each $i=0,1,\cdots,k(p)-1$, the function
        \[f_i:\mathbb{Z}\rightarrow \mathbb{C}_p; n\mapsto F_{i+k(p)n}\]
        can be continuously extended to an analytic function on 
        \[\{x\in\mathbb{C}_p\mid|x|_p<p^{1+r-\frac{1}{p-1}}\}.\]
    \end{enumerate}
\end{theorem}
The reformulations (b) and (c) are essentially implied by known results, whereas (d) appears to be new. However, it is unclear whether (d) contributes to investigating the existence of Wall-Sun-Sun primes.
\begin{proof}
    Before proving each implication, We state some general settings and properties that independently hold without each condition. 

    \begin{itemize}
        \item Let $\mathcal{O}$ be the integer ring of $K=\mathbb{Q}(\zeta_{q-1})$ where $\zeta_{q-1}$ is a primitive $(q-1)$-th root of unity. We have $\alpha,\beta\in\mathcal{O}$. Let 
        \[\zeta=\lim_{n\rightarrow\infty}\alpha^{q^n},\quad\eta=\lim_{n\rightarrow\infty}\beta^{q^n}.\]
        Then, $\zeta,\eta\in\mathcal{O}$ are $(q-1)$-th roots of unity, and we have
        \[v_p(\alpha-\zeta)\ge1,\quad v_p(\eta-\beta)\ge1.\]
    
        \item We have
        \[\alpha^{q-1}\equiv1\mod{p^{1+r}}\Leftrightarrow v_p(\alpha-\zeta)\ge 1+r.\]
        Indeed, if we write
        \[\alpha^q=\alpha+up^{r'}\]
        with $u\in\mathcal{O}^\times$ and $r'\ge1$, we obtain
        \[\alpha^{q^{n+1}}=\alpha^{q^{n}}\mod{p^{r'+1}}\]
        by induction. So we have
        \[\alpha^{q^n}\equiv\alpha+up^{r'}\mod{p^{r'+1}}\]
        for all $n\ge0$, and it follows that $v_p(\zeta-\alpha)=r'$.

        \item Since $\alpha\beta=1$, it follows that
        \begin{align*}
            \text{(b)}&\Leftrightarrow\alpha^{q-1}\equiv1\mod{p^{1+r}}\Leftrightarrow\beta^{q-1}\equiv1\mod{p^{1+r}}\\
            &\Leftrightarrow v_p(\alpha-\zeta)\ge1+r\Leftrightarrow v_p(\beta-\eta)\ge1+r
        \end{align*}
        and we obtain $r'\coloneqq v_p(\alpha-\zeta)=v_p(\beta-\eta)$. 
        
        \item We have 
        \begin{equation}
        \label{Fibonacci_zeta_eta_equ}
        \sum_{n=0}^{\infty}F_n\frac{t^n}{n!}=\frac{1}{\alpha-\beta}\left(e^{\zeta t}e^{(\alpha-\zeta)t}-e^{\eta t}e^{(\beta-\eta)t}\right).\end{equation}

        \item By induction, for each $n\ge1$, we have
        \begin{equation}
        \label{alpha^n_and_Fibonacci_equ}
            \alpha^n=F_n\alpha+F_{n-1},\quad\beta^n=F_n\beta+F_{n-1}.
        \end{equation}

        \item $k(p)$ is precisely the order of $\alpha$ in the residue Field $\mathbb{F}_q$, that is, the least positive integer $\ell$ such that $\alpha^\ell\equiv1\mod{p}$. Indeed, by \eqref{alpha^n_and_Fibonacci_equ},
        \[\alpha^{k(p)+1}=F_{k(p)+1}\alpha+F_{k(p)}\equiv F_1\alpha+F_0=\alpha\mod{p}\]
        and so, $\ell\mid k(p)$. Conversely, since $\beta^\ell\equiv1\mod{p}$ too by $\alpha\beta=1$, we have
        \begin{align*}
            0=1-1&\equiv\alpha^\ell-\beta^\ell=F_\ell(\alpha-\beta)\mod{p},\\
            1=\alpha+\beta&\equiv\alpha^{\ell+1}+\beta^{\ell+1}=F_{\ell+1}+2F_{\ell}\mod{p}.
        \end{align*}
        This leads to $F_\ell\equiv0,F_{\ell+1}\equiv1\mod{p}$ and thus, $k(p)\mid\ell$.
        
        \item Thus, $\zeta^{k(p)}=\eta^{k(p)}=1$. In particular, $p\nmid k(p)$ and $k(p)\mid q-1$.
    \end{itemize}
    
    \textbf{(a)$\Leftrightarrow$(b): }As shown above, the condition (b) is equivalent to $r'=v_p(\alpha-\zeta)\ge1+r$. The equation $v_p(F_{p-\left(\frac{5}{p}\right)})=v_p(\alpha-\zeta)$ is given in Theorem~10 of \cite{Bragman_and_Rowland}. Thus, we have (a)$\Leftrightarrow$(b).

    \textbf{(b)$\Leftrightarrow$(c): }First, note that 
    \[(1+r)n\ge\Theta_{1+r}(n)=\lfloor\log_pn\rfloor+1+r\]
    for all $r\ge1$ and $n\ge0$. Thus, by \ref{Fibonacci_zeta_eta_equ} and $r'=v_p(\alpha-\zeta)=v_p(\beta-\eta)$, 
    \[\sum_{n=0}^{\infty}F_n\frac{t^n}{n!}\in e^{\zeta t}\mathbb{C}_p\langle t|\ge\Theta_{r'}\rangle+e^{\eta t}\mathbb{C}_p\langle t|\ge\Theta_{r'}\rangle.\]
    Recall (b)$\Leftrightarrow r'\ge 1+r$. We have $r'\ge 1+r\Rightarrow$(c) by the decomposition \eqref{(1.2)_gen_funcs_decomposition}. Conversely, if (c) holds, again by \eqref{(1.2)_gen_funcs_decomposition}, there is a decomposition
    \begin{equation}
    \label{Fibonacci_decomp_equ}
        \sum_{n=0}^{\infty}F_n\frac{t^n}{n!}=\sum_{\gamma^{q-1}=1}e^{\gamma t}F_\gamma(t),\quad F_\gamma(t)\in\mathbb{C}_p\langle t|\ge\Theta_{1+r}\rangle.
    \end{equation}
    Now we know that at least 
    \[\frac{e^{(\alpha-\zeta)t}}{\alpha-\beta},\frac{e^{(\beta-\eta)t}}{\alpha-\beta}\in\mathbb{C}_p\langle t|\ge\Theta_1\rangle.\]
    So by considering the uniqueness of decomposition in $\mathcal{M}_{p,q-1,1}(\mathcal{O}_{\mathbb{C}_p})$, we conclude that the decomposition \eqref{Fibonacci_decomp_equ} is precisely \eqref{Fibonacci_zeta_eta_equ}. This forces
    \[\frac{e^{(\alpha-\zeta)t}}{\alpha-\beta}\in\mathbb{C}_p\langle t|\ge\Theta_{1+r}\rangle,\]
    and so
    \[v_p((\alpha-\zeta)^n)=r'n\ge\lfloor\log_pn\rfloor+1+r\]
    for all $n\ge0$. In particular, setting $n=1$, we obtain $r'\ge 1+r$. 

    \textbf{(b)$\Leftrightarrow$(d): }Let 
    \[f_\zeta(n)=\frac{1}{\alpha-\beta}\left(\frac{\alpha}{\zeta}\right)^n,\quad f_\eta(n)=\frac{1}{\alpha-\beta}\left(\frac{\beta}{\eta}\right)^n.\]
    (These functions correspond to $\{f_\alpha\}$ in Proposition~\ref{smoothness_of_interpolation_prop}.)
    
    We have
    \[a_n(f_\zeta)=\sum_{k=0}^{n}\binom{n}{k}(-1)^{n-k}f_\zeta(k)=\frac{(\alpha-\zeta)^n}{\zeta^n(\alpha-\beta)}.\]
    Thus, $v_p(a_n(f_\zeta))-v_p(n!)=r'n-v_p(n!)\rightarrow\infty$ as $n\rightarrow\infty$, and so $f_\zeta$ is continuously extended to an analytic function on $\mathbb{Z}_p$ with the Mahler expansion
    \begin{equation}
    \label{Mahler_exp_of_f_zeta}
        f_\zeta(x)=\sum_{n=0}^{\infty}\frac{(\alpha-\zeta)^n}{\zeta^n(\alpha-\beta)}\binom{x}{n}
    \end{equation}
    by Theorem~\ref{smooth_coeffi_thm}. Similarly, $f_\eta$ is continuously extended to an analytic function 
    \begin{equation}
    \label{Mahler_exp_of_f_eta}
        f_\eta(x)=\sum_{n=0}^{\infty}\frac{(\beta-\eta)^n}{\eta^n(\alpha-\beta)}\binom{x}{n}
    \end{equation}
    on $\mathbb{Z}_p$. 

    For each $i=0,1,\cdots,k(p)-1$,
    \[f_i(n)=\zeta^if_\zeta(i+k(p)n)+\eta^if_\eta(i+k(p)n).\]
    Thus, $f_i$ can be continuously extended as
    \begin{equation}
    \label{f_i=equ}
        f_i(x)=\zeta^if_\zeta(i+k(p)x)+\eta^if_\eta(i+k(p)x)
    \end{equation}
    for $x\in\mathbb{Z}_p$, and by the same argument in the proof of Proposition~\ref{smoothness_of_interpolation_prop}, we have
    \begin{equation}
    \label{f_zeta=equ}
        f_\zeta(x)=\sum_{i}b_{\zeta,i}f_i\left(\frac{x-i}{k(p)}\right),\quad f_\eta(x)=\sum_{i}b_{\eta,i}f_i\left(\frac{x-i}{k(p)}\right)
    \end{equation}
    with some $\{b_{\zeta,i}\},\{b_{\eta,i}\}\in\mathcal{O}$. By \eqref{Mahler_exp_of_f_zeta} and \eqref{Mahler_exp_of_f_eta}, we have
    \[f_\zeta(x)=\sum_{n=0}^{\infty}a^{(\zeta)}_nx^n,\quad f_\eta(x)=\sum_{n=0}^{\infty}a^{(\eta)}_nx^n\]
    for some $\{a_n^{(\zeta)}\},\{a_n^{(\eta)}\}\subset\mathcal{O}_{\mathbb{C}_p}$ with $v_p(a_n^{(\zeta)})=v_p(a_n^{\eta})=r'n-v_p(n!)$. Thus, the radius of convergence of $f_\zeta$ and $f_\eta$ is $p^{r'-\frac{1}{p-1}}$.

    By the relations \eqref{f_zeta=equ} and \eqref{f_i=equ}, we have
    \[\{f_i\}\text{ has radius of convergence at least }p^{R}\Leftrightarrow f_\zeta\text{ and }f_\eta\text{ has radius of convergence at least }p^{R}.\]

    Thus, (b)$\Leftrightarrow r'\ge1+r\Leftrightarrow$(d).
\end{proof}

\appendix
\section{Worst growth of $-v_p(b_n)$}
For each sequence $\{c_n\}_{n=0}^{\infty}\subset\mathcal{O}_{\mathbb{C}_p}$, set
\[b_n=\frac{1}{n!}\sum_{k=0}^{n}(-1)^{n-k}\genfrac{[}{]}{0pt}{}{n}{k}c_k\]
where $\genfrac{[}{]}{0pt}{}{n}{k}$ are the (unsigned) first kind of Stirling numbers. Then, we have the following fact.
\begin{proposition}
    Let $r\in\mathbb{R}$. Then, there exists a sequence $\{c_n\}_{n=0}^{\infty}\subset\mathcal{O}_{\mathbb{C}_p}$ such that
    \begin{equation}
    \label{appendix_equ_1}
        c_{n+p^k}\equiv c_n\mod{p^{k+r}}\quad\text{for all }n,k\ge0
    \end{equation}
    and
    \begin{equation}
    \label{appendix_equ_2}
        v_p(b_n)\le -v_p(n!)+\lfloor\log_pn\rfloor+2|r|\quad\text{for all }n\ge1.
    \end{equation}
\end{proposition}
\begin{proof}
    Suppose we find a required $\{c_n\}=\{c_n^{(0)}\}$ for $r=0$. Then, for each $r\in\mathbb{R}$, $\{c_n\}=\{\pi c_n^{(0)}\}$ satisfies \eqref{appendix_equ_1} and \eqref{appendix_equ_2} for $r$, where $\pi\in\mathcal{O}_{\mathbb{C}_p}$ is an arbitrary element satisfying $r\le v_p(\pi)\le2r$ if $r>0$ and $\pi=1$ if $r\le0$. Therefore, it suffices to show in the case of $r=1$. 
    
    Note that if $F(t)$ is the exponential generating function of $\{c_n\}$, then $G(t)=F(\log(1+t))$ is the exponential generating function of $\{n!b_n\}$. By Proposition~\ref{Basic_prop}, $\{c_n\}$ satisfies \eqref{appendix_equ_1} with $r=0$ if and only if we can write
    \[F(t)=e^tH(t)\quad\text{with some}\quad H(t)=\sum_{n-0}^{\infty}c'_n\frac{t^n}{n!}\in\mathbb{C}_p\langle t|\ge\Theta_0\rangle.\]
    If such $H(t)$ exists, we have
    \begin{align*}
        \sum_{n=0}^{\infty}b_nt^n&=\sum_{n=0}^{\infty}(n!b_n)\frac{t^n}{n!}=G(t)\\
        &=(1+t)H(\log(1+t))\\
        &=(1+t)\sum_{n=0}^{\infty}\left(\sum_{k=0}^{n}(-1)^{n-k}\genfrac{[}{]}{0pt}{}{n}{k}c'_k\right)\frac{t^n}{n!}\\
        &=\sum_{n=0}^{\infty}\left(c'_n+\sum_{k=0}^{n-1}(-1)^{n-k}\left\{\genfrac{[}{]}{0pt}{}{n}{k}-\genfrac{[}{]}{0pt}{}{n-1}{k}\right\}c'_k\right)\frac{t^n}{n!}
    \end{align*}

    Define the sequence $\{c'_n\}\subset\mathcal{O}_{\mathbb{C}_p}$ inductively by
    setting $c'_0=1$ and 
    \[c'_n=
    \begin{cases}
        p^{\lfloor\log_pn\rfloor} & (\text{if }v_p(A_n)=\lfloor\log_pn\rfloor-1),\\
        p^{\lfloor\log_pn\rfloor}-A_n & (\text{if }v_p(A_n)\ge\lfloor\log_pn\rfloor),
    \end{cases}
    \text{ where }A_n=\sum_{k=0}^{n-1}(-1)^{n-k}\left\{\genfrac{[}{]}{0pt}{}{n}{k}-\genfrac{[}{]}{0pt}{}{n-1}{k}\right\}c'_k\]
    for $n\ge1$. We claim that $v_p(A_n)\ge\lfloor\log_pn\rfloor-1$ and $v_p(c'_n)\ge\lfloor\log_pn\rfloor$ for all $n\ge1$, that is, above $\{c'_n\}$ is well defined. First, for $n\ge1$, since $\genfrac{[}{]}{0pt}{}{n}{k}=(-1)^{n-k}e_{n-k}(-1,\cdots,-(n-1))$ where $e_{n-k}$ are the elementary symmetric polynomials in $n-1$ variables, we have
    \[v_p\left(\genfrac{[}{]}{0pt}{}{n}{k}\right)\ge\max\left\{0,(n-k)-\left(n-1-\left\lfloor\frac{n-1}{p}\right\rfloor\right)\right\}=\max\left\{0,\left\lfloor\frac{n-1}{p}\right\rfloor-k+1\right\}.\]
    Suppose $v_p(c'_k)\ge\lfloor\log_pk\rfloor$ for all $0\le k\le n-1$. It is not defficult to see
    \[\max\left\{0,\left\lfloor\frac{n-2}{p}\right\rfloor-k+1\right\}+\lfloor\log_pk\rfloor\ge\max\left\{0,\left\lfloor\frac{n}{p}\right\rfloor-k\right\}+\lfloor\log_pk\rfloor\ge\left\lfloor\log_p\left\lfloor\frac{n}{p}\right\rfloor\right\rfloor\]
    for all $0\le k\le n-1$. Thus, we have
    $v_p\left(\genfrac{[}{]}{0pt}{}{n}{k}c'_k\right)\ge\lfloor\log_pn\rfloor-1$, and $v_p\left(\genfrac{[}{]}{0pt}{}{n-1}{k}c'_k\right)\ge\lfloor\log_pn\rfloor-1$.
    Therefore, $v_p(A_n)\ge\lfloor\log_pn\rfloor-1$, and $v_p(c'_n)\ge\lfloor\log_pn\rfloor$ by the above definition. 

    Hence, this choice of $\{c'_n\}$ leads to $H(t)\in\mathbb{C}_p\langle t|\ge\Theta_0\rangle$ and also implies that $\{c_n\}$ satisfies \eqref{appendix_equ_1} with $r=0$. Moreover, because $b_n=(c'_n+A_n)/n!$, we obtain
    \[v_p(b_n)\in\{\lfloor\log_pn\rfloor-v_p(n!),\lfloor\log_pn\rfloor-1-v_p(n!)\}\quad\text{for all }n\ge1.\]
\end{proof}

\end{document}